\documentclass[a4paper]{amsart}

\usepackage[normalem]{ulem}
\usepackage{amscd}
\usepackage{verbatim}

\def\mylabelonoff{off}
\def\allowdisbrkyesno{yes}
\def\numberingtheoremsectionyesno{yes}
\def\numberingequationsectionyesno{yes}
\def\pagesizeextendednormal{extended}

\def\reportudemathyesno{no}
\def\reportudemathnumber{SM-UDE-805}
\def\reportudemathyear{2016}
\def\reportudematheingang{\mydate}

\def\mytitle{On Closed and Exact $\Gradgrad$- and $\divDiv$-Complexes,\\
Corresponding Compact Embeddings for Tensor Rotations,\\
and a Related Decomposition Result for Biharmonic Problems in 3D}
\def\mytitlerepude{\mytitle}
\def\myshorttitle{$\Gradgrad$- and $\divDiv$-Complexes and Applications}
\def\myauthorone{Dirk Pauly}
\def\myauthortwo{Walter Zulehner}
\def\myauthors{\myauthorone\quad\&\quad\myauthortwo}
\def\myaddressone{Fakult\"at f\"ur Mathematik,
Universit\"at Duisburg-Essen, Campus Essen, Germany}
\def\myaddresstwo{Institut f\"ur Numerische Mathematik, 
Johannes Kepler Universit\"at Linz, Austria}
\def\myemailone{dirk.pauly@uni-due.de}
\def\myemailtwo{zulehner@numa.uni-linz.ac.at}
\def\mykeywords{biharmonic equations, Helmholtz decomposition, Hilbert complexes}
\def\mysubjclass{35G15, 58A14}
\def\mydate{\today}
\def\mythanks{The research of the second author was supported by the Austrian Science Fund (FWF): project S11702-N23}

\usepackage[mathscr]{eucal}
\usepackage[english]{babel}
\usepackage{a4,exscale,ifthen,amsfonts,amssymb,amsmath,amscd,graphicx,color,mathtools}
\usepackage{nicefrac,tikz,fancyhdr,caption,array}
\usepackage[all]{xy}

\ifthenelse{\equal{\mylabelonoff}{on}}
{\newcommand{\mylabel}[1]{\label{#1}\fbox{{\sf #1}}}}
{\newcommand{\mylabel}[1]{\label{#1}}}
\ifthenelse{\equal{\allowdisbrkyesno}{yes}}
{\allowdisplaybreaks}
{}

\ifthenelse{\equal{\pagesizeextendednormal}{extended}}
{\setlength{\textwidth}{16cm}
\setlength{\textheight}{22cm}
\setlength{\oddsidemargin}{-0.2cm}
\setlength{\evensidemargin}{-0.2cm}}
{}
\ifthenelse{\equal{\numberingequationsectionyesno}{yes}}
{\numberwithin{equation}{section}}
{}

\newcommand{\ovl}[1]{\overline{#1}}

\newcommand{\X}{\mathsf{X}}
\newcommand{\Y}{\mathsf{Y}}
\DeclareMathOperator{\diam}{diam}

\newcommand{\set}[2]{\{#1\,:\,#2\}}
\newcommand{\setb}[2]{\big\{#1\,:\,#2\big\}}

\ifthenelse{\equal{\numberingtheoremsectionyesno}{yes}}
{\newtheorem{lem}{Lemma}[section]}
{\newtheorem{lem}{Lemma}}
\newtheorem{defi}[lem]{Definition}
\newtheorem{theo}[lem]{Theorem}
\newtheorem{cor}[lem]{Corollary}
\newtheorem{rem}[lem]{Remark}

\newtheorem{genasslem}[lem]{General Assumption}

\newcommand{\om}{\Omega}

\newcommand{\ga}{\Gamma}

\newcommand{\bbT}{\mathbb{T}}
\newcommand{\bbS}{\mathbb{S}}
\newcommand{\bfM}{\mathbf{M}}
\newcommand{\bfN}{\mathbf{N}}
\newcommand{\bfI}{\mathbf{I}}
\newcommand{\bfF}{\mathbf{F}}
\newcommand{\bfPhi}{\mathbf{\Phi}}
\newcommand{\bfPsi}{\mathbf{\Psi}}

\newcommand{\reals}{\mathbb{R}}
\newcommand{\n}{\mathbb{N}}

\newcommand{\nz}{\n_{0}}

\newcommand{\rt}{\reals^{3}}
\newcommand{\rN}{\reals^{N}}

\newcommand{\rttt}{\reals^{3\times3}}

\newcommand{\foh}{\frac{1}{2}}
\newcommand{\oh}{\nicefrac{1}{2}}

\newcommand{\ot}{\leftarrow}
\newcommand{\To}{\longrightarrow}

\newcommand{\equi}{\Leftrightarrow}

\DeclareMathOperator{\sym}{sym}
\DeclareMathOperator{\skw}{skw}

\DeclareMathOperator{\dev}{dev}
\DeclareMathOperator{\tr}{tr}
\DeclareMathOperator{\supp}{supp}

\DeclareMathOperator{\A}{A}
\newcommand{\As}{\A^{*}}
\newcommand{\Aa}{\A^{\prime}}
\DeclareMathOperator{\cA}{\mathcal{A}}
\newcommand{\cAs}{\cA^{*}}

\DeclareMathOperator{\Pot}{P}

\DeclareMathOperator{\Potop}{P}
\renewcommand{\Pot}{\Potop}

\DeclareMathOperator{\p}{\partial}

\DeclareMathOperator{\grad}{grad}
\DeclareMathOperator{\Grad}{Grad}
\DeclareMathOperator{\gradc}{\overset{\circ}{\grad}}
\DeclareMathOperator{\Gradc}{\overset{\circ}{\Grad}}
\DeclareMathOperator{\rot}{rot}
\DeclareMathOperator{\rotc}{\overset{\circ}{\rot}}

\DeclareMathOperator{\Rot}{Rot}
\DeclareMathOperator{\Rotc}{\overset{\circ}{\Rot}}

\DeclareMathOperator{\divergence}{div}
\renewcommand{\div}{\divergence}
\DeclareMathOperator{\divc}{\overset{\circ}{\div}}
\DeclareMathOperator{\Div}{Div}
\DeclareMathOperator{\Divc}{\overset{\circ}{\Div}}
\DeclareMathOperator{\ed}{d}
\DeclareMathOperator{\cd}{\delta}

\newcommand{\csymbol}{\mathsf{C}}

\newcommand{\cic}{\cgen{\circ}{\infty}{}}
\newcommand{\Cic}{\Cgen{\circ}{\infty}{}}

\newcommand{\cicom}{\cic(\om)}
\newcommand{\Cicom}{\Cic(\om)}

\newcommand{\lsymbol}{\mathsf{L}}

\newcommand{\lgen}[3]{\overset{#1}{\lsymbol}{}^{#2}_{#3}}

\newcommand{\lt}{\lgen{}{2}{}}

\newcommand{\ltz}{\lgen{}{2}{0}}

\newcommand{\ltom}{\lt(\om)}
\newcommand{\Ltom}{\Lt(\om)}

\newcommand{\ltzom}{\ltz(\om)}

\newcommand{\hsymbol}{\mathsf{H}}

\newcommand{\hgen}[3]{\overset{#1}{\hsymbol}{}^{#2}_{#3}}

\newcommand{\hz}{\hgen{}{0}{}}
\newcommand{\ho}{\hgen{}{1}{}}

\newcommand{\htwo}{\hgen{}{2}{}}

\newcommand{\hm}{\hgen{}{m}{}}

\newcommand{\hmpo}{\hgen{}{m+1}{}}

\newcommand{\hoc}{\hgen{\circ}{1}{}}
\newcommand{\Hoc}{\Hgen{\circ}{1}{}}

\newcommand{\hzom}{\hz(\om)}
\newcommand{\hoom}{\ho(\om)}

\newcommand{\htom}{\htwo(\om)}

\newcommand{\hmom}{\hm(\om)}

\newcommand{\hmpoom}{\hmpo(\om)}

\newcommand{\Hoom}{\Ho(\om)}
\newcommand{\Htom}{\Htwo(\om)}
\newcommand{\Hocom}{\Hoc(\om)}

\newcommand{\hocom}{\hoc(\om)}
\newcommand{\htcom}{\htwoc(\om)}

\newcommand{\hmcom}{\hmc(\om)}
\newcommand{\hmmocom}{\hmmoc(\om)}
\newcommand{\hmpocom}{\hmpoc(\om)}

\newcommand{\hmo}{\hgen{}{-1}{}}

\newcommand{\hmoom}{\hmo(\om)}
\newcommand{\Hmoom}{\Hmo(\om)}

\newcommand{\hmt}{\hgen{}{-2}{}}

\newcommand{\hmtom}{\hmt(\om)}

\newcommand{\hmm}{\hgen{}{-m}{}}

\newcommand{\hmmpo}{\hgen{}{-m+1}{}}
\newcommand{\hmmmo}{\hgen{}{-m-1}{}}
\newcommand{\hmmmt}{\hgen{}{-m-2}{}}

\newcommand{\hmmom}{\hmm(\om)}
\newcommand{\Hmmom}{\Hmm(\om)}
\newcommand{\hmmpoom}{\hmmpo(\om)}
\newcommand{\hmmmoom}{\hmmmo(\om)}
\newcommand{\hmmmtom}{\hmmmt(\om)}

\newcommand{\rom}{\r(\om)}
\newcommand{\rcom}{\rc(\om)}
\newcommand{\Rczom}{\Rcz(\om)}
\newcommand{\rczom}{\rcz(\om)}

\newcommand{\rhmmom}{\rhmm(\om)}

\newcommand{\rhmozom}{\rhmoz(\om)}

\newcommand{\rhmmzom}{\rhmmz(\om)}

\newcommand{\Rom}{\R(\om)}
\newcommand{\Rcom}{\Rc(\om)}
\newcommand{\rzom}{\rz(\om)}
\newcommand{\Rzom}{\Rz(\om)}

\newcommand{\dom}{\d(\om)}
\newcommand{\Dom}{\D(\om)}
\newcommand{\dDom}{\dD(\om)}
\newcommand{\dDzom}{\dDz(\om)}
\newcommand{\dcom}{\dc(\om)}
\newcommand{\Dcom}{\Dc(\om)}
\newcommand{\dzom}{\dz(\om)}
\newcommand{\dczom}{\dcz(\om)}

\newcommand{\dhmoom}{\dhmo(\om)}

\newcommand{\dhmmom}{\dhmm(\om)}

\newcommand{\dhmozom}{\dhmoz(\om)}

\newcommand{\dhmmzom}{\dhmmz(\om)}

\newcommand{\Dzom}{\Dz(\om)}

\newcommand{\ggom}{\gg(\om)}
\newcommand{\ggcom}{\ggc(\om)}

\newcommand{\ggczom}{\ggcz(\om)}

\newcommand{\harmsymbol}{\mathcal{H}}
\newcommand{\harmgen}[3]{\overset{#1}{\harmsymbol}{}^{#2}_{#3}}

\newcommand{\harmd}{\harmgen{}{}{\mathsf{D}}}

\newcommand{\harmdom}{\harmd(\om)}

\newcommand{\harmn}{\harmgen{}{}{\mathsf{N}}}

\newcommand{\harmnom}{\harmn(\om)}

\newcommand{\norm}[1]{|#1|}

\newcommand{\bnorm}[1]{\big|#1\big|}

\newcommand{\normltom}[1]{\norm{#1}_{\ltom}}
\newcommand{\normLtom}[1]{\norm{#1}_{\Ltom}}

\newcommand{\normhoom}[1]{\norm{#1}_{\hoom}}
\newcommand{\normHoom}[1]{\norm{#1}_{\Hoom}}

\newcommand{\normhmoom}[1]{\norm{#1}_{\hmoom}}
\newcommand{\normhmmom}[1]{\norm{#1}_{\hmmom}}

\newcommand{\scp}[2]{\langle#1,#2\rangle}

\newcommand{\bscp}[2]{\big\langle#1,#2\big\rangle}

\newcommand{\scpltom}[2]{\scp{#1}{#2}_{\ltom}}
\newcommand{\scpLtom}[2]{\scp{#1}{#2}_{\Ltom}}

\newcommand{\scphmoom}[2]{\scp{#1}{#2}_{\hmoom}}

\newcommand{\scphmmom}[2]{\scp{#1}{#2}_{\hmmom}}
\newcommand{\scpHmmom}[2]{\scp{#1}{#2}_{\Hmmom}}

\newcommand{\preprintudemath}[5]{
\thispagestyle{empty}
\Large
\begin{center}SCHRIFTENREIHE DER FAKULT\"AT F\"UR MATHEMATIK\end{center}
\vspace*{5mm}
\begin{center}#1\end{center}
\vspace*{5mm}
\begin{center}by\end{center}
\begin{center}#2\end{center}
\vspace*{5mm}
\begin{center}#3\hspace{80mm}#4\end{center}
\newpage
\thispagestyle{empty}
\vspace*{210mm}
Received: #5
\newpage
\addtocounter{page}{-2}
\normalsize}

\title[\sc\myshorttitle]{\Large\sf\mytitle}
\author{\myauthorone}
\author{\myauthortwo}
\address{\myaddressone}
\email[\myauthorone]{\myemailone}
\address{\myaddresstwo}
\email[\myauthortwo]{\myemailtwo}
\subjclass{\mysubjclass}
\date{\mydate}
\thanks{\mythanks}

\renewcommand{\H}{\hsymbol}
\DeclareMathOperator{\spn}{spn}
\DeclareMathOperator{\Az}{A_{0}}
\DeclareMathOperator{\Azs}{A_{0}^{*}}
\DeclareMathOperator{\Aza}{A_{0}^{\prime}}
\DeclareMathOperator{\cAz}{\mathcal{A}_{0}}
\DeclareMathOperator{\cAzs}{\mathcal{A}_{0}^{*}}

\DeclareMathOperator{\Ao}{A_{1}}
\DeclareMathOperator{\Aos}{A_{1}^{*}}
\DeclareMathOperator{\Aoa}{A_{1}^{\prime}}
\DeclareMathOperator{\cAo}{\mathcal{A}_{1}}
\DeclareMathOperator{\cAos}{\mathcal{A}_{1}^{*}}

\DeclareMathOperator{\At}{A_{2}}
\DeclareMathOperator{\Ats}{A_{2}^{*}}

\DeclareMathOperator{\cAt}{\mathcal{A}_{2}}
\DeclareMathOperator{\cAts}{\mathcal{A}_{2}^{*}}

\newcommand{\vecv}{v}
\newcommand{\vecw}{w}
\newcommand{\bfE}{\mathbf{E}}
\newcommand{\Pott}{\tilde{\Pot}}
\newcommand{\RM}{\mathsf{RM}}
\newcommand{\RT}{\mathsf{RT}}
\newcommand{\RTz}{\RT_{0}}

\renewcommand{\gradc}{\mathring{\grad}}
\renewcommand{\rotc}{\mathring{\rot}}
\renewcommand{\divc}{\mathring{\div}}
\renewcommand{\Gradc}{\mathring{\Grad}}
\renewcommand{\Rotc}{\mathring{\Rot}}
\renewcommand{\Divc}{\mathring{\Div}}
\newcommand{\RotcS}{\mathring{\Rot}_{\bbS}}
\newcommand{\DivcT}{\mathring{\Div}_{\bbT}}
\newcommand{\DivcS}{\mathring{\Div}_{\bbS}}
\newcommand{\Gradgrad}{\Grad\!\grad}
\newcommand{\divDiv}{\div\!\Div}
\newcommand{\divDivS}{\divDiv_{\bbS}}
\newcommand{\symRot}{\sym\!\Rot}
\newcommand{\symRotT}{\symRot_{\bbT}}
\newcommand{\devGrad}{\dev\!\Grad}
\newcommand{\symGrad}{\sym\!\Grad}
\newcommand{\symGradc}{\mathring{\symGrad}}
\newcommand{\Gradgradc}{\mathring{\Gradgrad}}
\renewcommand{\cic}{\mathring{\csymbol}{}^{\infty}}
\renewcommand{\Cic}{\cic}
\renewcommand{\hoc}{\mathring{\H}^{1}}
\renewcommand{\hocom}{\hoc(\om)}
\renewcommand{\htcom}{\mathring{\H}^{2}(\om)}
\newcommand{\htcomi}{\mathring{\H}^{2}(\om_{i})}
\renewcommand{\hmcom}{\mathring{\H}^{m}(\om)}
\renewcommand{\hmmocom}{\mathring{\H}^{m-1}(\om)}
\renewcommand{\hmpocom}{\mathring{\H}^{m+1}(\om)}
\newcommand{\hmmtcom}{\mathring{\H}^{m-2}(\om)}

\renewcommand{\rom}{\H(\rot,\om)}
\renewcommand{\rzom}{\H(\rot_{0},\om)}
\renewcommand{\rcom}{\H(\rotc,\om)}
\renewcommand{\rczom}{\H(\rotc_{0},\om)}

\renewcommand{\rhmozom}{\H^{-1}(\rot_{0},\om)}
\renewcommand{\rhmmom}{\H^{-m}(\rot,\om)}
\renewcommand{\rhmmzom}{\H^{-m}(\rot_{0},\om)}
\renewcommand{\dom}{\H(\div,\om)}
\renewcommand{\dzom}{\H(\div_{0},\om)}
\renewcommand{\dcom}{\H(\divc,\om)}
\renewcommand{\dczom}{\H(\divc_{0},\om)}
\renewcommand{\dhmoom}{\H^{-1}(\div,\om)}
\renewcommand{\dhmozom}{\H^{-1}(\div_{0},\om)}
\renewcommand{\dhmmom}{\H^{-m}(\div,\om)}
\renewcommand{\dhmmzom}{\H^{-m}(\div_{0},\om)}

\renewcommand{\Cicom}{\cicom}
\renewcommand{\Ltom}{\ltom}
\newcommand{\ltomS}{\lsymbol^2_{\bbS}(\om)}
\newcommand{\ltomT}{\lsymbol^2_{\bbT}(\om)}
\renewcommand{\Hoom}{\hoom}
\newcommand{\HoomS}{\H^1_{\bbS}(\om)}
\newcommand{\HoomT}{\H^1_{\bbT}(\om)}
\newcommand{\HoomiS}{\H^1_{\bbS}(\om_{i})}
\newcommand{\HoomiT}{\H^1_{\bbT}(\om_{i})}
\renewcommand{\Hoc}{\hoc}
\renewcommand{\Hocom}{\Hoc(\om)}
\newcommand{\HocomS}{\Hoc_{\bbS}(\om)}
\newcommand{\HocomT}{\Hoc_{\bbT}(\om)}
\newcommand{\HocomiS}{\Hoc_{\bbS}(\om_{i})}

\renewcommand{\Htom}{\htom}
\newcommand{\HtomS}{\H^{2}_{\bbS}(\om)}

\renewcommand{\Hmoom}{\hmoom}
\renewcommand{\Hmmom}{\hmmom}
\newcommand{\Hmmmoom}{\hmmmoom}
\renewcommand{\Rom}{\H(\Rot,\om)}
\renewcommand{\Rzom}{\H(\Rot_{0},\om)}
\renewcommand{\Rcom}{\H(\Rotc,\om)}
\renewcommand{\Rczom}{\H(\Rotc_{0},\om)}
\newcommand{\RomS}{\H_{\bbS}(\Rot,\om)}
\newcommand{\RcomS}{\H_{\bbS}(\Rotc,\om)}
\newcommand{\RczomS}{\H_{\bbS}(\Rotc_{0},\om)}
\newcommand{\RomT}{\H_{\bbT}(\Rot,\om)}
\newcommand{\RcomT}{\H_{\bbT}(\Rotc,\om)}

\renewcommand{\Dom}{\H(\Div,\om)}
\renewcommand{\Dzom}{\H(\Div_{0},\om)}
\renewcommand{\Dcom}{\H(\Divc,\om)}

\newcommand{\DomT}{\H_{\bbT}(\Div,\om)}
\newcommand{\DcomT}{\H_{\bbT}(\Divc,\om)}
\newcommand{\DczomT}{\H_{\bbT}(\Divc_{0},\om)}
\newcommand{\DomS}{\H_{\bbS}(\Div,\om)}
\newcommand{\DcomS}{\H_{\bbS}(\Divc,\om)}

\newcommand{\DcomiT}{\H_{\bbT}(\Divc,\om_{i})}
\newcommand{\devGom}{\H(\devGrad,\om)}
\newcommand{\devGzom}{\H(\devGrad_{0},\om)}
\newcommand{\symRom}{\H(\symRot,\om)}
\newcommand{\symRzom}{\H(\symRot_{0},\om)}
\newcommand{\symRomT}{\H_{\bbT}(\symRot,\om)}
\newcommand{\symRzomT}{\H_{\bbT}(\symRot_{0},\om)}
\newcommand{\symRzomiT}{\H_{\bbT}(\symRot_{0},\om_{i})}
\newcommand{\symRomiT}{\H_{\bbT}(\symRot,\om_{i})}
\newcommand{\RcomiS}{\H_{\bbS}(\Rotc,\om_{i})}
\newcommand{\RczomiS}{\H_{\bbS}(\Rotc_{0},\om_{i})}
\renewcommand{\dDom}{\H(\divDiv,\om)}
\newcommand{\dDomS}{\H_{\bbS}(\divDiv,\om)}
\renewcommand{\dDzom}{\H(\divDiv_{0},\om)}
\newcommand{\dDzomS}{\H_{\bbS}(\divDiv_{0},\om)}
\newcommand{\dDzomiS}{\H_{\bbS}(\divDiv_{0},\om_{i})}
\renewcommand{\ggom}{\H(\Gradgrad,\om)}
\renewcommand{\ggcom}{\H(\Gradgradc,\om)}
\renewcommand{\ggczom}{\H(\Gradgradc_{0},\om)}
\newcommand{\dDzmoomS}{\H^{0,-1}_{\bbS}(\divDiv,\om)}
\newcommand{\dDzmoomiS}{\H^{0,-1}_{\bbS}(\divDiv,\om_{i})}

\newcommand{\harmdSom}{\harmgen{}{}{\mathsf{D},\bbS}(\om)}
\newcommand{\harmnTom}{\harmgen{}{}{\mathsf{N},\bbT}(\om)}

\renewcommand{\bfE}{E}
\renewcommand{\bfF}{F}
\renewcommand{\bfM}{M}
\renewcommand{\bfN}{N}
\renewcommand{\bfI}{I}
\renewcommand{\bfPhi}{\Phi}
\renewcommand{\bfPsi}{\Psi}

\begin{document}
%%%%%%%%%%%%%%%%%%%%%%%%%%%%%%%%%%%%%%%%%%%%%%%%%%%%%%%%%%%%%%%%%%%%%%%%%%%%%%%%

%%%%%%%%%%%%%%%%%%%%%%%%%%%%%%%%
% report series Duisburg-Essen %
%%%%%%%%%%%%%%%%%%%%%%%%%%%%%%%%

\ifthenelse{\equal{\reportudemathyesno}{yes}}
{\preprintudemath{\mytitlerepude}{\myauthors}{\reportudemathnumber}{\reportudemathyear}{\reportudematheingang}}
{}

%%%%%%%%%%%%%%%%%%%%%%%%%%%%%%%%%%%%%%%%%%%%%%%%%%%%%%%%%%%%%%%%%%%%%%%%%%%%%%%%

\begin{abstract}
It is shown that the first biharmonic boundary value problem on a topologically trivial domain in 3D is
equivalent to three (consecutively to solve) second-order problems. 
This decomposition result is based on a Helmholtz-like decomposition 
of an involved non-standard Sobolev space of tensor fields and a proper characterization 
of the operator $\divDiv$ acting on this space. 
Similar results for biharmonic problems in 2D and their impact on the construction and analysis 
of finite element methods have been recently published in \cite{zulehner-2016-02}. 
The discussion of the kernel of $\divDiv$ leads to (de Rham-like) 
closed and exact Hilbert complexes, 
the $\divDiv$-complex and its adjoint the $\Gradgrad$-complex, 
involving spaces of trace-free and symmetric tensor fields.
For these tensor fields we show Helmholtz type decompositions
and, most importantly, new compact embedding results. 
Almost all our results hold and are formulated for general bounded
strong Lipschitz domains of arbitrary topology.
There is no reasonable doubt that our results extend to strong Lipschitz domains in $\rN$.
\end{abstract}

\maketitle
\tableofcontents

%%%%%%%%%%%%%%%%%%%%%%%%%%%%%%%%%%%%%%%%%%%%%%%%%%%%%%%%%%%%%%%%%%%%%%%%%%%%%%%%

\section{Introduction}
\mylabel{introsec}

In \cite{zulehner-2016-02} it was shown that the fourth-order biharmonic boundary value problem
\begin{equation} 
\label{primal}
\Delta^2 u=f\quad\text{in }\om,\qquad
u=\p_nu=0\quad\text{on }\ga,
\end{equation}
where $\om$ is a bounded and simply connected domain in $\mathbb{R}^2$ with a
(strong) Lipschitz boundary\footnote{$\ga$ is locally a graph of a Lipschitz function.} 
$\ga$,
can be decomposed into three second-order problems. 
The first problem is a Dirichlet-Poisson problem for an auxiliary scalar field $p$
\begin{align*}
\Delta p&=f
\quad\text{in }\om,
&
p&=0
\quad\text{on }\ga,
\intertext{the second problem is a linear elasticity Neumann problem for an auxiliary vector field $\vecv$}
\Div(\symGrad\vecv)
&=-\grad p
\quad\text{in }\om,\hspace*{20mm}
&
(\symGrad\vecv)\,n=-p\,n
&=0
\quad\text{on }\ga,
\intertext{and, finally, the third problem is again a Dirichlet-Poisson problem for the original scalar field $u$}
\Delta u
&=2\,p+\div\vecv
\quad\text{in }\om,
&
u
&=0
\quad\text{on }\ga.
\intertext{Note that the second equation is equivalent to}
\Div(\symGrad\vecv+p\,\bfI)
&=0
\quad\text{in }\om,
&
(\symGrad\vecv+p\,\bfI)\,n
&=0
\quad\text{on }\ga.
\end{align*}
Here $f$ is a given right-hand side, $\Delta$, $n$, and $\p_n$ denote the Laplace operator, 
the outward normal vector to the boundary, and the derivative in this  direction, respectively. 
In matrix notation the latter system reads as the symmetric system
$$\begin{bmatrix}
2 & \div & -\mathring{\Delta} \\
-\gradc & -\DivcS\symGrad &  0 \\
-\mathring{\Delta} & 0 & 0 
\end{bmatrix}
\begin{bmatrix}
p \\[1.3ex]
\vecv \\[1.3ex]
u
\end{bmatrix}
=
\begin{bmatrix}
0 \\[1.3ex]
0 \\[1.3ex]
-f
\end{bmatrix}$$
with $\div^{*}=-\gradc$
and $\mathring{\Delta}=\div\gradc$.
Throughout this paper, `mathrings' indicate natural homogeneous boundary conditions for different operators.
While $-\mathring{\Delta}$ is continuously invertible, 
$-\DivcS\symGrad$ is not on its domain of definition $D(\symGrad)=\Hoom$, 
but on the more regular space 
$$\Hoom\cap N(\symGrad)^{\bot_{\ltom}}
=\Hoom\cap\RM^{\bot_{\ltom}},$$
which is easy to handle. We will see that the situation in $\rt$ is much more complicated.
The differential operators $\grad$, $\div$, and (for later use) $\rot$ denote 
the gradient of a scalar field, the divergence and the rotation of a vector field, respectively.
The corresponding capitalized differential operators $\Grad$, $\Div$, and $\Rot$ 
denote the row-wise application of $\grad$ to a vector field, $\div$ and $\rot$ to a tensor field. 
The prefix $\sym$ is used for the symmetric part of a matrix, 
for the skew-symmetric part we use the prefix $\skw$.
This decomposition is of triangular structure, i.e., 
the first problem is a well-posed second-order problem in $p$, 
the second problem is a well-posed second-order problem in $\vecv$ for given $p$, 
and the third problem is a well-posed  second-order problem in $u$ for given $p$ and $\vecv$. 
This allows to solve them consecutively analytically or numerically 
by means of techniques for second-order problems. 
 
This is - in the first place - a new analytic result for fourth-order problems. 
But it also has interesting implications for discretization methods applied to \eqref{primal}. 
It allows to re-interpret known finite element methods as well as 
to construct new discretization methods for \eqref{primal}
by exploiting the decomposable structure of the problem. 
In particular, it was shown in \cite{zulehner-2016-02} that the Hellan-Herrmann-Johnson mixed method 
(see \cite{hellan:67,herrmann:67,johnson:73}) for \eqref{primal} 
allows a similar decomposition as the continuous problem, 
which leads to a new and simpler assembling procedure for the discretization matrix 
and to more efficient solution techniques for the discretized problem. 
Moreover, a novel conforming variant of the Hellan-Herrmann-Johnson mixed method 
was found based on the decomposition.

The main application of this paper is to derive a similar decomposition result for biharmonic problems 
\eqref{primal} on bounded and topologically trivial three-dimensional domains $\om\subset\rt$ with a
(strong) Lipschitz boundary $\ga$. 
For this we proceed as in \cite{zulehner-2016-02} and reformulate \eqref{primal} using 
$$\Delta^2=\divDiv\Gradgrad$$
as a mixed problem by introducing the (negative) Hessian of the original scalar field $u$ 
as an auxiliary tensor field
\begin{equation} 
\label{defM}
\bfM=-\Gradgrad u\in\ltomS.
\end{equation}
Then the biharmonic differential equation reads
\begin{equation} 
\label{divDivM}
-\divDiv\bfM=f\quad\text{in }\om.
\end{equation}
For an appropriate non-standard Sobolev space for $\bfM$ 
it can be shown that the mixed problem in $\bfM$ and $u$ is well-posed,
see \eqref{var_mixed_final_one}-\eqref{var_mixed_final_two}. 
Then the decomposition of the biharmonic problem follows from a
regular decomposition of this non-standard Sobolev space, see Lemma \ref{Helmholtz1}. 
This part of the analysis carries over  
completely from the two-dimensional case to the three-dimensional case 
and is recalled in Section \ref{applsec}.
To efficiently utilize this
regular decomposition for the decomposition 
of the biharmonic problem an appropriate characterization of the kernel 
of the operator $\divDiv$ is required, which is well understood for the two-dimensional case, 
see, e.g., \cite{Beirao:2007a,huang:11,zulehner-2016-02}. 
Its extension to the three-dimensional case is one of the central topics of this paper.
We expect - as in the two-dimensional case - similar interesting implications 
for the study of appropriate discretization methods for four-order problems 
in the three-dimensional case.

Another application comes from the theory for general relativity and gravitational waves.
There, the so called linearized Einstein-Bianchi system reads as the Maxwell's equations
\begin{align*}
\p_{t}E+\Rot B=F,\quad
\Div E=f\qquad\text{in }\om,\\
\p_{t}B-\Rotc\,E=G,\quad
\Divc\,B=g\qquad\text{in }\om,
\end{align*}
but with symmetric and deviatoric (trace-free) tensor fields $E$ and $B$,
see \cite{Quenneville:2015} for more details, especially on the modeling.

The paper is organized as follows:
In Subsection \ref{mainressec} of this introduction
we will present some of the main results in a non-rigorous way
and the application to the three-dimensional biharmonic equation,
i.e., to \eqref{primal} for $\om\subset\rt$.
The mathematically rigorous part, where also all precise definitions can be found, begins with
preliminaries in Section \ref{prelimsec} and
introduces our general functional analytical setting.
Then we will discuss the relevant unbounded linear operators,
show closed and exact Hilbert complex properties, and
present a suitable representation of the kernel of $\divDiv$ 
for the three-dimensional case in Section \ref{kernelsecsimple} for topologically trivial domains. 
In Section \ref{kernelsecgeneral} we extend our results to (strong) Lipschitz domains
with arbitrary topology based on two new and crucial compact embeddings.
In the final Section \ref{applsec} we give a detailed study of
the application of our results to the three-dimensional biharmonic equation from Section \ref{mainressec}.
The proofs of some useful identities are presented in an appendix.

\subsection{Some Main Results}
\mylabel{mainressec}

Let $\om\subset\rt$ be a bounded and topologically trivial strong Lipschitz domain.
Based on a decomposition result of the non-standard Hilbert space
\begin{align}
\mylabel{firstdecohzmodd}
\dDzmoomS 
=\set{\bfM\in\ltomS}{\divDiv\bfM\in\hmoom},
\end{align}
see Lemma \ref{Helmholtz1}, where $\ltomS$ denotes the symmetric $\Ltom$-tensor fields, and a
representation of the kernel of $\divDivS$ as the range of symmetric rotations 
of deviatoric tensor fields, i.e., 
$$N(\divDivS)=R(\symRotT),$$
a decomposition of the three-dimensional 
biharmonic problem \eqref{primal} into three (consecutively to solve) second-order problems 
will be rigorously derived in Section \ref{applsec}.
For details, see \eqref{finalvareqone}-\eqref{finalvareqthree}
and the strong equations after the corresponding proof.
More precisely, the three resulting second order equations 
are a Dirichlet-Poisson problem for the auxiliary scalar function $p$
$$\Delta p=f\quad\text{in }\om,\qquad 
p=0\quad\text{on }\ga,$$
a second-order Neumann type $\Rot\symRot$-$\Div$-system for the auxiliary tensor field $\bfE$
\begin{align*}
\tr\bfE
&=0,
&
\Rot\symRot\bfE
&
=\spn\grad p,
&
\Div\bfE
&=0
&
&\text{in }\om,\\
&&
n\times\symRot\bfE
&
=p\spn n=0,
&
\bfE\,n
&=0
&
&\text{on }\ga,
\end{align*}
and, finally, a Dirichlet-Poisson problem for the original scalar function $u$
$$\Delta u=3p+\tr\symRot\bfE=\tr\,(p\,\bfI+\symRot\bfE)\quad\text{in }\om,\qquad 
u=0\quad\text{on }\ga.$$
The second system is equivalent to
\begin{align*}
\tr\bfE
&=0,
&
\Rot(\symRot\bfE+p\,\bfI)
&
=0,
&
\Div\bfE
&=0
&
&\text{in }\om,\\
&&
n\times(\symRot\bfE+p\,\bfI)
&
=0,
&
\bfE\,n
&=0
&
&\text{on }\ga.
\end{align*}
In matrix notation the latter system reads as the symmetric system
$$\begin{bmatrix}
3 & \tr\symRotT & -\mathring{\Delta} \\
\RotcS(\,\cdot\,\bfI) & \RotcS\symRotT &  0 \\
-\mathring{\Delta} & 0 & 0 
\end{bmatrix}
\begin{bmatrix}
p \\[1.3ex]
\bfE \\[1.3ex]
u
\end{bmatrix}
=
\begin{bmatrix}
0 \\[1.3ex]
0 \\[1.3ex]
-f
\end{bmatrix}$$
with $(\tr\symRotT)^{*}=\RotcS(\,\cdot\,\bfI)$
and $\mathring{\Delta}=\div\gradc$.
While $-\mathring{\Delta}$ is continuously invertible, 
$\RotcS\symRotT$ is not on its domain of definition $D(\symRotT)$, but on the more regular space 
\begin{align}
\mylabel{DsymRotTreg}
D(\symRotT)\cap N(\symRotT)^{\bot_{\ltomT}}
=D(\symRotT)\cap R(\devGrad)^{\bot_{\ltomT}}
=D(\symRotT)\cap N(\DivcT),
\end{align}
which leads to another difficulty. This is a well known and typical situation, e.g., in the theory
of static Maxwell equations, and it will turn out that it results into a symmetric saddle point system,
see Theorem \ref{finalvarformapp} and \eqref{dsadpointmatrix}.

The above mentioned crucial regular type decomposition of the space \eqref{firstdecohzmodd} will be proved
in Lemma \ref{Helmholtz1} and shows the direct and topological (continuous) decomposition
$$\dDzmoomS
=\hocom\cdot\bfI
\dotplus
\dDzomS.$$
Hence, the kernel $N(\divDivS)=\dDzomS$ is an important object.
In Theorem \ref{maintheo} we will show
\begin{align*}
\dDzomS
&=N(\divDivS)
=R(\symRotT)
=\symRot\HoomT\\
&=\symRot\symRomT
=\symRot\big(\symRomT\cap\DczomT\big).
\end{align*}
Especially, the range $R(\symRotT)$ is closed.
The potential on the right hand side of the first line is called a regular potential
and the potential on the right hand side of the second line is uniquely determined.
Both potentials depend continuously on the data. Here,
$\HoomT$ is the Sobolev space of deviatoric $\Hoom$-tensor fields and
$$\symRomT=D(\symRotT),\qquad
\DczomT=N(\DivcT).$$
Moreover, a corresponding Poincar\'e type estimate 
$$\exists\,c_{\mathsf{R}}>0\quad
\forall\,\bfE\in\symRomT\cap\DczomT\qquad
\normLtom{\bfE}\leq c_{\mathsf{R}}\,\normLtom{\symRot\bfE}$$
as well as a Helmholtz type decomposition
$$\ltomT
=\DczomT
\oplus_{\ltomT}
\symRzomT$$
hold. Similar results hold for the kernels of $\DivcT$ and $\symRotT$, which we have already used in \eqref{DsymRotTreg}.
More precisely, Theorem \ref{maintheo} also shows
\begin{align*}
\symRzomT
&=N(\symRotT)
=R(\devGrad)
=\devGrad\hoom
=\devGrad\big(\hoom\cap\RTz^{\bot_{\ltom}}\big),\\
\DczomT
&=N(\DivcT)
=R(\RotcS)
=\Rot\HocomS\\
&=\Rot\RcomS
=\Rot\big(\RcomS\cap\dDzomS\big)
\end{align*}
with the same properties of the respective potentials, and
\begin{align*}
\exists\,c_{\mathsf{D}}&>0
&
\forall\,\vecv&\in\hoom\cap\RTz^{\bot_{\ltom}}
&
\normltom{\vecv}&\leq c_{\mathsf{D}}\,\normLtom{\devGrad\vecv},\\
&
&
\forall\,\bfM&\in\RcomS\cap\dDzomS
&
\normLtom{\bfM}&\leq c_{\mathsf{R}}\,\normLtom{\Rot\bfM}=c_{\mathsf{R}}\,\normLtom{\dev\Rot\bfM},
\end{align*}
where $c_{\mathsf{R}}$ is the same as before 
and $\RTz$ denotes the space of lowest order Raviart-Thomas affine linear vector fields.

Our results rely on the study of the corresponding Hilbert complex
$$\begin{CD}
\{0\} @> 0 >>
\htcom @> \Gradgradc >>
\RcomS @> \RotcS >>
\DcomT @> \DivcT >>
\ltom @> \pi_{\RTz} >>
\RTz
\end{CD}$$
and its dual or adjoint Hilbert complex
$$\begin{CD}
\{0\} @< 0 <<
\ltom @< \divDivS <<
\dDomS @< \symRotT << 
\symRomT @< -\devGrad <<
\hoom @< \iota_{\RTz} <<
\RTz
\end{CD},$$
which will turn out to be closed (closed ranges)
an exact (trivial cohomology groups).
Here, the densely defined, closed, and unbounded linear operators
$\Gradgradc$, $\RotcS$, and $\DivcT$ are given as closures of
\begin{align*}
\widetilde\Gradgradc:\cicom\subset\ltom
&\To\ltomS,
&
u
&\mapsto\Gradgrad u,\\
\widetilde\RotcS:\Cicom\cap\ltomS\subset\ltomS
&\To\ltomT,
&
\bfM
&\mapsto\Rot\bfM,\\
\widetilde\DivcT:\Cicom\cap\ltomT\subset\ltomT
&\To\ltom,
&
\bfE
&\mapsto\Div\bfE
\end{align*}
with domains of definition
$$D(\Gradgradc)=\htcom,\qquad
D(\RotcS)=\RcomS,\qquad
D(\DivcT)=\DcomT$$
and kernels
\begin{align*}
N(\Gradgradc)
=\ggczom
&=\{0\},
&
N(\RotcS)
&=\RczomS,
&
N(\DivcT)
&=\DczomT.
\end{align*}
The adjoints are
\begin{align*}
\Gradgradc^{*}=\divDivS:\dDomS\subset\ltomS
&\To\ltom,
&
\bfM
&\mapsto\divDiv\bfM,\\
\RotcS^{*}=\symRotT:\symRomT\subset\ltomT
&\To\ltomS,
&
\bfE
&\mapsto\symRot\bfE,\\
\DivcT^{*}=-\devGrad:\devGom=\hoom\subset\ltom
&\To\ltomT,
&
\vecv
&\mapsto-\devGrad\vecv
\end{align*}
with kernels
\begin{align*}
N(\divDivS)
&=\dDzomS,
&
N(\symRotT)
&=\symRzomT,
&
N(\devGrad)
&=\RTz.
\end{align*}
In this contribution we will prove all important tools to handle
pde-systems involving the latter operators, such as
Helmholtz type decompositions, potentials, regular decompositions, regular potentials,
Poincar\'e type estimates, closed ranges, exactness, and, 
most importantly, the key property, that certain canonical embeddings are compact, e.g.,
$$D(\RotcS)\cap D(\divDivS)
=\RcomS\cap\dDomS\overset{\text{\sf cpt}}{\hookrightarrow}\ltomS$$
or
$$D(\symRotT)\cap D(\DivcT)
=\symRomT\cap\DcomT\overset{\text{\sf cpt}}{\hookrightarrow}\ltomT.$$

In principle, such results are known in simpler situations, e.g.,
in electro-magnetic theory or linear elasticity.
In electro-magnetic theory (Maxwell's equations) one has to deal with 
the de Rham complex ($\grad$-$\rot$-$\div$-complex), i.e., 
with the closed and exact Hilbert complex and its adjoint
$$\begin{CD}
\{0\} @> 0 >>
\hocom @> \gradc >>
\rcom @> \rotc >>
\dcom @> \divc >>
\ltom @> \pi_{\reals} >>
\reals
\end{CD},$$
$$\begin{CD}
\{0\} @< 0 <<
\ltom @< -\div <<
\dom @< \rot << 
\rom @< -\grad <<
\hoom @< \iota_{\reals} <<
\reals
\end{CD},$$
which have a well known generalization to differential forms 
and exterior derivatives $\ed$, $\mathring\ed$ and co-derivatives $\cd$, $\mathring\cd$ as well.
In linear elasticity we observe the elasticity complexes ($\Rot\!\Rot^{\top}$-complexes), i.e.,
$$\begin{CD}
\{0\} @> 0 >>
\hocom @> \symGradc >>
\H_{\bbS}(\mathring{\Rot\!\Rot}{}^{\top},\om) @> \mathring{\Rot\!\Rot}_{\bbS}\hspace*{-1mm}{}^{\top} >>
\H_{\bbS}(\Divc,\om) @> \DivcS >>
\ltom @> \pi_{\RM} >>
\RM
\end{CD},$$
$$\begin{CD}
\{0\} @< 0 <<
\ltom @< -\Div_{\bbS} <<
\H_{\bbS}(\Div,\om) @< \Rot\!\Rot_{\bbS}\hspace*{-1mm}{}^{\top} << 
\H_{\bbS}(\Rot\!\Rot\hspace*{-1mm}{}^{\top},\om) @< -\symGrad <<
\hoom @< \iota_{\RM} <<
\RM
\end{CD}.$$
Note that these complexes admit certain symmetries, which is not the case
for the $\Gradgrad$-$\divDiv$-complexes. On the other hand, there is no reasonable doubt that
similar results hold for the other set of boundary conditions as well, i.e., for the Hilbert complexes
$$\begin{CD}
\mathsf{P}_{1} @> \iota_{\mathsf{P}_{1}} >>
\htom @> \Gradgrad >>
\RomS @> \Rot_{\bbS} >>
\DomT @> \Div_{\bbT} >>
\ltom @> 0 >>
\{0\}
\end{CD},$$
$$\begin{CD}
\mathsf{P}_{1} @< \pi_{\mathsf{P}_{1}} <<
\ltom @< \mathring{\divDivS} <<
\H_{\bbS}(\mathring{\divDiv},\om) @< \mathring{\symRotT} << 
\H_{\bbT}(\mathring{\symRot},\om) @< -\mathring{\devGrad} <<
\hocom @< 0 <<
\{0\}
\end{CD}.$$

\section{Preliminaries}
\mylabel{prelimsec}

We start by recalling some basic concepts and abstract results from functional analysis
concerning Helmholtz decompositions, closed ranges, Friedrichs/Poincar\'e type estimates,
and bounded or even compact inverse operators.
Since we will need both the Banach space setting for bounded linear operators 
as well as the Hilbert space setting for (possibly unbounded) 
closed and densely defined linear operators, we will shortly recall these two variants.

\subsection{Functional Analysis Toolbox}
\label{fuanasec}

Let $\X$ and $\Y$ be real Banach spaces.
With $BL(\X,\Y)$ we introduce the space of bounded linear operators mapping $\X$ to $\Y$.
The dual spaces of $\X$ and $\Y$ are denoted by $\X^{\prime}:=BL(\X,\reals)$ and $\Y^{\prime}:=BL(\Y,\reals)$.
For a given $\A\in BL(\X,\Y)$ we write $\Aa\in BL(\Y^{\prime},\X^{\prime})$ 
for its Banach space dual or adjoint operator
defined by $\Aa y^{\prime}(x):=y^{\prime}(\A x)$ 
for all $y^{\prime}\in\Y^{\prime}$ and all $x\in\X$.
Norms and duality in $\X$ resp. $\X^{\prime}$
are denoted by $\norm{\,\cdot\,}_{\X}$, $\norm{\,\cdot\,}_{\X^{\prime}}$, 
and $\scp{\,\cdot\,}{\,\cdot\,}_{\X^{\prime}}$.

Suppose $\H_{1}$ and $\H_{2}$ are Hilbert spaces.
For a (possibly unbounded) densely defined linear operator $\A\!:\!D(\A)\subset\H_{1}\to\H_{2}$
we recall that its Hilbert space dual or adjoint
$\As\!:\!D(\As)\subset\H_{2}\to\H_{1}$ can be defined
via its Banach space adjoint $\Aa$ and the Riesz isomorphisms of $\H_{1}$ and $\H_{2}$
or directly as follows:
$y\in D(\As)$
if and only if $y\in\H_{2}$ and 
$$\exists\,f\in\H_{1}\quad
\forall\,x\in D(\A)\quad
\scp{\A x}{y}_{\H_{2}}=\scp{x}{f}_{\H_{1}}.$$
In this case we define $\As y:=f$.
We note that $\As$ has maximal domain of definition and that $\As$ is characterized by
$$\forall\,x\in D(\A)\quad
\forall\,y\in D(\As)\qquad
\scp{\A x}{y}_{\H_{2}}=\scp{x}{\As y}_{\H_{1}}.$$
Here $\scp{\,\cdot\,}{\,\cdot\,}_{\H}$ denotes the scalar product in a Hilbert space $\H$ 
and $D$ is used for the domain of definition of a linear operator.
Additionally, we introduce the notation $N$ for the kernel or null space 
and $R$ for the range of a linear operator.

Let $\A\!:\!D(\A)\subset\H_{1}\to\H_{2}$ be a (possibly unbounded) 
closed and densely defined linear operator on two Hilbert spaces $\H_{1}$ and $\H_{2}$
with adjoint $\As\!:\!D(\As)\subset\H_{2}\to\H_{1}$.
Note $(\As)^{*}=\ovl{\A}=\A$, i.e., $(\A,\As)$ is a dual pair.
By the projection theorem the Helmholtz type decompositions
\begin{align}
\mylabel{helm}
\H_{1}=N(\A)\oplus_{\H_{1}}\ovl{R(\As)},\quad
\H_{2}=N(\As)\oplus_{\H_{2}}\ovl{R(\A)}
\end{align}
hold and we can define the reduced operators
\begin{align*}
\cA&:=\A|_{\ovl{R(\As)}}:D(\cA)\subset\ovl{R(\As)}\to\ovl{R(\A)},&
D(\cA)&:=D(\A)\cap N(\A)^{\bot_{\H_{1}}}=D(\A)\cap\ovl{R(\As)},\\
\cAs&:=\As|_{\ovl{R(\A)}}:D(\cAs)\subset\ovl{R(\A)}\to\ovl{R(\As)},&
D(\cAs)&:=D(\As)\cap N(\As)^{\bot_{\H_{2}}}=D(\As)\cap\ovl{R(\A)},
\end{align*}
which are also closed and densely defined linear operators.
We note that $\cA$ and $\cAs$ are indeed adjoint to each other, i.e.,
$(\cA,\cAs)$ is a dual pair as well. Now the inverse operators 
$$\cA^{-1}:R(\A)\to D(\cA),\qquad
(\cAs)^{-1}:R(\As)\to D(\cAs)$$
exist and they are bijective, 
since $\cA$ and $\cAs$ are injective by definition.
Furthermore, by \eqref{helm} we have
the refined Helmholtz type decompositions
\begin{align}
\label{DacA}
D(\A)&=N(\A)\oplus_{\H_{1}}D(\cA),&
D(\As)&=N(\As)\oplus_{\H_{2}}D(\cAs)
\intertext{and thus we obtain for the ranges}
\label{RacA}
R(\A)&=R(\cA),&
R(\As)&=R(\cAs).
\end{align}

By the closed range theorem and the closed graph theorem we get immediately the following.

\begin{lem}
\label{poincarerange}
The following assertions are equivalent:
\begin{itemize}
\item[\bf(i)] 
$\exists\,c_{\A}\in(0,\infty)$ \quad 
$\forall\,x\in D(\cA)$ \qquad 
$\norm{x}_{\H_{1}}\leq c_{\A}\norm{\A x}_{\H_{2}}$
\item[\bf(i${}^{*}$)] 
$\exists\,c_{\As}\in(0,\infty)$ \quad 
$\forall\,y\in D(\cAs)$ \qquad 
$\norm{y}_{\H_{2}}\leq c_{\As}\norm{\As y}_{\H_{1}}$
\item[\bf(ii)] 
$R(\A)=R(\cA)$ is closed in $\H_{2}$.
\item[\bf(ii${}^{*}$)] 
$R(\As)=R(\cAs)$ is closed in $\H_{1}$.
\item[\bf(iii)] 
$\cA^{-1}:R(\A)\to D(\cA)$ is continuous and bijective
with norm bounded by $(1+c_{\A}^2)^{\oh}$.
\item[\bf(iii${}^{*}$)] 
$(\cAs)^{-1}:R(\As)\to D(\cAs)$ is continuous and bijective
with norm bounded by $(1+c_{\As}^2)^{\oh}$.
\end{itemize}
\end{lem}

In case that one of the assertions of Lemma \ref{poincarerange} is true,
e.g., $R(\A)$ is closed, we have
\begin{align*}
\H_{1}
&=N(\A)\oplus_{\H_{1}}R(\As),
&
\H_{2}
&=N(\As)\oplus_{\H_{2}}R(\A),\\
D(\A)
&=N(\A)\oplus_{\H_{1}}D(\cA),
&
D(\As)
&=N(\As)\oplus_{\H_{2}}D(\cAs),\\
D(\cA)
&=D(\A)\cap R(\As),
&
D(\cAs)
&=D(\As)\cap R(\A).
\end{align*}

For the ``best'' constants $c_{\A}$, $c_{\As}$ we have the following lemma.

\begin{lem}
\label{lemconstants}
The Rayleigh quotients 
$$\frac{1}{c_{\A}}
:=\inf_{0\neq x\in D(\cA)}\frac{\norm{\A x}_{\H_{2}}}{\norm{x}_{\H_{1}}}
=\inf_{0\neq y\in D(\cAs)}\frac{\norm{\As y}_{\H_{1}}}{\norm{y}_{\H_{2}}}
=:\frac{1}{c_{\As}}$$
coincide, i.e., $c_{\A}=c_{\As}$, if either $c_{\A}$ or $c_{\As}$ exists in $(0,\infty)$.
Otherwise they also coincide, i.e., it holds $c_{\A}=c_{\As}=\infty$.
\end{lem}

From now on and throughout this paper,
we always pick the best possible constants in the various 
Friedrichs/Poincar\'e type estimates.

%\begin{proof}
%Let, e.g., $c_{\As}$ exist in $(0,\infty)$. 
%By Lemma \ref{poincarerange} also $c_{\A}$ exists in $(0,\infty)$,
%$R(\As)$ is closed in $\H_{1}$, and 
%for any $0\neq x\in D(\cA)=D(\A)\cap R(\As)$ 
%there exists $y\in D(\cAs)$ with $x=\As y$. 
%More precisely, $y:=(\cAs)^{-1}x\in D(\cAs)$ is uniquely determined
%and we have $\norm{y}_{\H_{2}}\leq c_{\As}\norm{\As y}_{\H_{1}}$. 
%But then
%$$\norm{x}_{\H_{1}}^2
%=\scp{x}{\As y}_{\H_{1}}
%=\scp{\A x}{y}_{\H_{2}}
%\leq \norm{\A x}_{\H_{2}}\norm{y}_{\H_{2}}
%\leq c_{\As}\norm{\A x}_{\H_{2}}\norm{\As y}_{\H_{1}},$$
%yielding $\norm{x}_{\H_{1}}\leq c_{\As}\norm{\A x}_{\H_{2}}$.
%Therefore, $0<c_{\A}\leq c_{\As}$ 
%and by symmetry we obtain $c_{\A}=c_{\As}$.
%\end{proof}

A standard indirect argument shows the following.

\begin{lem}
\label{lemtoolbox}
Let $D(\cA)=D(\A)\cap\ovl{R(\As)}\hookrightarrow\H_{1}$ be compact. 
Then the assertions of Lemma \ref{poincarerange} hold.
Moreover, the inverse operators 
$$\cA^{-1}:R(\A)\to R(\As),\quad
(\cAs)^{-1}:R(\As)\to R(\A)$$
are compact with norms 
$\bnorm{\cA^{-1}}_{R(\A),R(\As)}
=\bnorm{(\cAs)^{-1}}_{R(\As),R(\A)}
=c_{\A}$.
\end{lem}

%\begin{proof}
%If, e.g., Lemma \ref{poincarerange} (i) was wrong,
%there exists a sequence $(x_{n})\subset D(\cA)$ with $\norm{x_{n}}_{\H_{1}}=1$
%and $\A x_{n}\to0$. As $(x_{n})$ is bounded in $D(\cA)$
%we can extract a subsequence, again denoted by $(x_{n})$,
%with $x_{n}\to x\in\H_{1}$ in $\H_{1}$. Since $\A$ is closed,
%we have $x\in D(\A)$ and $\A x=0$. Hence $x\in N(\A)$.
%On the other hand, $(x_{n})\subset D(\cA)\subset\ovl{R(\As)}=N(\A)^{\bot}$
%implies $x\in N(\A)^{\bot}$. Thus $x=0$, in contradiction to $1=\norm{x_{n}}_{\H_{1}}\to\norm{x}_{\H_{1}}=0$.
%\end{proof}

Moreover, we have

\begin{lem}
\mylabel{cptequi}
$D(\cA)\hookrightarrow\H_{1}$ is compact,
if and only if $D(\cAs)\hookrightarrow\H_{2}$ is compact.
\end{lem}

%\begin{proof}
%By symmetry it is enough to show one direction.
%Let $D(\cA)\hookrightarrow\H_{1}$ be compact 
%and let $(y_{n})\subset D(\cAs)$ be a bounded sequence.
%By Lemma \ref{poincarerange} and Lemma \ref{lemtoolbox} we get $y_{n}=\A x_{n}$ 
%with $(x_{n})\subset D(\cA)$, which is bounded in $D(\cA)$ 
%by Lemma \ref{poincarerange} (i). Hence we may extract a subsequence,
%again denoted by $(x_{n})$ converging in $\H_{1}$. 
%Thus with $x_{n,m}:=x_{n}-x_{m}$ and $y_{n,m}:=y_{n}-y_{m}$ we see
%\begin{align*}
%\norm{y_{n,m}}_{\H_{2}}^2
%&=\bscp{y_{n,m}}{\A(x_{n,m})}_{\H_{2}}
%=\bscp{\As(y_{n,m})}{x_{n,m}}_{\H_{1}}
%\leq c\,\norm{x_{n,m}}_{\H_{1}},
%\end{align*}
%and hence $(y_{n})$ is a Cauchy sequence in $\H_{2}$.
%\end{proof}

Now, let $\Az\!:\!D(\Az)\subset\H_{0}\to\H_{1}$ and $\Ao\!:\!D(\Ao)\subset\H_{1}\to\H_{2}$ 
be (possibly unbounded) closed and densely defined linear operators 
on three Hilbert spaces $\H_{0}$, $\H_{1}$ and $\H_{2}$
with adjoints $\Azs\!:\!D(\Azs)\subset\H_{1}\to\H_{0}$
and $\Aos\!:\!D(\Aos)\subset\H_{2}\to\H_{1}$
as well as reduced operators $\cAz$, $\cAzs$, and $\cAo$, $\cAos$.
Furthermore, we assume the sequence or complex property of $\Az$ and $\Ao$, 
that is, $\Ao\Az=0$, i.e.,
\begin{align}
\mylabel{sequenceprop}
R(\Az)\subset N(\Ao).
\end{align}
Then also $\Azs\Aos=0$, i.e., $R(\Aos)\subset N(\Azs)$.
The Helmholtz type decompositions of \eqref{helm} for $\A=\Ao$ and $\A=\Az$ read
\begin{align}
\label{helmappcl}
\H_{1}&=N(\Ao)\oplus_{\H_{1}}\ovl{R(\Aos)},
&
\H_{1}&=N(\Azs)\oplus_{\H_{1}}\ovl{R(\Az)}
\intertext{and by \eqref{sequenceprop} we see}
\label{helmkernappcl}
N(\Azs)&=N_{0,1}\oplus_{\H_{1}}\ovl{R(\Aos)},
&
N(\Ao)&=N_{0,1}\oplus_{\H_{1}}\ovl{R(\Az)},
&
N_{0,1}&:=N(\Ao)\cap N(\Azs)
\end{align}
yielding the refined Helmholtz type decomposition
\begin{align}
\label{helmappfullcl}
\H_{1}=\ovl{R(\Az)}\oplus_{\H_{1}}N_{0,1}\oplus_{\H_{1}}\ovl{R(\Aos)},\qquad
R(\Az)=R(\cAz),\qquad
R(\Aos)=R(\cAos).
\end{align}
The previous results of this section imply immediately the following.

\begin{lem}
\label{helmrefined}
Let $\Az$, $\Ao$ be as introduced before with $\Ao\Az=0$, i.e., \eqref{sequenceprop}.
Moreover, let $R(\Az)$ and $R(\Ao)$ be closed.
Then, the assertions of Lemma \ref{poincarerange} and Lemma \ref{lemconstants}
hold for $\Az$ and $\Ao$. Moreover, the refined Helmholtz type decompositions
\begin{align*}
\H_{1}
&=R(\Az)\oplus_{\H_{1}}N_{0,1}\oplus_{\H_{1}}R(\Aos),
&
N_{0,1}
&=N(\Ao)\cap N(\Azs),\\
N(\Ao)
&=R(\Az)\oplus_{\H_{1}}N_{0,1},
&
N(\Azs)
&=N_{0,1}\oplus_{\H_{1}}R(\Aos),\\
%R(\cAz)=R(\Az)
%&=N(\Ao)\ominus_{\H_{1}}N_{0,1},
%&
%R(\cAos)=R(\Aos)
%&=N(\Azs)\ominus_{\H_{1}}N_{0,1},\\
D(\Ao)
&=R(\Az)\oplus_{\H_{1}}N_{0,1}\oplus_{\H_{1}}D(\cAo),
&
D(\Azs)
&=D(\cAzs)\oplus_{\H_{1}}N_{0,1}\oplus_{\H_{1}}R(\Aos),\\
D(\Ao)\cap D(\Azs)
&=D(\cAzs)\oplus_{\H_{1}}N_{0,1}\oplus_{\H_{1}}D(\cAo)
\end{align*}
hold. Especially, $R(\Az)$, $R(\Azs)$, $R(\Ao)$, and $R(\Aos)$ are closed,
the respective inverse operators, i.e.,
\begin{align*}
\cAz^{-1}&:R(\Az)\to D(\cAz),
&
\cAo^{-1}&:R(\Ao)\to D(\cAo),\\
(\cAzs)^{-1}&:R(\Azs)\to D(\cAzs),
&
(\cAos)^{-1}&:R(\Aos)\to D(\cAos),
\end{align*}
are continuous, and there exist positive constants $c_{\Az}$, $c_{\Ao}$,
such that the Friedrichs/Poincar\'e type estimates
\begin{align*}
\forall\,x&\in D(\cAz)
&
\norm{x}_{\H_{0}}&\leq c_{\Az}\norm{\Az x}_{\H_{1}},
&
\forall\,y&\in D(\cAo)
&
\norm{y}_{\H_{1}}&\leq c_{\Ao}\norm{\Ao y}_{\H_{2}},\\
\forall\,y&\in D(\cAzs)
&
\norm{y}_{\H_{1}}&\leq c_{\Az}\norm{\Azs y}_{\H_{0}},
&
\forall\,z&\in D(\cAos)
&
\norm{z}_{\H_{2}}&\leq c_{\Ao}\norm{\Aos z}_{\H_{1}}
\end{align*}
hold.
\end{lem}

\begin{rem}
\label{clrangecompembrem}
Note that $R(\Az)$ resp. $R(\Ao)$ is closed,
if e.g. $D(\cAz)\hookrightarrow\H_{0}$ resp. $D(\cAo)\hookrightarrow\H_{1}$ is compact.
In this case, the respective inverse operators, i.e.,
\begin{align*}
\cAz^{-1}&:R(\Az)\to R(\Azs),
&
\cAo^{-1}&:R(\Ao)\to R(\Aos),\\
(\cAzs)^{-1}&:R(\Azs)\to R(\Az),
&
(\cAos)^{-1}&:R(\Aos)\to R(\Ao),
\end{align*}
are compact.
\end{rem}

Observe
$D(\cAo)
=D(\Ao)\cap\ovl{R(\Aos)}
\subset D(\Ao)\cap N(\Azs)
\subset D(\Ao)\cap D(\Azs)$.
Utilizing the Helmholtz type decompositions of Lemma \ref{helmrefined} we immediately have: 

\begin{lem}
\label{compemblem}
The embeddings 
$D(\cAz)\hookrightarrow\H_{0}$, $D(\cAo)\hookrightarrow\H_{1}$,
and $N_{0,1}\hookrightarrow\H_{1}$ are compact, if and only if the embedding
$D(\Ao)\cap D(\Azs)\hookrightarrow\H_{1}$ is compact.
In this case $N_{0,1}$ has finite dimension.
\end{lem}

\begin{rem}
\label{complexone}
The assumptions in Lemma \ref{helmrefined} on $\Az$ and $\Ao$
are equivalent to the assumption that
$$\begin{CD}
D(\Az)\subset\H_{0} @> \Az >>
D(\Ao)\subset\H_{1} @> \Ao >>
\H_{2}
\end{CD}$$
is a closed Hilbert complex, meaning that the ranges are closed. 
As a result of the previous lemmas, the adjoint complex
$$\begin{CD}
\H_{0} @< \Azs <<
D(\Azs)\subset\H_{1} @< \Aos <<
D(\Aos)\subset\H_{2} .
\end{CD}$$
is a closed Hilbert complex as well. 
\end{rem}

We can summarize.

\begin{theo}
\label{fatbmaintheogen}
Let $\Az$, $\Ao$ be as introduced before, i.e., having the complex property $\Ao\Az=0$, 
i.e., $R(\Az)\subset N(\Ao)$. Moreover, let $D(\Ao)\cap D(\Azs)\hookrightarrow\H_{1}$ be compact.
Then the assertions of Lemma \ref{helmrefined} hold, $N_{0,1}$ is finite dimensional
and the corresponding inverse operators are continuous resp. compact.
Especially, all ranges are closed and the corresponding Friedrichs/Poincar\'e type estimates hold.
\end{theo}

A special situation is the following.

\begin{lem}
\label{Hilbertadjoint}
Let $\Az$, $\Ao$ be as introduced before 
with $R(\Az)=N(\Ao)$ and $R(\Ao)$ closed in $\H_{2}$. 
Then $R(\Azs)$ and $R(\Aos)$ are closed as well, and the simplified Helmholtz type decompositions
\begin{align*}
\H_{1}
&=R(\Az)\oplus_{\H_{1}}R(\Aos),
&
N_{0,1}
&=\{0\},\\
N(\Ao)
&=R(\Az)=R(\cAz),
&
N(\Azs)
&=R(\Aos)=R(\cAos),\\
D(\Ao)
&=R(\Az)\oplus_{\H_{1}}D(\cAo),
&
D(\Azs)
&=D(\cAzs)\oplus_{\H_{1}}R(\Aos),\\
D(\Ao)\cap D(\Azs)
&=D(\cAzs)\oplus_{\H_{1}}D(\cAo)
\end{align*}
are valid. Moreover, the respective inverse operators are continuous
and the corresponding Friedrichs/ Poincar\'e type estimates hold.
\end{lem}

\begin{rem}
\label{equiasshelmlem}
Note that $R(\Aos)=N(\Azs)$ and $R(\Azs)$ closed
are equivalent assumptions for Lemma \ref{Hilbertadjoint} to hold.
\end{rem}

\begin{lem}
\label{assduallem}
Let $\Az$, $\Ao$ be as introduced before with
the sequence property \eqref{sequenceprop}, i.e., $R(\Az)\subset N(\Ao)$.
If the embedding $D(\Ao)\cap D(\Azs)\hookrightarrow\H_{1}$ is compact
and $N_{0,1}=\{0\}$, then 
the assumptions of Lemma \ref{Hilbertadjoint} are satisfied.
\end{lem}

\begin{rem}
\label{complextwo}
The assumptions in Lemma \ref{Hilbertadjoint} on $\Az$ and $\Ao$
are equivalent to the assumption that
$$\begin{CD}
D(\Az)\subset\H_{0} @> \Az >>
D(\Ao)\subset\H_{1} @> \Ao >>
\H_{2}
\end{CD}$$
is a closed and exact Hilbert complex. 
By Lemma \ref{Hilbertadjoint} the adjoint complex
$$\begin{CD}
\H_{0} @< \Azs <<
D(\Azs)\subset\H_{1} @< \Aos <<
D(\Aos)\subset\H_{2} .
\end{CD}$$
is a closed and exact Hilbert complex as well. 
\end{rem}

Parts of Lemma \ref{Hilbertadjoint} hold also in the Banach space setting.
As a direct consequence of the closed range theorem
and the closed graph theorem the following abstract result holds.

\begin{lem}
\label{Banachadjoint}
Let $\X_0$, $\X_1$, $\X_2$ be Banach spaces
and suppose $\Az\in BL(\X_0,\X_1)$, $\Ao\in BL(\X_1,\X_2)$  
with $R(\Az)=N(\Ao)$ and that $R(\Ao)$ is closed in $\X_2$. 
Then $R(\Aza)$ is closed in $\X_0^{\prime}$ and $R(\Aoa)=N(\Aza)$.
Moreover, $(\Aoa)^{-1}\in BL\big(R(\Aoa),R(\Ao)^{\prime}\big)$.
\end{lem}

Note that in the latter context we consider the operators
\begin{align*}
\Ao:\X_{1}&\To R(\Ao),
&
\Aoa:R(\Ao)^{\prime}&\To R(\Aoa)
&
(\Aoa)^{-1}:R(\Aoa)\To R(\Ao)^{\prime},
\end{align*}
with $N(\Aoa)=R(\Ao)^{\circ}=\{0\}$.

\begin{rem}
\label{complexthree}
The conditions on $\Az$ and $\Ao$ in Lemma \ref{Banachadjoint} 
are identical to the assumption that
$$\begin{CD}
\X_0 @> \Az >>
\X_1 @> \Ao >>
\X_2
\end{CD}$$
is a closed and exact complex of Banach spaces. 
The consequences of Lemma \ref{Banachadjoint} can be rephrased as follows. 
The adjoint complex of Banach spaces
$$\begin{CD}
\X_0^{\prime} @< \Aza <<
\X_1^{\prime} @< \Aoa <<
\X_2^{\prime}
\end{CD}$$
is closed and exact as well.
\end{rem}

\begin{lem}
\mylabel{falem}
$(\Aoa)^{-1}\in BL\big(R(\Aoa),R(\Ao)^{\prime}\big)$ is equivalent to
\begin{align}
\mylabel{invcontmod}
\exists\,c_{\Aoa}>0\quad
\forall\,y^{\prime}\in R(\Ao)^{\prime}\qquad
\norm{y^{\prime}}_{R(\Ao)^{\prime}}\leq c_{\Aoa}\norm{\Aoa y^{\prime}}_{\X_{1}^{\prime}} .
\end{align}
For the best constant $c_{\Aoa}$, 
\eqref{invcontmod} is equivalent to the general inf-sup-condition
\begin{align}
\mylabel{infsupgenprime}
0<\frac{1}{c_{\Aoa}}
=\inf_{0\neq y^{\prime}\in R(\Ao)^{\prime}}\sup_{0\neq x\in\X_{1}}
\frac{\scp{y^{\prime}}{\Ao x}_{R(\Ao)^{\prime}}}{\norm{y^{\prime}}_{R(\Ao)^{\prime}}\norm{x}_{\X_{1}}}.
\end{align}
In the special case that $\X_{2}=\H_{2}$ is a Hilbert space the closed subspace $R(\Ao)$ is isometrically isomorphic to $R(\Ao)^{\prime}$ and we obtain the following form of the inf-sup-condition
\begin{align}
\mylabel{infsupgen}
0<\frac{1}{c_{\Aoa}}
=\inf_{0\neq y\in R(\Ao)}\sup_{0\neq x\in\X_{1}}
\frac{\scp{y}{\Ao x}_{\H_{2}}}{\norm{y}_{\H_{2}}\norm{x}_{\X_{1}}}.
\end{align}
\end{lem}

The results collected in this section are well-known in functional analysis. We refer to \cite{Arnold:2006} for a presentation of some results of this section
from a numerical analysis perspective.

\subsection{Sobolev Spaces}

Next we introduce our notations for several classes of Sobolev spaces 
on a bounded domain $\om\subset\rt$. 
Let $m\in\nz$. We denote by $\ltom$ and $\hmom$
the standard Lebesgue and Sobolev spaces and write $\hzom=\ltom$. 
Our notation of spaces will not indicate whether the elements are scalar functions or vector fields.
For the rotation and divergence we define the Sobolev spaces
$$\rom:=\setb{\vecv\in\ltom}{\rot \vecv\in\ltom},\quad
\dom:=\setb{\vecv\in\ltom}{\div\vecv\in\ltom}$$
with the respective graph norms, where $\rot$ and $\div$ have to be understood 
in the distributional or weak sense.
We introduce spaces with homogeneous boundary conditions in the weak sense naturally by
$$\hocom:=\H(\gradc,\om):=\ovl{\cicom}^{\hoom},\quad
\rcom:=\ovl{\cicom}^{\rom},\quad
\dcom:=\ovl{\cicom}^{\dom},$$
and
$$\hmcom:=\ovl{\cicom}^{\hmom},$$
i.e., as closures of test functions or fields under the respective graph norms,
which generalizes homogeneous scalar, tangential and normal boundary conditions, respectively.
We also introduce the well known dual spaces 
$$\hmmom:=\big(\hmcom\big)^{\prime}$$
with the standard dual or operator norm defined by
\begin{align*}
\normhmmom{u}
&:=\sup_{0\neq\varphi\in\hmcom}\frac{\scphmmom{u}{\varphi}}{\norm{\varphi}_{\hmcom}}
\qquad
\text{for} \ u\in\hmmom,
\end{align*}
where we recall the duality pairing 
$\scp{\,\cdot\,}{\,\cdot\,}_{\hmmom}$ in $\big(\hmmom,\hmcom\big)$. 
Moreover, we define with respective graph norms
\begin{align*}
\rhmmom
&:=\setb{\vecv\in\hmmom}{\rot \vecv\in\hmmom},
\\
\dhmmom
&:=\setb{\vecv\in\hmmom}{\div\vecv\in\hmmom}.
\end{align*}
A vanishing differential operator
will be indicated by a zero at the lower right corner of the operator, e.g.,
\begin{align*}
\rzom
&\;=\setb{\vecv\in\rom}{\rot \vecv=0},
&
\dczom
&\;=\setb{\vecv\in\dcom}{\div\vecv=0},\\
\rhmmzom
&\;=\setb{\vecv\in\rhmmom}{\rot \vecv=0},
&\dhmozom
&\;=\setb{\vecv\in\dhmoom}{\div\vecv=0}.
\end{align*}
Let us also introduce
$$\ltzom
:=\setb{u\in\ltom}{u\,\bot_{\ltom}\,\reals}
=\setb{u\in\ltom}{\int_{\om}u=0},$$
where $\bot_{\ltom}$ denotes orthogonality in $\ltom$.

\subsection{General Assumptions}

We will impose the following regularity and topology assumptions on our domain $\om$.

\begin{defi}
\mylabel{admdef}
Let $\om$ be an open subset of $\rt$ with boundary $\ga:=\p\om$.
We will call $\om$ 
\begin{itemize}
\item[\bf(i)] 
strong Lipschitz, if $\ga$ is locally a graph of 
a Lipschitz function $\psi:U\subset\reals^2\to\reals$,
\item[\bf(ii)] 
topologically trivial, if $\om$ is simply connected with connected boundary $\ga$.
\end{itemize}
\end{defi}

\begin{genasslem}
\mylabel{genasspaper}
From now on and throughout this paper 
it is assumed that $\om\subset\rt$ is a bounded strong Lipschitz domain.
\end{genasslem}

If the domain $\om$ has to be topologically trivial,
we will always indicate this in the respective result.
Note that several results will hold for arbitrary open subsets $\om$ of $\rt$.
All results are valid for bounded and topologically trivial strong Lipschitz domains $\om\subset\rt$.
Nevertheless, most of the results will remain true for bounded strong Lipschitz domains $\om\subset\rt$.

\subsection{Vector Analysis}

In this last part of the preliminary section we summarize and prove several results 
related to scalar and vector potentials of various smoothness,
corresponding Friedrichs/Poincar\'e type estimates,
and related Helmholtz decompositions of $\ltom$ and other Hilbert and Sobolev spaces.
This is a first application of the functional analysis toolbox Section \ref{fuanasec}
for the operators $\gradc$, $\rotc$, $\divc$, and their adjoints $-\div$, $\rot$, $-\grad$.
Although these are well known facts, we recall and collect them here,
as we will use later similar techniques to obtain related results
for the more complicated operators $\Gradgradc$, $\RotcS$, $\DivcT$,
and their adjoints $\divDivS$, $\symRotT$, $-\devGrad$. Let 
\begin{align*}
\Az:=\gradc:\hocom\subset\ltom&\To\ltom,\\
\Ao:=\rotc:\rcom\subset\ltom&\To\ltom,\\
\At:=\divc:\dcom\subset\ltom&\To\ltom.
\intertext{Then $\Az$, $\Ao$, and $\At$ are unbounded, densely defined, and closed linear operators with adjoints}
\Azs=\gradc{}^{*}=-\div:\dom\subset\ltom&\To\ltom,\\
\Aos=\rotc{}^{*}=\rot:\rom\subset\ltom&\To\ltom,\\
\Ats=\divc{}^{*}=-\grad:\hoom\subset\ltom&\To\ltom
\end{align*}
and the sequence or complex properties
\begin{align*}
R(\Az)=\gradc\,\hocom&\subset\rczom=N(\Ao),
&
R(\Aos)=\rot\rom&\subset\dzom=N(\Azs),\\
R(\Ao)=\rotc\,\rcom&\subset\dczom=N(\At),
&
R(\Ats)=\grad\hoom&\subset\rzom=N(\Aos)
\end{align*}
hold. Note $N(\Az)=\{0\}$ and $N(\Ats)=\reals$.
Moreover, the embeddings
\begin{align*}
D(\Ao)\cap D(\Azs)
&=\rcom\cap\dom\hookrightarrow\ltom,\\
D(\At)\cap D(\Aos)
&=\dcom\cap\rom\hookrightarrow\ltom
\end{align*}
are compact. The latter compact embeddings are called
Maxwell compactness properties or Weck's selection theorems.
The first proof for strong Lipschitz domains (uniform cone like domains) 
avoiding smoothness of $\ga$ was given by Weck in \cite{weckmax}.
Generally, Weck's selection theorems hold e.g. for weak Lipschitz domains, see \cite{picardcomimb},
or even for more general domains with $p$-cusps or antennas, see \cite{witschremmax,picardweckwitschxmas}.
See also \cite{webercompmax} for a different proof in the case of a strong Lipschitz domain. 
Weck's selection theorem for mixed boundary conditions 
has been proved in \cite{jochmanncompembmaxmixbc} for strong Lipschitz domains
and recently in \cite{bauerpaulyschomburgmcpweaklip} for weak Lipschitz domains.
Similar to Rellich's selection theorem, i.e.,
the compact embedding of $\hocom$ resp. $\hoom$ into $\ltom$,
it is crucial that the domain $\om$ is bounded.
Finally, the kernels
\begin{align*}
N(\Ao)\cap N(\Azs)
&=\rczom\cap\dzom=:\harmdom,\\
N(\At)\cap N(\Aos)
&=\dczom\cap\rzom=:\harmnom,
\end{align*}
are finite dimensional, as the unit balls are compact,
i.e., the spaces of Dirichlet resp. Neumann fields are finite dimensional.
More precisely, the dimension of the Dirichlet resp. Neumann fields
depends on the topology or cohomology of $\om$, i.e., equals the second resp. first Betti number,
see e.g. \cite{picardharmdiff,picardboundaryelectro}.
Especially we have
$$\harmdom=\{0\},\text{ if $\ga$ is connected,}\qquad
\harmnom=\{0\},\text{ if $\om$ is simply connected.}$$

\begin{rem}
\mylabel{admrem}
Our general assumption on $\om$ to be bounded and strong Lipschitz
ensures that Weck's selection theorems (and thus also Rellich's) hold.
The additional assumption that $\om$ is also topologically trivial
excludes the existence of non-trivial Dirichlet
or Neumann fields, as $\om$ is simply connected with a connected boundary $\ga$.
\end{rem}

By the results of the functional analysis toolbox Section \ref{fuanasec} 
we see that all ranges are closed with 
\begin{align*}
R(\Az)&=R(\cAz),
&
R(\Ao)&=R(\cAo),
&
R(\At)&=R(\cAt),\\
R(\Azs)&=R(\cAzs),
&
R(\Aos)&=R(\cAos),
&
R(\Ats)&=R(\cAts),
\end{align*}
i.e., the ranges 
\begin{align}
\nonumber
\gradc\,\hocom&,
&
\grad\hoom&=\grad\big(\hoom\cap\ltzom\big),\\
\label{rangesgradrotdiv}
\rotc\,\rcom&=\rotc\big(\rcom\cap\rot\rom\big),
&
\rot\rom&=\rot\big(\rom\cap\rotc\,\rcom\big),\\
\nonumber
\divc\,\dcom&=\divc\big(\dcom\cap\grad\hoom\big),
&
\div\dom&=\div\big(\dom\cap\gradc\,\hocom\big)
\end{align}
are closed, and the reduced operators are
\begin{align*}
\cAz=\gradc:\hocom\subset\ltom&\To\gradc\,\hocom,\\
\cAo=\rotc:\rcom\cap\rot\rom\subset\rot\rom&\To\rot\rcom,\\
\cAt=\divc:\dcom\cap\grad\hoom\subset\grad\hoom&\To\ltzom,\\
\cAzs=-\div:\dom\cap\grad\hocom\subset\grad\hocom&\To\ltom,\\
\cAos=\rot:\rom\cap\rotc\,\rcom\subset\rotc\,\rcom&\To\rot\rom,\\
\cAts=-\grad:\hoom\cap\ltzom\subset\ltzom&\To\grad\hoom.
\end{align*}
Moreover, we have the following well known Helmholtz decompositions of $\lt$-vector fields
into irrotational and solenoidal vector fields,
corresponding Friedrichs/Poincar\'e type estimates and continuous 
or compact inverse operators.

\begin{lem}
\mylabel{genhelmlt}
The Helmholtz decompositions
\begin{align*}
\ltom&=\divc\,\dcom\oplus_{\ltom}\reals,\quad
\divc\,\dcom=\ltzom,\\
\ltom&=\div\dom,\\
\ltom&=\gradc\,\hocom\oplus_{\ltom}\dzom\\
&=\rczom\oplus_{\ltom}\rot\rom\\
&=\gradc\,\hocom\oplus_{\ltom}\harmdom\oplus_{\ltom}\rot\rom,\\
\ltom&=\grad\hoom\oplus_{\ltom}\dczom\\
&=\rzom\oplus_{\ltom}\rotc\,\rcom\\
&=\grad\hoom\oplus_{\ltom}\harmnom\oplus_{\ltom}\rotc\,\rcom
\end{align*}
hold. Moreover, \eqref{rangesgradrotdiv} is true for the respective ranges
and the ``better'' potentials in \eqref{rangesgradrotdiv} are uniquely determined
and depend continuously in the right hand sides.
If $\ga$ is connected, it holds $\harmdom=\{0\}$ and, e.g.,
\begin{align*}
\ltom
&=\rczom
\oplus\dzom,\\
\rczom
&=\gradc\,\hocom,
&
\dzom
&=\rot\rom=\rot\big(\rom\cap\dczom\big).
\intertext{If $\om$ is simply connected, it holds $\harmnom=\{0\}$ and, e.g.,}
\ltom
&=\rzom
\oplus\dczom,\\
\rzom
&=\grad\hoom,
&
\dczom
&=\rotc\,\rcom=\rotc\big(\rcom\cap\dzom\big).
\end{align*}
\end{lem}

\begin{lem}
\label{poincaregradrotdiv}
The following Friedrichs/Poincar\'e type estimates hold.
There exist positive constants $c_{\mathsf g}$, $c_{\mathsf r}$, $c_{\mathsf d}$, such that
\begin{align*}
\forall\,u&\in\hocom
&
\normltom{u}&\leq c_{\mathsf g}\,\normltom{\grad u},\\
\forall\,\vecv&\in\dom\cap\gradc\,\hocom
&
\normltom{\vecv}&\leq c_{\mathsf g}\,\normltom{\div\vecv},\\
\forall\,\vecv&\in\rcom\cap\rot\rom
&
\normltom{\vecv}&\leq c_{\mathsf r}\,\normltom{\rot \vecv},\\
\forall\,\vecv&\in\rom\cap\rotc\,\rcom
&
\normltom{\vecv}&\leq c_{\mathsf r}\,\normltom{\rot \vecv},\\
\forall\,\vecv&\in\dcom\cap\grad\hoom
&
\normltom{\vecv}&\leq c_{\mathsf d}\,\normltom{\div\vecv},\\
\forall\,u&\in\hoom\cap\ltzom
&
\normltom{u}&\leq c_{\mathsf d}\,\normltom{\grad u}.
\end{align*}
Moreover, the reduced versions of the operators 
$$\gradc,\quad
\rotc,\quad
\divc,\quad
\grad,\quad
\rot,\quad
\div$$
have continuous resp. compact inverse operators 
\begin{align*}
\gradc{}^{-1}:
\gradc\,\hocom&\To\hocom,
&
\gradc{}^{-1}:
\gradc\,\hocom&\To\ltom,\\
\div^{-1}:
\ltom&\To\dom\cap\gradc\,\hocom,
&
\div^{-1}:
\ltom&\To\gradc\,\hocom\subset\ltom,\\
\rotc{}^{-1}:
\rotc\,\rcom&\To\rcom\cap\rot\rom,
&
\rotc{}^{-1}:
\rotc\,\rcom&\To\rot\rom\subset\ltom,\\
\rot^{-1}:
\rot\rom&\To\rom\cap\rotc\,\rcom,
&
\rot^{-1}:
\rot\rom&\To\rotc\,\rcom\subset\ltom,\\
\divc{}^{-1}:
\ltzom&\To\dcom\cap\grad\hoom,
&
\divc{}^{-1}:
\ltzom&\To\grad\hoom\subset\ltom,\\
\grad^{-1}:
\grad\hoom&\To\hoom\cap\ltzom,
&
\grad^{-1}:
\grad\hoom&\To\ltzom,
\end{align*}
with norms $(1+c_{\mathsf g}^2)^{\oh}$, $(1+c_{\mathsf r}^2)^{\oh}$, $(1+c_{\mathsf d}^2)^{\oh}$
resp. $c_{\mathsf g}$, $c_{\mathsf r}$, $c_{\mathsf d}$.
In other words, the operators
\begin{align*}
\gradc:\hocom&\To\gradc\,\hocom,
&
\div:\dom\cap\gradc\,\hocom&\To\ltom,\\
u&\longmapsto\grad u
&
\vecv&\longmapsto\div\vecv\\
\rotc:\rcom\cap\rot\rom&\To\rotc\,\rcom,
&
\rot:\rom\cap\rotc\,\rcom&\To\rot\rom,\\
\vecv&\longmapsto\rot \vecv
&
\vecv&\longmapsto\rot \vecv\\
\divc:\dcom\cap\grad\hoom&\To\ltzom,
&
\grad:\hoom\cap\ltzom&\To\grad\hoom,\\
\vecv&\longmapsto\div\vecv
&
u&\longmapsto\grad u
\end{align*}
are topological isomorphisms. If $\om$ is topologically trivial, then
\begin{align}
\nonumber
\gradc:\hocom&\To\rczom,
&
\div:\dom\cap\rczom&\To\ltom,\\
\nonumber
u&\longmapsto\grad u
&
\vecv&\longmapsto\div\vecv\\
\label{topiso}
\rotc:\rcom\cap\dzom&\To\dczom,
&
\rot:\rom\cap\dczom&\To\dzom,\\
\nonumber
\vecv&\longmapsto\rot \vecv
&
\vecv&\longmapsto\rot \vecv\\
\nonumber
\divc:\dcom\cap\rzom&\To\ltzom,
&
\grad:\hoom\cap\ltzom&\To\rzom,\\
\nonumber
\vecv&\longmapsto\div\vecv
&
u&\longmapsto\grad u
\end{align}
are topological isomorphisms.
\end{lem}

\begin{rem} 
\label{Maxconsts}
Recently it has been shown in \cite{paulymaxconst0,paulymaxconst1,paulymaxconst2}, 
that for bounded and convex $\om\subset\rt$ it holds
$$c_{\mathsf r}\leq c_{\mathsf d}\leq\frac{\diam\om}{\pi},$$
i.e., the Maxwell constant $c_{\mathsf r}$ can be estimates from above by 
the Poincar\'e constant $c_{\mathsf d}$.
\end{rem}

\begin{rem} 
\label{complexgradrotdiv}
Some of the previous results can be formulated equivalently in terms of complexes: The sequence
$$\begin{CD}
\{0\} 
@> 0 >>
\hocom @> \gradc >>
\rcom @> \rotc >>
\dcom @> \divc >>
\ltom @> \pi_{\reals} >>
\reals
\end{CD}$$
and thus also its dual or adjoint sequence
$$\begin{CD}
\{0\} 
@< 0 <<
\ltom @< -\div <<
\dom @< \rot << 
\rom @< -\grad <<
\hoom @< \iota_{\reals} <<
\reals
\end{CD}$$
are closed Hilbert complexes. 
Here $\pi_{\reals}:\ltom\to\reals$ 
denotes the orthogonal projector onto $\reals$
with adjoint $\pi_{\reals}^{*}=\iota_{\reals}:\reals\to\ltom$,
the canonical embedding. If $\om$ is additionally topologically trivial,
then the complexes are also exact. These complexes are widely known as de Rham complexes.
\end{rem}

Let $\om$ be additionally topologically trivial.
For irrotational vector fields in $\hmcom$ resp. $\hmom$ we have smooth potentials,
which follows immediately by $\rczom=\gradc\,\hocom$ resp. $\rzom=\grad\hoom$ from the previous lemma.

\begin{lem}
\mylabel{napotlem}
Let $\om$ be additionally topologically trivial and $m\in\nz$. Then 
$$\hmcom\cap\rczom=\gradc\,\hmpocom,\qquad
\hmom\cap\rzom=\grad\hmpoom$$
hold with linear and continuous potential operators $\Pot_{\gradc}$, $\Pot_{\grad}$. 
\end{lem}

So, for each $\vecv \in \hmcom\cap\rczom$, we have
$\vecv = \gradc\,u$ for the potential $u = \Pot_{\gradc} \vecv \in \hmpocom$ and, analogously,
for each $\vecv \in \hmom\cap\rom$, it holds
$\vecv = \grad u$ for the potential $u = \Pot_{\grad} \vecv \in \hmpoom$.
Note that the potential in $\hmpoom$ is uniquely determined only up to a constant.

For solenoidal vector fields in $\hmcom$ resp. $\hmom$ we have smooth potentials, too.

\begin{lem}
\mylabel{rotpotlem}
Let $\om$ be additionally topologically trivial and $m\in\nz$. Then 
$$\hmcom\cap\dczom=\rotc\,\hmpocom,\qquad
\hmom\cap\dzom=\rot\hmpoom$$
hold with linear and continuous potential operators $\Pot_{\rotc}$, $\Pot_{\rot}$.
\end{lem}

For a proof see, e.g., \cite[Corollary 4.7]{costabelmcintoshgenbogovskii}
or with slight modifications the generalized lifting lemma 
\cite[Corollary 5.4]{hiptmairlizouuniext} for the case $d=3$, $k=m$, $l=2$.
Moreover, the potential in $\hmpocom$ resp. $\hmpoom$ is no longer uniquely determined.

For the divergence operator we have the following result.

\begin{lem}
\mylabel{divpotlem}
Let $m\in\nz$. Then 
$$\hmcom\cap\ltzom=\divc\,\hmpocom,\qquad
\hmom=\div\hmpoom$$
hold with linear and continuous potential operators $\Pot_{\divc}$, $\Pot_{\div}$.
\end{lem}

Again, the potential in $\hmpocom$ resp. $\hmpoom$ is no longer uniquely determined.
Also Lemma \ref{napotlem} resp. Lemma \ref{divpotlem}
has been proved in \cite[Corollary 4.7(b)]{costabelmcintoshgenbogovskii}
and in \cite[Corollary 5.4]{hiptmairlizouuniext} for the case 
$d=3$, $k=m$, $l=1$ resp. $d=3$, $k=m$, $l=3$.

\begin{rem}
\mylabel{bogovskiirem}
Lemma \ref{divpotlem}, which shows a classical result on the solvability
and on the properties of the solution operator of the divergence equation,
is an important tool in fluid dynamics, i.e., in the theory
of Stokes or Navier-Stokes equations.
The potential operator is often called Bogovski\u{i} operator, 
see \cite{bogovskii1979,bogovskii1980} for the original works
and also \cite[p. 179, Theorem III.3.3]{galdibook1}, \cite[Lemma 2.1.1]{sohrbook}.
Moreover, there are also versions of Lemma \ref{napotlem} and Lemma \ref{rotpotlem},
if $\om$ is not topologically trivial, which we will not need in the paper at hand.
\end{rem}

\begin{rem}
\mylabel{narotpotlemrem}
A closer inspection of Lemma \ref{napotlem} and Lemma \ref{rotpotlem} and their proofs shows,
that these results extend to general topologies as well. More precisely we have:
\begin{itemize}
\item[\bf(i)]
It holds
\begin{align*}
\hmcom\cap\gradc\,\hocom=\hmcom\cap\rczom\cap\harmdom^{\bot}&=\gradc\,\hmpocom,\\
\hmom\cap\grad\hoom=\hmom\cap\rzom\cap\harmnom^{\bot}&=\grad\hmpoom
\end{align*}
with linear and continuous potential operators $\Pot_{\gradc}$, $\Pot_{\grad}$. 
\item[\bf(ii)]
It holds
\begin{align*}
\hmcom\cap\rotc\,\rcom=\hmcom\cap\dczom\cap\harmnom^{\bot}&=\rotc\,\hmpocom,\\
\hmom\cap\rot\rom=\hmom\cap\dzom\cap\harmdom^{\bot}&=\rot\hmpoom
\end{align*}
with linear and continuous potential operators $\Pot_{\rotc}$, $\Pot_{\rot}$.
\end{itemize}
\end{rem}

Using the latter three results and Lemma \ref{Banachadjoint}, 
irrotational and solenoidal vector fields in $\hmmom$ can be characterized.

\begin{cor}
\mylabel{nahmorangecor}
Let $\om$ be additionally topologically trivial and $m\in\n$. Then 
$$\rhmmzom=\grad\hmmpoom=\grad\big(\hmmocom\cap\ltzom\big)^{\prime}$$
is closed in $\hmmom$ with continuous inverse, i.e., 
$\grad^{-1}\in BL\big(\rhmmzom,(\hmmocom\cap\ltzom)^{\prime}\big)$.
Especially for $m=1$, 
$$\rhmozom=\grad\ltom=\grad\ltzom$$
is closed in $\hmoom$ with continuous inverse
$\grad^{-1}\in BL\big(\rhmozom,\ltzom\big)$
and uniquely determined potential in $\ltzom$. Moreover,
$$\exists\,c_{{\mathsf g},-1}>0\quad
\forall\,u\in\ltzom\qquad
\normltom{u}\leq c_{{\mathsf g},-1}\normhmoom{\grad u}\leq\sqrt{3}\,c_{{\mathsf g},-1}\normltom{u}$$
and the inf-sup-condition
\begin{align*}
0<\frac{1}{c_{{\mathsf g},-1}}
=\inf_{0\neq u\in\ltzom}
\frac{\norm{\grad u}_{\hmoom}}{\normltom{u}}
=\inf_{0\neq u\in\ltzom}\sup_{0\neq \vecv\in\hocom}
\frac{\scp{u}{\div\vecv}_{\ltom}}{\normltom{u}\norm{\Grad\vecv}_{\ltom}}.
\end{align*}
holds.
\end{cor}

\begin{proof}
Let $\X_0:=\hmpocom$, $\X_1:=\hmcom$, $\X_2:=\hmmocom$ and 
$$\Az:=\rotc:\hmpocom\to\hmcom,\quad
\Ao:=-\divc:\hmcom\to\hmmocom.$$
These linear operators are bounded, 
$R(\Az)=\rotc\,\hmpocom=\hmcom\cap\dczom=N(\Ao)$ by Lemma \ref{rotpotlem}, 
and $R(\Ao)=\divc\hmcom=\hmmocom\cap\ltzom$ by Lemma \ref{divpotlem}. 
Therefore, $R(\Ao)$ is closed.
For the adjoint operators we get
$$\Aza=\rot=\rotc{}^{\prime}:\hmmom\to\hmmmoom,\quad
\Aoa=\grad=-\divc{}^{\prime}:\hmmpoom\to\hmmom$$
and obtain from Lemma \ref{Banachadjoint} that
$$\rhmmzom=N(\Aza)=R(\Aoa)=\grad\hmmpoom$$
is closed and
$$\grad^{-1}=(\Aoa)^{-1}\in 
BL\big(R(\Aoa),R(\Ao)^{\prime}\big) 
=BL\big(\rhmmzom,(\hmmocom\cap\ltzom)^{\prime}\big),$$
which completes the proof for general $m$.
If $m=1$, we get the assertions about the Friedrichs/Poincar\'e/ Ne\u{c}as inequality
and inf-sup-condition by Lemma \ref{falem}, i.e., \eqref{invcontmod} and \eqref{infsupgen}.
\end{proof}

\begin{cor}
\mylabel{rothmorangecor}
Let $\om$ be additionally topologically trivial and $m\in\n$. Then 
$$\dhmmzom=\rot\hmmpoom=\rot\big(\hmmocom\cap\dczom\big)^{\prime}$$
is closed in $\hmmom$ with continuous inverse, i.e.,
$\rot^{-1}\in BL\big(\dhmmzom,(\hmmocom\cap\dczom)^{\prime}\big)$.
Especially for $m=1$, 
$$\dhmozom=\rot\ltom=\rot\dczom$$
is closed in $\hmoom$ with continuous inverse
$\rot^{-1}\in BL\big(\dhmozom,\dczom\big)$
and uniquely determined potential in $\dczom$. Moreover,
$$\exists\,c_{{\mathsf r},-1}>0\quad
\forall\,\vecv\in\dczom\qquad
\normltom{\vecv}\leq c_{{\mathsf r},-1}\normhmoom{\rot \vecv}\leq\sqrt{2}\,c_{{\mathsf r},-1}\normltom{\vecv}$$
and the inf-sup-condition
\begin{align*}
0<\frac{1}{c_{{\mathsf r},-1}}
=\inf_{0\neq \vecv\in\dczom}
\frac{\norm{\rot \vecv}_{\hmoom}}{\normltom{\vecv}}
=\inf_{0\neq\vecv\in\dczom}\sup_{0\neq \vecw\in\hocom}
\frac{\scp{\vecv}{\rot\vecw}_{\ltom}}{\normltom{\vecv}\norm{\Grad\vecw}_{\ltom}}.
\end{align*}
holds.
\end{cor}

\begin{proof}
Let $\X_0:=\hmpocom$, $\X_1:=\hmcom$, $\X_2:=\hmmocom$ and 
$$\Az:=\gradc:\hmpocom\to\hmcom,\quad
\Ao:=\rotc:\hmcom\to\hmmocom.$$
These linear operators are bounded, $R(\Az)=\gradc\,\hmpocom=\hmcom\cap\rczom=N(\Ao)$ 
by Lemma \ref{napotlem}, and $R(\Ao)=\rot\hmcom=\hmmocom\cap\dczom$ by Lemma \ref{rotpotlem}. 
Therefore, $R(\Ao)$ is closed.
For the adjoint operators we get
$$\Aza=-\div=\gradc{}^{\prime}:\hmmom\to\hmmmoom,\quad
\Aoa=\rot=\rotc{}^{\prime}:\hmmpoom\to\hmmom$$
and obtain from Lemma \ref{Banachadjoint} that
$$\dhmmzom=N(\Aza)=R(\Aoa)=\rot\hmmpoom$$
is closed and
$$\rot^{-1}=(\Aoa)^{-1}\in
BL\big(R(\Aoa),R(\Ao)^{\prime}\big)
=BL\big(\dhmmzom,(\hmmocom\cap\dczom)^{\prime}\big),$$
which completes the proof for general $m$.
If $m=1$, we get the assertions about the Friedrichs/Poincar\'e/ Ne\u{c}as inequality
and inf-sup-condition by Lemma \ref{falem}, i.e., \eqref{invcontmod} and \eqref{infsupgen}.
\end{proof}

Let us present the corresponding result for the divergence as well.

\begin{cor}
\mylabel{divhmorangecor}
Let $\om$ be additionally topologically trivial and $m\in\n$. Then 
$$\hmmom=\div\hmmpoom=\div\big(\hmmocom\cap\rczom\big)^{\prime}$$
(is closed in $\hmmom$) with continuous inverse, i.e.,
$\div^{-1}\in BL\big(\hmmom,(\hmmocom\cap\rczom)^{\prime}\big)$.
Especially for $m=1$, 
$$\hmoom=\div\ltom=\div\rczom$$
(is closed in $\hmoom$) with continuous inverse
$\div^{-1}\in BL\big(\hmoom,\rczom\big)$
and uniquely determined potential in $\rczom$. Moreover,
$$\exists\,c_{{\mathsf d},-1}>0\quad
\forall\,\vecv\in\rczom\qquad
\normltom{\vecv}\leq c_{{\mathsf d},-1}\normhmoom{\div\vecv}\leq c_{{\mathsf d},-1}\,\normltom{\vecv}$$
and the inf-sup-condition
\begin{align*}
0<\frac{1}{c_{{\mathsf d},-1}}
=\inf_{0\neq \vecv\in\rczom}
\frac{\norm{\div\vecv}_{\hmoom}}{\normltom{\vecv}}
=\inf_{0\neq \vecv\in\dczom}\sup_{0\neq u\in\hocom}
\frac{\scp{\vecv}{\grad u}_{\ltom}}{\normltom{\vecv}\norm{\grad u}_{\ltom}}.
\end{align*}
holds.
\end{cor}

\begin{proof}
Let $\X_1:=\hmcom$, $\X_2:=\hmmocom$ and $\Ao:=-\gradc:\hmcom\to\hmmocom$.
$\Ao$ is linear and bounded with $R(\Ao)=\grad\hmcom=\hmmocom\cap\rczom$ by Lemma \ref{napotlem}. 
Therefore, $R(\Ao)$ is closed. 
The adjoint is $\Aoa=\div=-\gradc{}^{\prime}:\hmmpoom\to\hmmom$
with closed range $R(\Aoa)=\div\hmmpoom$ by the closed range theorem.
Moreover, $N(\Ao)=\{0\}$. Hence $\Aoa$ is surjective as $\Ao$ is injective, i.e.,
$$\hmmom=N(\Ao)^{\circ}=R(\Aoa)=\div\hmmpoom.$$
As $\Ao$ is also surjective onto its range, $\Aoa=\div:\hmmpoom\to R(\Aoa)$ is bijective.
By the bounded inverse theorem we get
$$\div^{-1}=(\Aoa)^{-1}\in
BL\big(R(\Aoa),R(\Ao)^{\prime}\big)
=BL\big(\hmmom,(\hmmocom\cap\rczom)^{\prime}\big),$$
which completes the proof for general $m$.
If $m=1$, we get the assertions about the Friedrichs/Poincar\'e/ Ne\u{c}as inequality
and inf-sup-condition by Lemma \ref{falem}, i.e., \eqref{invcontmod} and \eqref{infsupgen}.
\end{proof}

\begin{rem} 
\label{complexgradrotdivhmc}
The results of the latter three lemmas and corollaries can be formulated equivalently in terms of complexes: 
Let $\om$ be additionally topologically trivial.
Then the sequence
$$\begin{CD}
\hmpocom @> \gradc >>
\hmcom @> \rotc >>
\hmmocom @> \divc >>
\hmmtcom 
\end{CD}$$
and thus also its dual or adjoint sequence
$$\begin{CD}
\hmmmoom @< -\div <<
\hmmom @< \rot << 
\hmmpoom @< -\grad <<
\hgen{}{-m+2}{}(\om)
\end{CD}$$
are closed and exact Banach complexes. 
\end{rem}

\section{The $\Gradgrad$- and $\divDiv$-Complexes}
\mylabel{kernelsec}

We will use the following standard notations from linear algebra. 
For vectors $a,b\in\rt$ and matrices $A,B\in\rttt$ 
the expressions $a\cdot b$ and $A:B$ denote the inner product of vectors 
and the Frobenius inner product of matrices, respectively. 
For a vector $a\in\rt$ with components $a_i$ for $i=1,2,3$  
the matrix $\spn a\in\rttt$ is defined by
$$\spn a 
=\begin{bmatrix}
0 & -a_3 & a_2 \\
a_3 & 0 & -a_1 \\
-a_{2} & a_1 & 0
\end{bmatrix}.$$ 
Observe that $(\spn a)\,b=a\times b$ for $a,b\in\rt$, where $a\times b$ denotes the exterior product of vectors.
The exterior product $a\times B$ of a vector $a\in\rt$ and a matrix $B\in\rttt$ 
is defined as the matrix which is obtained by applying the exterior product row-wise.
Note that $\spn$ is a bijective mapping from $\rt$ to the set of skew-symmetric matrices in $\rttt$ with the inverse mapping $\spn^{-1}$.
In addition to $\sym A$ and $\skw A$ for the symmetric part and the skew-symmetric part of a matrix $A$, 
we use $\dev A$ and $\tr A$ for denoting the deviatoric part and the trace of a matrix $A$.
Finally, the set of symmetric matrices in $\rttt$ is denoted by $\bbS$, 
the set of matrices in $\rttt$ with vanishing trace is denoted by $\bbT$.

We need several spaces of tensor fields. The spaces
$$\Cicom,\quad\Ltom,\quad\Hoom,\quad\Hocom,\quad\Dom,\quad\Dcom,\quad\Rczom,\quad\dots$$
are introduced for tensor fields as well, such that all rows belong to the corresponding spaces of vector fields 
$\cicom$, $\ltom$, $\hoom$, $\hocom$, $\dom$, $\dcom$, $\rczom$, $\dots$, respectively.
Additionally, we will need spaces allowing 
for a deviatoric gradient, a symmetric rotation, and a double divergence, i.e.,
\begin{align*}
\devGom
&:=\setb{\vecv\in\ltom}{\devGrad\vecv\in\Ltom},
&
\devGzom
&:=\setb{\vecv\in\ltom}{\devGrad\vecv=0},\\
\symRom
&:=\setb{\bfE\in\Ltom}{\symRot\bfE\in\Ltom},
&
\symRzom
&:=\setb{\bfE\in\Ltom}{\symRot\bfE=0},\\
\dDom
&:=\setb{\bfM\in\Ltom}{\divDiv\bfM\in\ltom},
&
\dDzom
&:=\setb{\bfM\in\Ltom}{\divDiv\bfM=0}.
\end{align*}
Moreover, we introduce various spaces of symmetric 
tensor fields without prescribed boundary conditions, i.e.,
\begin{align*}
\ltomS
:=\set{\bfM\in\Ltom}{\bfM^{\top}=\bfM},\quad
\dDomS
:=\dDom\cap\ltomS,\quad\dots,
\end{align*}
and with homogeneous boundary conditions as closures of symmetric test tensor fields, i.e.,
\begin{align*}
\HocomS
:=\ovl{\Cicom\cap\ltomS}^{\Hoom},\quad
\RcomS
:=\ovl{\Cicom\cap\ltomS}^{\Rom},\quad\dots,
\end{align*}
as well as spaces of tensor fields with vanishing trace 
and without prescribed boundary conditions, i.e.,
\begin{gather*}
\ltomT
:=\set{\bfE\in\Ltom}{\tr\bfE=0},\quad
\HoomT
:=\Hoom\cap\ltomT,\quad\dots,
\end{gather*}
and with homogeneous boundary conditions as closures of trace-free test tensor fields, i.e.,
\begin{align*}
\HocomT
:=\ovl{\Cicom\cap\ltomT}^{\Hoom},\quad
\DcomT
:=\ovl{\Cicom\cap\ltomT}^{\Dom},\quad\dots.
\end{align*}
We note
$$\HocomS=\sym\Hocom=\Hocom\cap\ltomS,\quad
\HocomT=\dev\Hocom=\Hocom\cap\ltomT,$$
but generally only
$$\RcomS\subset\Rcom\cap\ltomS,\quad
\DcomT\subset\Dcom\cap\ltomT,\quad\dots.$$
Let us also mention that  trivially
$$\devGrad\devGom\subset\ltomT,\quad
\symRot\symRom\subset\ltomS$$
hold. This can be seen as follows.
Pick $\vecv\in\devGom$ with $\bfE:=\devGrad\vecv$
and $\bfN\in\symRom$ with $\bfM:=\symRot\bfN$. 
Then for all $\varphi\in\cicom$ and $\bfPhi\in\Cicom$
\begin{align*}
\scpltom{\tr\bfE}{\varphi}
&=\scpLtom{\bfE}{\varphi\,\bfI}
=-\scpLtom{\vecv}{\Div\dev\varphi\,\bfI}
=0,\\
\scpLtom{\skw\bfM}{\bfPhi}
&=\scpLtom{\bfM}{\skw\bfPhi}
=\scpLtom{\bfN}{\Rot\sym\skw\bfPhi}
=0.
\end{align*}
Before we proceed we need a few technical lemmas.

\begin{lem}
\label{GraddevGrad}
For any distributional vector field $\vecv$ it holds for $i,j,k=1,\dots,3$
$$\p_{k}(\Grad\vecv)_{ij}
=\begin{cases}
\p_{k}(\devGrad\vecv)_{ij}
&\text{, if }i\neq j,\\
\p_{j}(\devGrad\vecv)_{ik}
&\text{, if }i\neq k,\\
\displaystyle\frac{3}{2}\p_{i}(\devGrad\vecv)_{ii}
+\frac{1}{2}\sum_{l\neq i}\p_{l}(\devGrad\vecv)_{li}
&\text{, if }i=j=k.
\end{cases}$$
\end{lem}

\begin{proof}
Let $\phi\in\cic(\rt)$ be a vector field. We want to express the second derivatives of $\phi$
by the derivatives of the deviatoric part of the Jacobian, i.e., of $\devGrad\phi$.
Recall that we have $\dev\bfE=\bfE-\frac{1}{3}(\tr\bfE)\,\bfI$ for a tensor $\bfE$.
Hence $\devGrad\phi$ coincides with $\Grad\phi$ outside the diagonal entries, i.e., we observe
$(\Grad\phi)_{ij}=(\devGrad\phi)_{ij}$ for $i\neq j$. Hence, looking at second derivatives,
we see immediately 
\begin{align*}
\p_{k}\p_{j}\phi_{i}
&=\p_{k}(\Grad\phi)_{ij}=\p_{k}(\devGrad\phi)_{ij}
&&\text{for }i\neq j,\\
\p_{k}\p_{j}\phi_{i}
&=\p_{j}\p_{k}\phi_{i}=\p_{j}(\Grad\phi)_{ik}=\p_{j}(\devGrad\phi)_{ik}
&&\text{for }i\neq k.
\end{align*}
Thus it remains to represent $\p_{i}^2\phi_{i}$ by the derivatives of $\devGrad\phi$. By
\begin{align*}
\p_{i}^2\phi_{i}
=\p_{i}(\Grad\phi)_{ii}
=\p_{i}(\devGrad\phi)_{ii}+\frac{1}{3}\p_{i}\div\phi
\end{align*}
we get
\begin{align*}
\frac{2}{3}\p_{i}^2\phi_{i}
=\p_{i}(\devGrad\phi)_{ii}+\frac{1}{3}\sum_{l\neq i}\p_{i}\p_{l}\phi_{l}
=\p_{i}(\devGrad\phi)_{ii}+\frac{1}{3}\sum_{l\neq i}\p_{l}(\devGrad\phi)_{li},
\end{align*}
yielding the stated result for test vector fields.
Testing extends the formulas to distributions,
which finishes the proof.
\end{proof}

We note that the latter trick is similar to the well known fact
that second derivatives of a vector field can always be written 
as derivatives of the symmetric gradient of the vector field,
leading by Ne\u{c}as estimate to Korn's second and first inequalities.
We will now do the same for the operator $\devGrad$.

\begin{lem}
\label{kerneldevGrad}
It holds:
\begin{itemize}
\item[\bf(i)]
There exists $c>0$, such that for all vector fields $\vecv\in\hoom$
$$\normltom{\Grad\vecv}
\leq c\,\big(\normltom{\vecv}+\normltom{\devGrad\vecv}\big).$$
\item[\bf(ii)]
$\devGom=\hoom$.
\item[\bf(iii)]
For $\devGrad:\devGom\subset\ltom\To\ltomT$ it holds 
$$D(\devGrad)=\devGom=\hoom,$$
and the kernel of $\devGrad$ equals the space of (global) shape functions 
of the lowest order Raviart-Thomas elements, i.e.,
$$N(\devGrad)
=\devGzom
=\RTz 
:=\set{P}{P(x)=a\,x+b,\ a\in\reals,\ b\in\rt},$$
which dimension is $\dim\RTz=4$.
\item[\bf(iv)]
There exists $c>0$, such that for all vector fields $\vecv\in\hoom\cap\RTz^{\bot_{\ltom}}$
$$\normhoom{\vecv}
\leq c\,\normltom{\devGrad\vecv}.$$
\end{itemize}
\end{lem}

\begin{proof}
Let $\vecv\in\hoom$. By the latter lemma and Ne\u{c}as estimate, i.e., 
$$\exists\,c>0\quad\forall\,u\in\ltom\qquad
c\,\normltom{u}\leq\normhmoom{\grad u}+\normhmoom{u}\leq(\sqrt{3}+1)\normltom{u},$$
we get 
\begin{align*}
\norm{\Grad\vecv}_{\Ltom}
&\leq c\,\big(\sum_{k=1}^{3}\norm{\p_{k}\Grad\vecv}_{\Hmoom}+\norm{\Grad\vecv}_{\Hmoom}\big)\\
&\leq c\,\big(\sum_{k=1}^{3}\norm{\p_{k}\devGrad\vecv}_{\Hmoom}+\norm{\Grad\vecv}_{\Hmoom}\big)\\
&\leq c\,\big(\norm{\devGrad\vecv}_{\Ltom}+\norm{\vecv}_{\Ltom}\big),
\end{align*}
which shows (i). As $\om$ has the segment property and by standard mollification
we obtain that restrictions of $\cic(\rt)$-vector fields are dense in $\devGom$.
Especially $\hoom$ is dense in $\devGom$.
Let $\vecv\in\devGom$ and $(\vecv_{n})\subset\hoom$ with $\vecv_{n}\to \vecv$ in $\devGom$. 
By (i) $(\vecv_{n})$ is a Cauchy sequence in $\hoom$ converging to $\vecv$ in $\hoom$,
which proves $\vecv\in\hoom$ and hence (ii).
For $P\in\RTz$ it holds $\devGrad P=a\dev\bfI=0$.
Let $\devGrad\vecv=0$ for some vector field $\vecv\in\devGom=\hoom$. 
By Lemma \ref{GraddevGrad} we get $\p_{k}\Grad\vecv=0$ for all $k=1,\dots,3$,
and therefore $\vecv(x) = A \, x + b$ for some matrix $A \in \rttt$ and vector $b \in \rt$. 
Then $0=\devGrad\vecv=\dev A$, if and only if $A=\frac{1}{3}(\tr A)\,\bfI$,
which shows (iii). If (iv) was wrong, 
there exists a sequence $(\vecv_{n})\subset\hoom\cap\RTz^{\bot_{\ltom}}$
with $\normhoom{\vecv_{n}}=1$ and $\devGrad\vecv_{n}\to0$. As $(\vecv_{n})$ is bounded in $\hoom$,
by Rellich's selection theorem there exists a subsequence, again denoted by $(\vecv_{n})$,
and some $\vecv\in\ltom$ with $\vecv_{n}\to \vecv$ in $\ltom$. By (i), $(\vecv_{n})$ is a Cauchy sequence in $\hoom$.
Hence $\vecv_{n}\to \vecv$ in $\hoom$ and $\vecv\in\hoom\cap\RTz^{\bot_{\ltom}}$.
As $0\ot\devGrad\vecv_{n}\to\devGrad\vecv$, we have by (iii) 
$\vecv\in\RTz\cap\RTz^{\bot_{\ltom}}=\{0\}$,
a contradiction to $1=\normhoom{\vecv_{n}}\to0$. 
The proof is complete.
\end{proof}

We recall the following well-known result for the spaces
\begin{align*}
\ggom&:=\set{u\in\ltom}{\Gradgrad u\in\Ltom},&
\ggcom&:=\ovl{\cicom}^{\ggom}.
\end{align*}

\begin{lem} 
\label{formulasone}
It holds $\ggcom=\htcom$ and $\ggczom=\{0\}$,
and there exists $c>0$ such that for all $u\in\htcom$
$$\norm{u}_{\htom}\leq c\,\normltom{\Gradgrad u} = c\,\normltom{\Delta u},\qquad
c\leq\sqrt{1+c_{\mathsf g}^2(1+c_{\mathsf g}^2)}\leq1+c_{\mathsf g}^2.$$
\end{lem}

%\begin{proof}
%Let $\varphi\in\cicom$. Then by the Friedrichs/Poincar\'e inequality, i.e.,
%$\normltom{\varphi}\leq c_{\mathsf g}\,\normltom{\grad \varphi}$, and 
%\begin{align*}
%\normltom{\grad \varphi}^2
%=-\scpltom{\varphi}{\Delta \varphi}
%\leq c_{\mathsf{g}}\,\normltom{\grad \varphi}\normltom{\Delta \varphi},
%\end{align*}
%we obtain $\normltom{\grad \varphi}\leq c_{\mathsf{g}}\,\normltom{\Delta \varphi}$.
%Moreover, by
%\begin{align*}
%\sum_{i,j=1}^{3}\normltom{\p_{i}\p_{j}\varphi}^2
%=\sum_{i,j=1}^{3}\scpltom{\p_{i}\p_{i}\varphi}{\p_{i}\p_{i}\varphi}
%=\normltom{\Delta \varphi}^2
%\end{align*}
%we see 
%\begin{align}
%\mylabel{htestcic}
%\norm{\varphi}_{\htom}^2\leq\big(1+c_{\mathsf{g}}^2(1+c_{\mathsf{g}}^2)\big)\normltom{\Delta\varphi}^2.
%\end{align}
%Let $u\in\ggcom$ and $(u_{n})\subset\cicom$ with $u_{n}\to u$ in $\ggom$. 
%By \eqref{htestcic}, $(u_{n})$ is a Cauchy sequence in $\htom$ converging to $u$ in $\htom$,
%which proves $u\in\htcom$. By continuity, \eqref{htestcic} holds for $u\in\htcom$,
%which finishes the proof.
%\end{proof}

By straight forward calculations and standard arguments for distributions, 
see the Appendix, we get the following.

\begin{lem} 
\label{formulastwo}
It holds:
\begin{itemize}
\item[\bf(i)]
$\skw\Gradgrad\htom=0$, i.e., Hessians are symmetric.
\item[\bf(ii)]
$\tr\Rot\RomS=0$, i.e., rotations of symmetric tensors are trace free.
\end{itemize}
These formulas extend to distributions as well.
\end{lem}

With Lemma \ref{formulasone} and Lemma \ref{formulastwo} let us now consider the linear operators
\begin{align}
\label{Gradgradcdef}
\Az:=\Gradgradc:\ggcom=\htcom\subset\ltom
&\To\ltomS,
&
u
&\mapsto\Gradgrad u,\\
\label{RotcSdef}
\Ao:=\RotcS:\RcomS\subset\ltomS
&\To\ltomT,
&
\bfM
&\mapsto\Rot\bfM,\\
\label{DivcTdef}
\At:=\DivcT:\DcomT\subset\ltomT
&\To\ltom,
&
\bfE
&\mapsto\Div\bfE.
\end{align}
These are well and densely defined and closed.
Closedness is clear. For densely definedness we look, e.g., at $\RotcS$.
For $\bfM\in\ltomS$ pick $(\bfPhi_{n})\subset\Cicom$ with $\bfPhi_{n}\to\bfM$ in $\Ltom$. Then 
$$\norm{\bfM-\sym\bfPhi_{n}}_{\Ltom}^2
+\norm{\skw\bfPhi_{n}}_{\Ltom}^2
=\norm{\bfM-\bfPhi_{n}}_{\Ltom}^2\to0,$$
showing $(\sym\bfPhi_{n})\subset\Cicom\cap\ltomS\subset\RcomS$
and $\sym\bfPhi_{n}\to\bfM$ in $\ltomS$. By Lemma \ref{formulasone} the kernels are
\begin{align*}
N(\Gradgradc)
=\ggczom
&=\{0\},
&
N(\RotcS)
&=\RczomS,
&
N(\DivcT)
&=\DczomT.
\end{align*}

\begin{lem} 
\label{adjoints}
The adjoints of \eqref{Gradgradcdef}, \eqref{RotcSdef}, \eqref{DivcTdef} are
\begin{align*}
\Azs=(\Gradgradc)^{*}=\divDivS:\dDomS\subset\ltomS
&\To\ltom,
&
\bfM
&\mapsto\divDiv\bfM,\\
\Aos=(\RotcS)^{*}=\symRotT:\symRomT\subset\ltomT
&\To\ltomS,
&
\bfE
&\mapsto\symRot\bfE,\\
\Ats=(\DivcT)^{*}=-\devGrad:\devGom=\hoom\subset\ltom
&\To\ltomT,
&
\vecv
&\mapsto-\devGrad\vecv
\end{align*}
with kernels
\begin{align*}
N(\divDivS)
&=\dDzomS,
&
N(\symRotT)
&=\symRzomT,
&
N(\devGrad)
&=\RTz.
\end{align*}
\end{lem} 

\begin{proof}
We have $\bfM\in D\big((\Gradgradc)^{*}\big)\subset\ltomS$
and $(\Gradgradc)^{*}\bfM=u\in\ltom$, if and only if
$\bfM\in\ltomS$ and there exists $u\in\ltom$, such that
\begin{align*}
&
&
\forall\,\varphi&\in D(\Gradgradc)=\htcom
&
\scp{\Gradgrad\varphi}{\bfM}_{\ltomS}
&=\scp{\varphi}{u}_{\ltom}\\
&\equi
&
\forall\,\varphi&\in\cicom
&
\scp{\Gradgrad\varphi}{\bfM}_{\Ltom}
&=\scp{\varphi}{u}_{\ltom},
\end{align*}
if and only if $\bfM\in\dDom\cap\ltomS=\dDomS$
and $\divDiv\bfM=u$.
Moreover, we observe that $\bfE\in D\big((\RotcS)^{*}\big)\subset\ltomT$
and $(\RotcS)^{*}\bfE=\bfM\in\ltomS$, if and only if
$\bfE\in\ltomT$ and there exists $\bfM\in\ltomS$, such that
(note $\sym^2=\sym$)
\begin{align*}
&
&
\forall\,\bfPhi&\in D(\RotcS)=\RcomS
&
\scp{\Rot\bfPhi}{\bfE}_{\ltomT}
&=\scp{\bfPhi}{\bfM}_{\ltomS}\\
&\equi
&
\forall\,\bfPhi&\in\Cicom\cap\ltomS
&
\scp{\Rot\sym\bfPhi}{\bfE}_{\Ltom}
&=\scp{\sym\bfPhi}{\bfM}_{\Ltom}\\
&\equi
&
\forall\,\bfPhi&\in\Cicom
&
\scp{\Rot\sym\bfPhi}{\bfE}_{\Ltom}
&=\scp{\sym\bfPhi}{\bfM}_{\Ltom}\\
&\equi
&
\forall\,\bfPhi&\in\Cicom
&
\scp{\Rot\sym\bfPhi}{\bfE}_{\Ltom}
&=\scp{\bfPhi}{\bfM}_{\Ltom},
\end{align*}
if and only if $\bfE\in\symRom\cap\ltomT=\symRomT$
and $\symRot\bfE=\bfM$. Similarly, we see that $\vecv\in D\big((\DivcT)^{*}\big)\subset\ltom$
and $(\DivcT)^{*}\vecv=\bfE\in\ltomT$, if and only if
$\vecv\in\ltom$ and there exists $\bfE\in\ltomT$, such that
(note $\dev^2=\dev$)
\begin{align*}
&
&
\forall\,\bfPhi&\in D(\Divc_{\bbS})=\DcomT
&
\scp{\Div\bfPhi}{\vecv}_{\ltom}
&=\scp{\bfPhi}{\bfE}_{\ltomT}\\
&\equi
&
\forall\,\bfPhi&\in\Cicom\cap\ltomT
&
\scp{\Div\dev\bfPhi}{\vecv}_{\ltom}
&=\scp{\dev\bfPhi}{\bfE}_{\Ltom}\\
&\equi
&
\forall\,\bfPhi&\in\Cicom
&
\scp{\Div\dev\bfPhi}{\vecv}_{\ltom}
&=\scp{\dev\bfPhi}{\bfE}_{\Ltom}\\
&\equi
&
\forall\,\bfPhi&\in\Cicom
&
\scp{\Div\dev\bfPhi}{\vecv}_{\ltom}
&=\scp{\bfPhi}{\bfE}_{\Ltom},
\end{align*}
if and only if $\vecv\in\devGom=\hoom$ and $-\devGrad\vecv=\bfE$
using Lemma \ref{kerneldevGrad}. Lemma \ref{kerneldevGrad} also shows
$N(\devGrad)
=\devGzom
=\RTz$,
completing the proof.
\end{proof}

\begin{rem}
\label{opcomporem}
Note that, e.g., the second order operator $\Gradgradc$
is ``one'' operator and not a composition
of the two first order operators $\Gradc$ and $\gradc$.
Similarly the operator $\divDivS$, $\symRotT$, resp. $\devGrad$
has to be understood as ``one'' operator.
\end{rem}

We observe the following complex properties for $\Az$, $\Ao$, $\At$, and $\Azs$, $\Aos$, $\Ats$.

\begin{lem}
\mylabel{complexgradGradRotlem}
It holds
\begin{align*}
\RotcS\Gradgradc&=0,
&
\DivcT\RotcS&=0,
&
\divDivS\symRotT&=0,
&
\symRotT\devGrad&=0,
\end{align*}
i.e.,
\begin{align*}
R(\Gradgradc)&\subset N(\RotcS),
&
R(\symRotT)&\subset N(\divDivS),\\
R(\RotcS)&\subset N(\DivcT),
&
R(\devGrad)&\subset N(\symRotT).
\end{align*}
\end{lem}

\begin{proof}
For $\bfE=\Rot\bfM\in R(\RotcS)$
with $\bfM\in D(\RotcS)$ there exists a sequence $(\bfM_{n})\subset\Cicom\cap\ltomS$
such that $\bfM_{n}\to\bfM$ in the graph norm of $D(\RotcS)$. As 
$$\Rot\big(\Cicom\cap\ltomS\big)
\subset\Cicom\cap\ltomT\cap\Dzom
\subset N(\DivcT)$$
we have $\bfE\in N(\DivcT)$ since $\bfE\ot\Rot\bfM_{n}\in N(\DivcT)$.
Hence $R(\RotcS)\subset N(\DivcT)$, i.e.,
$\DivcT\RotcS=0$ and for the adjoints we have
$\symRotT\devGrad=0$. Analogously, we see the other two inclusions.
\end{proof}

\begin{rem} 
\label{complexGradgrad}
The latter considerations show that the sequence
$$\begin{CD}
\{0\} @> 0 >>
\htcom @> \Gradgradc >>
\RcomS @> \RotcS >>
\DcomT @> \DivcT >>
\ltom @> \pi_{\RTz} >>
\RTz
\end{CD}$$
and thus also its dual or adjoint sequence
$$\begin{CD}
\{0\} @< 0 <<
\ltom @< \divDivS <<
\dDomS @< \symRotT << 
\symRomT @< -\devGrad <<
\hoom @< \iota_{\RTz} <<
\RTz
\end{CD}$$
are Hilbert complexes. 
Here $\pi_{\RTz}:\ltom\to\RTz$ 
denotes the orthogonal projector onto $\RTz$
with adjoint $\pi_{\RTz}^{*}=\iota_{\RTz}:\RTz\to\ltom$,
the canonical embedding. 
The first complex might by called $\Gradgradc$-complex
and the second one $\divDiv$-complex.
\end{rem}

\subsection{Topologically Trivial Domains}
\mylabel{kernelsecsimple}

We start with a useful lemma, which will be shown in the Appendix, 
collecting a few differential identities, 
which will be utilized in the proof of the subsequent main theorem.

\begin{lem} 
\label{formulas}
Let $u$, $\vecv$, and $\bfE$ be distributional
scalar, vector, and tensor fields. Then
\begin{enumerate}
\item[\bf(i)]
$2\skw\Grad\vecv=\spn\rot\vecv$,
\item[\bf(ii)]
$\Rot\spn\vecv=(\div\vecv)\,\bfI-(\Grad\vecv)^{\top}$
and, as a consequence, $\tr\Rot\spn\vecv=2\div\vecv$,
\item[\bf(iii)]
$\Div(u\,\bfI)=\grad u$ and 
$\Rot(u\,\bfI)=-\spn\grad u$,
\item[\bf(iv)]
$2\grad\div\vecv=3\Div\big(\dev\,(\Grad\vecv)^{\top}\big)$,
\item[\bf(v)]
$\skw\Rot\bfE=\spn\vecw$
and
$\Div(\symRot\bfE)=\rot\vecw$
with
$2\vecw=\Div\bfE^{\top}-\grad(\tr\bfE)$,
\item[\bf(vi)]
$\Div(\spn\vecv)=-\rot \vecv$.
\end{enumerate}
\end{lem}

Observe that we already know that $N(\Gradgradc)=\{0\}$ 
and $N(\devGrad)=\RTz$.
If the topology of the underlying domain is trivial, 
we will now characterize the remaining kernels and the ranges of the linear operators 
$\Gradgradc$, $\RotcS$, $\DivcT$, and 
$\devGrad$, $\symRotT$, $\divDivS$.

\begin{theo} 
\label{maintheoregpot}
Let $\om$ be additionally topologically trivial. Then 
\begin{align*}
\text{\bf(i)}&
&
\RczomS
=N(\RotcS)
&=R(\Gradgradc)
=\Gradgrad\htcom,\\
\text{\bf(ii)}&
&
\DczomT
=N(\DivcT)
&=R(\RotcS)
=\Rot\HocomS,\\
\text{\bf(iii)}&
&
\RTz^{\bot_{\ltom}}
=N(\pi_{\RTz})
&=R(\DivcT)
=\Div\HocomT,\\
\text{\bf(iv)}&
&
\symRzomT
=N(\symRotT)
&=R(\devGrad)
=\devGrad\hoom,\\
\text{\bf(v)}&
&
\dDzomS
=N(\divDivS)
&=R(\symRotT)
=\symRot\HoomT,\\
\text{\bf(vi)}&
&
\ltom 
=N(0)
&=R(\divDivS)
=\divDiv\HtomS.
\end{align*}
Especially, all latter ranges are closed and admit regular $\ho$-potentials.
The corresponding linear and continuous (regular) potential operators are given by
\begin{align*}
\Pot_{\Gradgradc}
=\Pot_{\gradc}\Pot_{\Gradc}
:\RczomS&\To\htcom,\\
\Pot_{\RotcS}
=\sym\big(1-2\Grad\Pot_{\rotc}\spn^{-1}\skw\big)\Pot_{\Rotc}
:\DczomT&\To\HocomS,\\
\Pot_{\DivcT}
=\dev\big(1+\foh\Grad^{\top}\Pot_{\divc}\tr\big)\Pot_{\Divc}
:\RTz^{\bot_{\ltom}}&\To\HocomT,\\
\Pot_{\devGrad}
=\Grad^{-1}\big(1+\foh(\grad^{-1}\Div(\,\cdot\,)^{\top})\,\bfI\big)
:\symRzomT&\To\hoom,\\
\Pot_{\symRotT}
=\dev\Pot_{\Rot}\big(1+\spn\rot^{-1}\Div\big)
:\dDzomS&\To\HoomT,\\
\Pot_{\divDivS}
=\sym\Pot_{\Div}\Pot_{\div}
:\ltom&\To\HtomS.
\end{align*}
\end{theo}

\begin{rem} 
\label{maintheoregpotrem}
It holds
$$\HoomS
=\sym\Hoom,\quad
\HoomT
=\dev\Hoom,\quad
\HocomS
=\sym\Hocom,\quad
\HocomT
=\dev\Hocom$$
as, e.g., $\dev\Hoom\subset\HoomT=\dev\HoomT\subset\dev\Hoom$.
The same holds for the corresponding spaces of skew-symmetric tensor fields as well. 
Moreover:
\begin{itemize}
\item[\bf(i)]
Theorem \ref{maintheoregpot} holds also for the other set
of canonical boundary conditions, which follows directly from the proof.
\item[\bf(ii)]
A closer inspection shows, that for (iii) and (vi), i.e.,
$\Pot_{\DivcT}$ and $\Pot_{\divDivS}$, 
only the potential operators corresponding to the divergence, i.e., 
$\Pot_{\divc}$, $\Pot_{\Divc}$,
$\Pot_{\Div}$, $\Pot_{\div}$, are involved.
As Lemma \ref{divpotlem} does not need any topological assumptions,
(iii) and (vi), together with the representations of the potential operators,
hold for general topologies as well.
\end{itemize}
\end{rem} 

\begin{proof}[Proof of Theorem \ref{maintheoregpot}]
Note that by Lemma \ref{kerneldevGrad} (iii), Lemma \ref{formulasone}, and
Lemma \ref{complexgradGradRotlem} all inclusions of the type $R(\ldots)\subset N(\ldots)$ easily follow. 
Therefore it suffices to show that $N(\ldots)$ is included in the corresponding space appearing 
at the end of each line in (i) - (vi), which itself is obviously included in $R(\ldots)$.
Throughout the proof we will frequently use the formulas of Lemma \ref{formulas}.

ad (i):
Let $\bfM\in\RczomS=N(\RotcS)$. 
Applying Lemma \ref{napotlem} for $m=0$ row-wise, there is a vector field 
$\vecv := \Pot_{\Gradc} \bfM \in\hocom$ with $\bfM=\Grad\vecv$.
Since $\skw \bfM=0$ and $2\skw \Grad\vecv = \spn \rot \vecv$, it follows that $\rot \vecv = 0$.
By Lemma \ref{napotlem} for $m=1$ there is a function 
$u := \Pot_{\gradc} \vecv \in\htcom$ with $\vecv=\grad u$.
Hence $\bfM = \Grad\vecv = \Gradgrad u\in\Grad \grad \htcom$. 
So $\RczomS\subset \Gradgrad\htcom$,
which completes the proof of (i).
Note that
$$\Pot_{\Gradgradc}\bfM
:=u=\Pot_{\gradc}\Pot_{\Gradc}\bfM
\in\htcom,$$
from which it directly follows that $\Pot_{\Gradgradc}$ is linear and bounded.

ad (ii):
Let $\bfE\in\DczomT=N(\DivcT)$. 
Then there is a tensor field $\bfN := \Pot_{\Rotc} \bfE \in \Hocom$ 
with $\bfE = \Rot\bfN$, see Lemma \ref{rotpotlem} for $m=0$ applied row-wise. 
Since $\tr \bfE = 0$ and $\tr \Rot \sym\bfN = 0$, it follows that $\tr \Rot \skw\bfN = 0$. 
Now let $\vecv := \spn^{-1} \skw\bfN \in \hocom$, i.e., $\skw \bfN = \spn\vecv$. 
Since $\tr \Rot \spn\vecv = 2 \, \div\vecv$, it follows that $\div\vecv = 0$.
Therefore, there is a vector field $\vecw := \Pot_{\rotc} \vecv \in \htcom$ 
such that $\vecv = \rot \vecw$, see Lemma \ref{rotpotlem} for $m=1$.
So we have
$$\Rot\skw\bfN
=\Rot\spn\rot\vecw
=2\Rot\skw\Grad\vecw
=-2\Rot\sym\Grad\vecw.$$
Hence
$$ \bfE 
  = \Rot \bfN 
  = \Rot \sym \bfN + \Rot \skw \bfN 
  = \Rot \bfM,\qquad
\bfM 
:= \sym\bfN - 2  \sym \Grad \vecw \in \HocomS,  $$
So $ \DczomT\subset \Rot \HocomS$,
which completes the proof of (ii).
Note that
\begin{align*}
\Pot_{\RotcS}\bfE
:=\bfM 
&=\sym \Pot_{\Rotc}\bfE 
-2\sym\Grad\big(\Pot_{\rotc}\spn^{-1}\skw \Pot_{\Rotc}\bfE\big)\\
&=\sym\big(1-2\Grad \Pot_{\rotc}\spn^{-1}\skw\big)
\Pot_{\Rotc}\bfE
\in\HocomS,
\end{align*}
from which it directly follows that $\Pot_{\RotcS}$ is linear and bounded.

ad (iii): 
Let $\vecv \in \RTz^{\bot_{\ltom}}=N(\pi_{\RTz})$. 
As $\vecv \in (\rt)^{\bot_{\ltom}}$,
there is a tensor field $\bfF = \Pot_{\Divc} \vecv \in \Hocom$ 
with $\vecv = \Div \bfF$, see Lemma \ref{divpotlem} for $m=0$ applied row-wise. 
We have $\Div\bfF\in \RTz^{\bot_{\ltom}}$ as well as $\Div \dev\bfF \in \RTz^{\bot_{\ltom}}$.
Hence $\grad (\tr\bfF)=\Div ((\tr \bfF) \, \bfI) \in  \RTz^{\bot_{\ltom}}$,
which implies $\tr \bfF \in \hocom\cap\ltzom$. 
Therefore, there is a vector field $\vecw := \Pot_{\divc} \tr \bfF \in \htcom$ 
with $\tr \bfF = \div \vecw$, see Lemma \ref{divpotlem} for $m=1$. Thus
$$  \Div ((\tr \bfF) \, \bfI)
   = \grad \div \vecw = \frac{3}{2} \Div \big(\dev\,(\Grad \vecw)^{\top} \big).$$
Hence
$$  \vecv 
  = \Div \bfF 
  = \Div \dev\bfF + \frac{1}{3} \Div((\tr \bfF) \bfI) 
  = \Div \bfE,\qquad
  \bfE := \dev\big(\bfF + \foh(\Grad \vecw)^{\top}\big) \in \HocomT.$$
So $\RTz^{\bot_{\ltom}} \subset \Div \HocomT$,
which completes the proof of (iii).
Note that
\begin{align*}
\Pot_{\DivcT} \vecv 
:=\bfE
&=\dev\big(\Pot_{\Divc}\vecv
+\foh(\Grad \Pot_{\divc}\tr \Pot_{\Divc}\vecv)^{\top}\big)\\ 
&=\dev\big(1+\foh\Grad^{\top}\Pot_{\divc}\tr\big)
\Pot_{\Divc}\vecv
\in\HocomT,
\end{align*}
from which it directly follows that $\Pot_{\DivcT}$ is linear and bounded.

ad (iv): 
Let $\bfE\in\symRzomT=N(\symRotT)$.
Then (trivially)
$\Div\symRot\bfE=0$ and it follows
$$\rot\vecw=0
\quad\text{with}\quad
\vecw:=\frac{1}{2}\big(\Div\bfE^{\top}-\grad(\tr\bfE)\big)
=\frac{1}{2}\Div\bfE^{\top}$$
and
\begin{equation} 
\label{skwRot}
\skw\Rot\bfE=\spn\vecw.
\end{equation}
So $\vecw\in\rhmozom$. Therefore, there is a unique scalar field $u:=\grad^{-1}\vecw\in\ltzom$, such that
$$\vecw=\grad u,$$
see Corollary \ref{nahmorangecor} for $m = 1$.
As $\Rot(u\,\bfI)=-\spn\grad u$ implies $\symRot(u\,\bfI)=0$, we see
$$\bfF:=\bfE+u\,\bfI\in\symRzom.$$
Moreover, by \eqref{skwRot}
$$\skw\Rot\bfF
=\skw\Rot\bfE+\skw\Rot(u\,\bfI)
=\spn\vecw-\spn\grad u
=0.$$
Hence $\bfF\in\Rzom$. 
Therefore, there is a unique vector field $\vecv:=\Grad^{-1}\bfF\in\hoom\cap\ltzom$, such that
$\bfF=\Grad\vecv$, see Lemma \ref{napotlem} for $m=0$.
So we have
$$\bfE=\Grad\vecv -u\,\bfI.$$
From the additional condition $\tr\bfE=0$ it follows that $3\,u=\tr\Grad\vecv=\div\vecv$ leading to
$$\bfE=\devGrad\vecv,\quad
\vecv\in\hoom .$$
So $\symRzomT \subset \devGrad \hoom$,
which completes the proof of (iv). Note that
\begin{align*}
\Pot_{\devGrad} \bfE
:=\vecv
&=\Grad^{-1}\big(\bfE+\foh(\grad^{-1}\Div\bfE^{\top})\,\bfI\big)\\
&=\Grad^{-1}\big(1+\foh(\grad^{-1}\Div(\,\cdot\,)^{\top})\,\bfI\big)\bfE
\in\hoom,
\end{align*}
from which it directly follows that $\Pot_{\devGrad}$ is linear and bounded.

ad (v): 
Let $\bfM\in\dDzomS=N(\divDivS)$.
So $\Div\bfM\in\dhmozom$
and there is a unique vector field $\vecv:=\rot^{-1}\Div\bfM\in\dczom$, such that 
$$\Div\bfM=\rot\vecv=-\Div(\spn\vecv),$$
see  Corollary \ref{rothmorangecor} for $m=1$.
Hence $\Div(\bfM+\spn\vecv)=0$, i.e., $\bfM+\spn\vecv\in\Dzom$,
and by Lemma \ref{rotpotlem} there is a tensor field $\bfF:=\Pot_{\Rot}(\bfM+\spn\vecv)\in\Hoom$, such that 
$$\bfM+\spn\vecv=\Rot\bfF.$$
Observe that $\bfM$ is symmetric and $\spn\vecv$ is skew-symmetric. Thus
$$\bfM=\symRot\bfF
\quad\text{and}\quad
\spn\vecv=\skw\Rot\bfF,\qquad
\bfF\in\Hoom,$$
and hence
$$\bfM=\symRot\bfF=\symRot\bfE
\quad\text{with}\quad 
\bfE:=\dev\bfF\in\HoomT,$$
as $\dev\bfF=\bfF-\frac{1}{3}(\tr\bfF)\,\bfI$
and $\symRot((\tr\bfF)\,\bfI)=0$.
So $\dDzomS\subset\symRot\HoomT$,
which completes the proof of (v).
Note that
\begin{align*}
\Pot_{\symRotT} \bfM 
:=\bfE
&=\dev\Pot_{\Rot}\big(\bfM+\spn\rot^{-1}\Div\bfM\big)\\
&=\dev\Pot_{\Rot}\big(1+\spn\rot^{-1}\Div\big)\bfM
\in\HoomT,
\end{align*}
from which it directly follows that $\Pot_{\symRotT}$ is linear and bounded.

ad (vi): 
Let $u \in \ltom =N(0)$. Then there is a vector field $\vecv = \Pot_{\div} u \in \hoom$ 
with $u = \div\vecv$, see Lemma \ref{divpotlem} for $m=0$,
and a tensor field $\bfN = \Pot_{\Div} \vecv \in \Htom$ such that $\vecv = \Div \bfN$, 
see Lemma \ref{divpotlem} for $m=1$ applied row-wise. 
Since $\divDiv \skw \bfN=0$, it follows that
$$  u = \divDiv \bfN = \div \Div \bfM
  \quad \text{with} \quad 
  \bfM =: \sym\bfN \in \HtomS.$$
So $\ltom\subset\divDiv\HtomS$,
which completes the proof of (vi). Note that
$$  \Pot_{\divDivS} u
  :=\bfM = \sym \Pot_{\Div} \Pot_{\div} u 
   \in \HtomS,$$
from which it directly follows that $\Pot_{\divDivS}$ is linear and bounded.
\end{proof}

Provided that the domain $\om$ has trivial topology, 
Theorem \ref{maintheoregpot} implies that
the densely defined, closed and unbounded linear operators
$\Gradgradc$, $\RotcS$, $\DivcT$,
and their adjoints 
$\divDivS$, $\symRotT$, $\devGrad$ have closed ranges
and that all relevant cohomology groups are trivial, as
\begin{align*}
N(\Gradgradc)\cap N(0)
&=\{0\}\cap\ltom
=\{0\},\\
N(\RotcS)\cap N(\divDivS)
&=\RczomS\cap\dDzomS
=\RczomS\cap\symRot\HoomT\\
&=N(\RotcS)\cap R(\symRotT)
=\{0\},\\
N(\DivcT)\cap N(\symRotT)
&=\DczomT\cap\symRzomT
=\DczomT\cap\devGrad\hoom\\
&=N(\DivcT)\cap R(\devGrad)
=\{0\},\\
N(\pi_{\RTz})\cap N(\devGrad)
&=\RTz^{\bot_{\ltom}}\cap\RTz
=\{0\}.
\end{align*}
In this case, the reduced operators are
\begin{align*}
\cAz=\Gradgradc:\htcom\subset\ltom
&\To\RczomS,\\
\cAo=\RotcS:\RcomS\cap\dDzomS\subset\dDzomS
&\To\DczomT,\\
\cAt=\DivcT:\DcomT\cap\symRzomT\subset\symRzomT
&\To\RTz^{\bot_{\ltom}},\\
\cAzs=\divDivS:\dDomS\cap\RczomS\subset\RczomS
&\To\ltom,\\
\cAos=\symRotT:\symRomT\cap\DczomT\subset\DczomT
&\To\dDzomS,\\
\cAts=-\devGrad:\hoom\cap\RTz^{\bot_{\ltom}}\subset\RTz^{\bot_{\ltom}}
&\To\symRzomT
\end{align*}
as 
\begin{align*}
R(\divDivS)=\ltom,\qquad
R(\DivcT)=\RTz^{\bot_{\ltom}}.
\end{align*}
The functional analysis toolbox Section \ref{fuanasec},
e.g., Lemma \ref{Hilbertadjoint},
immediately lead to the following implications about
Helmholtz type decompositions, Friedrichs/Poincar\'e type estimates
and continuous inverse operators.

\begin{theo} 
\label{maintheo}
Let $\om$ be additionally topologically trivial. 
Then all occurring ranges are closed and all related cohomology groups are trivial.
Moreover, the Helmholtz type decompositions 
\begin{align*}
\ltomS
&=\RczomS
\oplus_{\ltomS}
\dDzomS,
&
\ltomT
&=\DczomT
\oplus_{\ltomT}
\symRzomT
\end{align*} 
hold. The kernels can be represented by the following closed ranges
\begin{align*}
\RczomS
&=\Gradgrad\htcom,\\
\symRot\HoomT
=\dDzomS
&=\symRot\symRomT
=\symRot\big(\symRomT\cap\DczomT\big),\\
\Rot\HocomS
=\DczomT
&=\Rot\RcomS
=\Rot\big(\RcomS\cap\dDzomS\big),\\
\symRzomT
&=\devGrad\hoom
=\devGrad\big(\hoom\cap\RTz^{\bot_{\ltom}}\big),
\intertext{and it holds}
\divDiv\HtomS
=\ltom
&=\divDiv\dDomS
=\divDiv\big(\dDomS\cap\RczomS\big),\\
\Div\HocomT
=\RTz^{\bot_{\ltom}}
=N(\pi_{\RTz})
&=\Div\DcomT
=\Div\big(\DcomT\cap\symRzomT\big).
\end{align*}
All potentials depend continuously on the data.
The potentials on the very right hand sides are uniquely determined.
There exist positive constants $c_{\mathsf{Gg}}$, $c_{\mathsf{D}}$, $c_{\mathsf{R}}$
such that the Friedrichs/Poincar\'e type estimates
\begin{align*}
\forall\,u&\in\htcom
&
\normltom{u}&\leq c_{\mathsf{Gg}}\,\normLtom{\Gradgrad u},\\
\forall\,\bfM&\in\dDomS\cap\RczomS
&
\normLtom{\bfM}&\leq c_{\mathsf{Gg}}\,\normltom{\divDiv\bfM},\\
\forall\,\bfE&\in\DcomT\cap\symRzomT
&
\normLtom{\bfE}&\leq c_{\mathsf{D}}\,\normltom{\Div\bfE},\\
\forall\,\vecv&\in\hoom\cap\RTz^{\bot_{\ltom}}
&
\normltom{\vecv}&\leq c_{\mathsf{D}}\,\normLtom{\devGrad\vecv},\\
\forall\,\bfM&\in\RcomS\cap\dDzomS
&
\normLtom{\bfM}&\leq c_{\mathsf{R}}\,\normLtom{\Rot\bfM},\\
\forall\,\bfE&\in\symRomT\cap\DczomT
&
\normLtom{\bfE}&\leq c_{\mathsf{R}}\,\normLtom{\symRot\bfE}
\end{align*}
hold.
Moreover, the reduced versions of the operators 
$$\Gradgradc,\quad
\divDivS,\quad
\DivcT,\quad
\devGrad,\quad
\RotcS,\quad
\symRotT\quad$$
have continuous inverse operators 
\begin{align*}
(\Gradgradc)^{-1}:
\RczomS&\To\htcom,\\
(\divDivS)^{-1}:
\ltom&\To\dDomS\cap\RczomS,\\
(\DivcT)^{-1}:
\RTz^{\bot_{\ltom}}&\To\DcomT\cap\symRzomT,\\
(\devGrad)^{-1}:
\symRzomT&\To\hoom\cap\RTz^{\bot_{\ltom}},\\
(\RotcS)^{-1}:
\DczomT&\To\RcomS\cap\dDzomS,\\
(\symRotT)^{-1}:
\dDzomS&\To\symRomT\cap\DczomT
\end{align*}
with norms $(1+c_{\mathsf{Gg}}^2)^{\oh}$, $(1+c_{\mathsf{D}}^2)^{\oh}$, resp.
$(1+c_{\mathsf{R}}^2)^{\oh}$.
\end{theo}

\begin{rem} 
\label{devRotrem}
Let $\om$ be additionally topologically trivial.
The Friedrichs/Poincar\'e type estimate for $\Rot\bfM$ in the latter theorem can be slightly sharpened.
Utilizing Lemma \ref{formulastwo} we observe $\tr\Rot\bfM=0$ 
and thus $\dev\Rot\bfM=\Rot\bfM$ for $\bfM\in\RomS$. Hence
$$\forall\,\bfM\in\RcomS\cap\dDzomS\qquad
\normLtom{\bfM}\leq c_{\mathsf{R}}\,\normLtom{\dev\Rot\bfM}.$$
Similarly and trivially we see
$$\forall\,u\in\htcom\qquad
\normltom{u}\leq c_{\mathsf{Gg}}\,\normLtom{\sym\Gradgrad u}.$$
\end{rem}

Recalling Remark \ref{complexGradgrad} we have the following result.

\begin{rem} 
\label{Hilbertcomplex}
Let $\om$ be additionally topologically trivial.
Theorem \ref{maintheoregpot} and Theorem \ref{maintheo}  easily lead to the following result in terms of complexes: 
The sequence
$$\begin{CD}
\{0\} @> 0 >>
\htcom @> \Gradgradc >>
\RcomS @> \RotcS >>
\DcomT @> \DivcT >>
\ltom @> \pi_{\RTz} >>
\RTz
\end{CD}$$
and thus also its dual or adjoint sequence
$$\begin{CD}
\{0\} @< 0 <<
\ltom @< \divDivS <<
\dDomS @< \symRotT << 
\symRomT @< -\devGrad <<
\hoom @< \iota_{\RTz} <<
\RTz
\end{CD}$$
are closed and exact Hilbert complexes. 
\end{rem}

\begin{rem}
The part
$$\begin{CD}
\{0\} @> 0 >>
\htcom @> \Gradgradc >>
\RcomS @> \RotcS >>
\Ltom
\end{CD}$$
of the Hilbert complex from above and the related adjoint complex
$$\begin{CD}
\{0\} @< 0 <<
\ltom @< \divDivS <<
\dDomS @< \symRotT << 
\symRomT 
\end{CD}$$
have been discussed in \cite{Quenneville:2015} for problems in general relativity.
\end{rem} 

\begin{rem}
In 2D and under similar assumptions
we obtain by completely analogous but much simpler arguments that the Hilbert complexes
\begin{align*}
\begin{CD}
\{0\} @> 0 >>
\htcom @> \Gradgradc >>
\RcomS @> \RotcS >>
\ltom @> \pi_{\RTz} >>
\RTz
\end{CD},\\
\begin{CD}
\{0\} @< 0 <<
\ltom @< \divDivS <<
\dDomS @< \symRot << 
\hoom @< \iota_{\RTz} <<
\RTz
\end{CD}
\end{align*}
are dual to each other, closed and exact. 
Contrary to the 3D case, the operator 
$\RotcS$ maps a tensor field to a vector field
and the operator $\symRot\cong\sym\Grad$ is applied row-wise to a vector field 
and maps this vector field to a tensor field. 
The associated Helmholtz decomposition is
$$\ltomS
=\RczomS
\oplus_{\ltomS}
\dDzomS$$
with
$$\RczomS
=\Gradgrad\htcom,\quad
\dDzomS
=\symRot\hoom.$$
\end{rem} 

Theorem \ref{maintheoregpot} leads to the following so called regular decompositions.

\begin{theo} 
\label{regdecotheo}
Let $\om$ be additionally topologically trivial. Then the regular decompositions
\begin{align*}
\RcomS
&=\HocomS+\RczomS,
&
\RczomS
&=\Gradgrad\htcom,\\
\DcomT
&=\HocomT+\DczomT,
&
\DczomT
&=\Rot\HocomS,\\
\symRomT
&=\HoomT+\symRzomT,
&
\symRzomT
&=\devGrad\hoom,\\
\dDomS
&=\HtomS+\dDzomS,
&
\dDzomS
&=\symRot\HoomT
\end{align*}
hold with linear and continuous (regular) decomposition resp. potential operators
\begin{align*}
\Pot_{\RcomS,\HocomS}
&:\RcomS\To\HocomS,
&
\Pot_{\RcomS,\htcom}
&:\RcomS\To\htcom,\\
\Pot_{\DcomT,\HocomT}
&:\DcomT\To\HocomT,
&
\Pot_{\DcomT,\HocomS}
&:\DcomT\To\HocomS,\\
\Pot_{\symRomT,\HoomT}
&:\symRomT\To\HoomT,
&
\Pot_{\symRomT,\hoom}
&:\symRomT\To\hoom,\\
\Pot_{\dDomS,\HtomS}
&:\dDomS\To\HtomS,
&
\Pot_{\dDomS,\HoomT}
&:\dDomS\To\HoomT.
\end{align*}
\end{theo}

\begin{proof}
Let, e.g., $\bfE\in\symRomT$. Then 
$$\symRot\bfE\in\dDzomS=\symRot\HoomT$$
with linear and continuous potential operator $\Pot_{\symRotT}:\dDzomS\To\HoomT$ 
by Theorem \ref{maintheoregpot}.
Thus, there is $\tilde\bfE:=\Pot_{\symRotT}\symRot\bfE\in\HoomT$ 
depending linearly and continuously on $\bfE$ with $\symRot\tilde\bfE=\symRot\bfE$. Hence, 
$$\bfE-\tilde\bfE\in\symRzomT=\devGrad\hoom$$
with linear and continuous potential operator $\Pot_{\devGrad}:\symRzomT\To\hoom$ 
by Theorem \ref{maintheoregpot}.
Hence, there exists $\vecv:=\Pot_{\devGrad}(\bfE-\tilde\bfE)\in\hoom$ 
with $\devGrad\vecv=\bfE-\tilde\bfE$ and $\vecv$ depends linearly and continuously on $\bfE$.
The other assertions are proved analogously.
\end{proof}

Looking at the latter proof we see that the regular potential operators are given by
\begin{align}
\mylabel{PPdef}
\begin{split}
\Pot_{\RcomS,\HocomS}
=\Pot_{\RotcS}\Rot
&:\RcomS\To\HocomS,\\
\Pot_{\RcomS,\htcom}
=\Pot_{\Gradgradc}(1-\Pot_{\RotcS}\Rot)
&:\RcomS\To\htcom,\\
\Pot_{\DcomT,\HocomT}
=\Pot_{\DivcT}\Div
&:\DcomT\To\HocomT,\\
\Pot_{\DcomT,\HocomS}
=\Pot_{\RotcS}(1-\Pot_{\DivcT}\Div)
&:\DcomT\To\HocomS,\\
\Pot_{\symRomT,\HoomT}
=\Pot_{\symRotT}\symRot
&:\symRomT\To\HoomT,\\
\Pot_{\symRomT,\hoom}
=\Pot_{\devGrad}(1-\Pot_{\symRotT}\symRot)
&:\symRomT\To\hoom,\\
\Pot_{\dDomS,\HtomS}
=\Pot_{\divDivS}\divDiv
&:\dDomS\To\HtomS,\\
\Pot_{\dDomS,\HoomT}
=\Pot_{\symRotT}(1-\Pot_{\divDivS}\divDiv)
&:\dDomS\To\HoomT.
\end{split}
\end{align}
Hence the regular decompositions of Theorem \ref{regdecotheo} 
can be slightly refined to even direct regular decompositions.

\begin{cor} 
\label{regdecotheocor}
Let $\om$ be additionally topologically trivial. 
Then the direct regular decompositions
\begin{align*}
\RcomS
&=\Pot_{\RotcS}\DczomT\dotplus\RczomS,
&
\Pot_{\RotcS}\DczomT
&\subset\HocomS,\\
\DcomT
&=\Pot_{\DivcT}\RTz^{\bot_{\ltom}}\dotplus\DczomT,
&
\Pot_{\DivcT}\RTz^{\bot_{\ltom}}
&\subset\HocomT,\\
\symRomT
&=\Pot_{\symRotT}\dDzomS\dotplus\symRzomT,
&
\Pot_{\symRotT}\dDzomS
&\subset\HoomT,\\
\dDomS
&=\Pot_{\divDivS}\ltom\dotplus\dDzomS,
&
\Pot_{\divDivS}\ltom
&\subset\HtomS
\end{align*}
hold. More precisely
\begin{align*}
\RcomS
&=\Pot_{\RotcS}\DczomT\dotplus\Gradgrad\Pot_{\Gradgradc}\RczomS,\\
\DcomT
&=\Pot_{\DivcT}\RTz^{\bot_{\ltom}}\dotplus\Rot\Pot_{\RotcS}\DczomT,\\
\symRomT
&=\Pot_{\symRotT}\dDzomS\dotplus\devGrad\Pot_{\devGrad}\symRzomT,\\
\dDomS
&=\Pot_{\divDivS}\ltom\dotplus\symRot\Pot_{\symRotT}\dDzomS
\end{align*}
with
\begin{align*}
\Pot_{\Gradgradc}\RczomS
&\subset\htcom,
&
\Pot_{\devGrad}\symRzomT
&\subset\hoom,\\
\Pot_{\RotcS}\DczomT
&\subset\HocomS,
&
\Pot_{\symRotT}\dDzomS
&\subset\HoomT.
\end{align*}
\end{cor}

Here, $\dotplus$ denotes the direct sum.

\begin{proof}
For $\bfM\in\RczomS\cap\Pot_{\RotcS}\DczomT$ we have
$\bfM=\Pot_{\RotcS}\bfE$ with some $\bfE\in\DczomT$.
Thus $0=\Rot\bfM=\bfE$ showing $\bfM=0$ and hence the directness of the first regular decomposition.
The directness of the others follows similarly.
\end{proof}

\subsection{General Bounded Strong Lipschitz Domains}
\mylabel{kernelsecgeneral}

In this section we consider bounded strong Lipschitz domains $\Omega$ 
of general topology and we will extend the results of the previous section as follows. 
The $\Gradgradc$- and the $\divDiv$-complexes remain closed 
and all associated cohomology groups are finite-dimensional.
Moreover, the respective inverse operators are continuous and even compact,
and corresponding Friedrichs/Poincar\'e type estimates hold.
We will show this by verifying the compactness properties of Lemma \ref{compemblem} 
for the various linear operators of the complexes. 
Then Lemma \ref{helmrefined}, Remark \ref{clrangecompembrem}, and Theorem \ref{fatbmaintheogen} 
immediately lead to the desired results.
Using Rellich's selection theorem, we have the following compact embeddings
\begin{align*}
D(\Gradgradc)\cap D(0)
=\htcom&\overset{\text{\sf cpt}}{\hookrightarrow}\ltom,\\
D(\pi_{\mathsf{RT_0}})\cap D(\devGrad) 
=\hoom&\overset{\text{\sf cpt}}{\hookrightarrow}\ltom.
\end{align*}
The two missing compactness results that would immediately lead to the desired results are 
\begin{align}
\mylabel{crucialcptembS}
D(\RotcS)\cap D(\divDivS)
=\RcomS\cap\dDomS&\overset{\text{\sf cpt}}{\hookrightarrow}\ltomS,\\
\mylabel{crucialcptembT}
D(\DivcT) \cap D(\symRotT) 
=\DcomT\cap\symRomT&\overset{\text{\sf cpt}}{\hookrightarrow}\ltomT.
\end{align}

The main aim of this section is to show the compactness of the two crucial embeddings 
\eqref{crucialcptembS} and \eqref{crucialcptembT}.
As a first step we consider a trivial topology.

\begin{lem}
\label{cptemblemtrivtop}
Let $\om$ be additionally topologically trivial.
Then the embeddings \eqref{crucialcptembS}, \eqref{crucialcptembT} are compact.
\end{lem}

\begin{proof}
Let $(\bfM_{n})$ be a bounded sequence in $\RcomS\cap\dDomS$.
By Theorem \ref{maintheo}
and Theorem \ref{maintheoregpot} we have
\begin{align*}
\RcomS\cap\dDomS
&=\big(\RczomS\cap\dDomS\big)
\oplus_{\ltomS}
\big(\RcomS\cap\dDzomS\big),\\
\RczomS
&=\Gradgrad\htcom,\\
\dDzomS
&=\symRot\HoomT
\end{align*}
with linear and continuous potential operators.
Therefore, we can decompose
\begin{align*}
\bfM_{n}
=\bfM_{n,\mathsf{r}}+\bfM_{n,\mathsf{d}}
\in\big(\RczomS\cap\dDomS\big)
\oplus_{\ltomS}
\big(\RcomS\cap\dDzomS\big)
\end{align*}
with $\bfM_{n,\mathsf{r}}\in\Gradgrad\htcom\cap\dDomS$,
$\Rot\bfM_{n,\mathsf{d}}=\Rot\bfM_{n}$,
and $\bfM_{n,\mathsf{r}}=\Gradgrad u_{n}$, $u_{n}\in\htcom$,
as well as $\bfM_{n,\mathsf{d}}\in\RcomS\cap\symRot\HoomT$,
$\divDiv\bfM_{n,\mathsf{r}}=\divDiv\bfM_{n}$,
and $\bfM_{n,\mathsf{d}}=\symRot\bfE_{n}$, $\bfE_{n}\in\HoomT$,
and both $u_{n}$ and $\bfE_{n}$ depend continuously on $\bfM_{n}$, i.e.,
\begin{align*}
\norm{u_{n}}_{\htom}
\leq c\,\normLtom{\bfM_{n,\mathsf{r}}}
\leq c\,\normLtom{\bfM_{n}},\qquad
\normHoom{\bfE_{n}}
\leq c\,\normLtom{\bfM_{n,\mathsf{d}}}
\leq c\,\normLtom{\bfM_{n}}.
\end{align*}
By Rellich's selection theorem, there exist subsequences, again denoted by $(u_{n})$ and $(\bfE_{n})$,
such that $(u_{n})$ converges in $\hoom$ and $(\bfE_{n})$ converges in $\Ltom$.
Thus with $\bfM_{n,m}:=\bfM_{n}-\bfM_{m}$, 
and similarly for $\bfM_{n,m,\mathsf{r}}$, $\bfM_{n,m,\mathsf{d}}$, $u_{n,m}$, $\bfE_{n,m}$, we see
\begin{align*}
\normLtom{\bfM_{n,m,\mathsf{r}}}^2
&=\scpLtom{\bfM_{n,m,\mathsf{r}}}{\Gradgrad u_{n,m}}
=\scpltom{\divDiv\bfM_{n,m,\mathsf{r}}}{u_{n,m}}\\
&=\scpltom{\divDiv\bfM_{n,m}}{u_{n,m}}
\leq c\,\normltom{u_{n,m}},\\
\normLtom{\bfM_{n,m,\mathsf{d}}}^2
&=\scpLtom{\bfM_{n,m,\mathsf{d}}}{\symRot\bfE_{n,m}}
=\scpLtom{\Rot\bfM_{n,m,\mathsf{d}}}{\bfE_{n,m}}\\
&=\scpLtom{\Rot\bfM_{n,m}}{\bfE_{n,m}}
\leq c\,\normLtom{\bfE_{n,m}}.
\end{align*}
Hence, $(\bfM_{n})$ is a Cauchy sequence in $\ltomS$. So 
\begin{align*}
\RcomS\cap\dDomS\hookrightarrow\ltomS
\end{align*}
is compact. To show the second compact embedding, let 
$(\bfE_{n})\subset\symRomT\cap\DcomT$ be a bounded sequence.
By Theorem \ref{maintheo} and Theorem \ref{maintheoregpot} we have
\begin{align*}
\symRomT\cap\DcomT
&=\big(\symRzomT\cap\DcomT\big)
\oplus_{\ltomT}
\big(\symRomT\cap\DczomT\big),\\
\symRzomT
&=\devGrad\hoom,\\
\DczomT
&=\Rot\HocomS
\end{align*}
with linear and continuous potential operators.
Therefore, we can decompose
\begin{align*}
\bfE_{n}
=\bfE_{n,\mathsf{r}}+\bfE_{n,\mathsf{d}}
\in\big(\symRzomT\cap\DcomT\big)
\oplus_{\ltomT}
\big(\symRomT\cap\DczomT\big)
\end{align*}
with $\bfE_{n,\mathsf{r}}\in\devGrad\hoom\cap\DcomT$,
$\symRot\bfE_{n,\mathsf{d}}=\symRot\bfE_{n}$,
$\bfE_{n,\mathsf{r}}=\devGrad\vecv_{n}$, $\vecv_{n}\in\hoom$,
as well as $\bfE_{n,\mathsf{d}}\in\symRomT\cap\Rot\HocomS$,
$\Div\bfE_{n,\mathsf{r}}=\Div\bfE_{n}$,
and $\bfE_{n,\mathsf{d}}=\Rot\bfM_{n}$, $\bfM_{n}\in\HocomS$,
and both $\vecv_{n}$ and $\bfM_{n}$ depend continuously on $\bfE_{n}$, i.e.,
\begin{align*}
\normhoom{\vecv_{n}}
\leq c\,\normLtom{\bfE_{n,\mathsf{r}}}
\leq c\,\normLtom{\bfE_{n}},\qquad
\norm{\bfM_{n}}_{\Hoom}
\leq c\,\normLtom{\bfE_{n,\mathsf{d}}}
\leq c\,\normLtom{\bfE_{n}}.
\end{align*}
By Rellich's selection theorem,
there exist subsequences, again denoted by $(\vecv_{n})$ and $(\bfM_{n})$,
such that $(\vecv_{n})$ converges in $\ltom$ and $(\bfM_{n})$ converges in $\Ltom$.
Thus with $\bfE_{n,m}:=\bfE_{n}-\bfE_{m}$, 
and similarly for $\bfE_{n,m,\mathsf{r}}$, $\bfE_{n,m,\mathsf{d}}$, $\vecv_{n,m}$, $\bfM_{n,m}$, we see
\begin{align*}
\normLtom{\bfE_{n,m,\mathsf{r}}}^2
&=\scpLtom{\bfE_{n,m,\mathsf{r}}}{\devGrad\vecv_{n,m}}
=-\scpltom{\Div\bfE_{n,m,\mathsf{r}}}{\vecv_{n,m}}\\
&=-\scpltom{\Div\bfE_{n,m}}{\vecv_{n,m}}
\leq c\,\normltom{\vecv_{n,m}},\\
\normLtom{\bfE_{n,m,\mathsf{d}}}^2
&=\scpLtom{\bfE_{n,m,\mathsf{d}}}{\Rot\bfM_{n,m}}
=\scpLtom{\symRot\bfE_{n,m,\mathsf{d}}}{\bfM_{n,m}}\\
&=\scpLtom{\symRot\bfE_{n,m}}{\bfM_{n,m}}
\leq c\,\normLtom{\bfM_{n,m}}.
\end{align*}
Note, that here the symmetry of $\bfM_{n,m}$ is crucial.
Finally, $(\bfE_{n})$ is a Cauchy sequence in $\ltomT$. So 
\begin{align*}
\symRomT\cap\DcomT\hookrightarrow\ltomT
\end{align*}
is compact.
\end{proof}

For general topologies we will use a partition of unity argument. 
The next lemma, which we will prove in the Appendix, provides the necessary tools for this.

\begin{lem}
\mylabel{comfctlem}
Let $\varphi\in\Cic(\rt)$.
\begin{itemize}
\item[\bf(i)]
If $\bfM\in\Rcom$ resp. $\RcomS$ resp. $\RcomT$, 
then $\varphi\bfM\in\Rcom$ resp. $\RcomS$ resp. $\RcomT$ and 
\begin{align}
\mylabel{RotphiM}
\Rot(\varphi\bfM)
=\varphi\Rot\bfM+\grad\varphi\times\bfM.
\end{align}
\item[\bf(ii)]
If $\bfM\in\Rom$ resp. $\RomS$ resp. $\RomT$, 
then $\varphi\bfM\in\Rom$ resp. $\RomS$ resp. $\RomT$ and 
\eqref{RotphiM} holds.
\item[\bf(iii)]
If $\bfE\in\Dcom$ resp. $\DcomT$ resp. $\DcomS$, 
then $\varphi\bfE\in\Dcom$ resp. $\DcomT$ resp. $\DcomS$ and 
\begin{align}
\mylabel{DivphiM}
\Div(\varphi\bfE)
=\varphi\Div\bfE+\grad\varphi\cdot\bfE.
\end{align}
\item[\bf(iv)]
If $\bfE\in\Dom$ resp. $\DomT$ resp. $\DomS$, 
then $\varphi\bfE\in\Dom$ resp. $\DomT$ resp. $\DomS$ and 
\eqref{DivphiM} holds.
\item[\bf(v)]
If $\bfE\in\symRomT$, then $\varphi\bfE\in\symRomT$ and 
$$\symRot(\varphi\bfE)
=\varphi\symRot\bfE
+\sym(\grad\varphi\times\bfE).$$
\item[\bf(vi)]
If $\bfM\in\dDomS$, then $\varphi\bfM\in\dDzmoomS$ and 
$$\divDiv(\varphi\bfM)
=\varphi\divDiv\bfM
+2\grad\varphi\cdot\Div\bfM
+\tr(\bfM\Gradgrad\varphi).$$
\end{itemize}
By mollifying these formulas extend to $\varphi\in\mathring{\mathsf{C}}^{0,1}(\rt)$
resp. $\varphi\in\mathring{\mathsf{C}}^{1,1}(\rt)$.
\end{lem}

Here $\grad\varphi\,\times$ resp. $\grad\varphi\,\cdot$ 
is applied row-wise to a tensor $\bfM$ and we see
$\grad\varphi\cdot\bfM=\bfM\grad\varphi$
as well as $\grad\varphi\times\bfM=-\bfM\spn(\grad\varphi)$. Moreover, we introduce the new space
$$\dDzmoomS 
:=\set{\bfM\in\ltomS}{\divDiv\bfM\in\hmoom}.$$

Another auxiliary result required for the compactness proof is presented in the next lemma.

\begin{lem}
\label{Helmholtz1}
The regular (type) decomposition
$$\dDzmoomS
=\hocom\cdot\bfI
\dotplus
\dDzomS$$
holds. More precisely, for $\bfM\in\dDzmoomS$ 
there are unique $u\in\hocom$ and $\bfM_{0}\in\dDzomS$ such that $\bfM=u\,\bfI+\bfM_{0}$.
The scalar function $u\in\hocom$ is given as the unique solution of the Dirichlet-Poisson problem
$$\scpltom{\grad u}{\grad\varphi}
=-\scphmoom{\divDiv\bfM}{\varphi}
\quad\text{for all}\quad 
\varphi\in\hocom,$$
and the decomposition is continuous, more precisely there exists $c>0$, such that
\begin{align*}
\normhoom{u}
&\leq c\,\normhmoom{\divDiv\bfM},
&
\normltom{\bfM-u\,\bfI}
&\leq c\,\norm{\bfM}_{\dDzmoomS}.
\end{align*}
\end{lem}

\begin{proof}
The unique solution $u\in\hocom$ satisfies
$$\hmoom\ni\divDiv u\,\bfI=\div\grad u=\divDiv\bfM,$$
i.e., $\bfM_{0}:=\bfM-u\,\bfI\in\dDzomS$, which shows the decomposition.
Moreover, 
$$\normhoom{u}\leq(1+c_{\mathsf{g}}^2)^{\oh}\,\normhmoom{\divDiv\bfM}$$
shows, that $u$ depends continuously on $\bfM$ and hence also $\bfM_{0}$ since 
$$\normltom{\bfM_{0}}
\leq\normLtom{\bfM}+\normltom{u}
\leq(2+c_{\mathsf{g}}^2)^{\oh}\,\norm{\bfM}_{\dDzmoomS}.$$
Let $u\,\bfI\in\dDzomS$ with $u\in\hocom$.
Then $0=\divDiv u\,\bfI=\div\grad u=\Delta u$, yielding $u=0$.
Hence, the decomposition is direct, completing the proof. 
\end{proof}

\begin{lem}
\label{cptemblem}
The embeddings \eqref{crucialcptembS} and \eqref{crucialcptembT} are compact, i.e.,
\begin{align*}
\RcomS\cap\dDomS
\overset{\text{\sf cpt}}{\hookrightarrow}\ltomS,\qquad
\symRomT\cap\DcomT
\overset{\text{\sf cpt}}{\hookrightarrow}\ltomT.
\end{align*}
\end{lem}

\begin{proof}
Let $(U_{i})$ be an open covering of $\ovl{\om}$, such that $\om_{i}:=\om\cap U_{i}$
is topologically trivial for all $i$. As $\ovl{\om}$ is compact, there is a finite subcovering
denoted by $(U_{i})_{i=1,\ldots,I}$ with $I\in\n$. Let $(\varphi_{i})$ 
with $\varphi_{i}\in\cic(U_{i})$ be a partition of unity subordinate to $(U_{i})$.
Suppose $(\bfE_{n})\subset\symRomT\cap\DcomT$ is a bounded sequence.
Then $\bfE_{n}=\sum_{i=1}^{I}\varphi_{i}\bfE_{n}$
and $(\varphi_{i}\bfE_{n})\subset\symRomiT\cap\DcomiT$ is a bounded sequence
for all $i$ by Lemma \ref{comfctlem}. 
As $\om_{i}$ is topologically trivial, there exists a subsequence, again denoted by
$(\varphi_{i}\bfE_{n})$, which is a Cauchy sequence in $\lt(\om_{i})$ by Lemma \ref{cptemblemtrivtop}.
Picking successively subsequences yields that 
$(\varphi_{i}\bfE_{n})$ is a Cauchy sequence in $\lt(\om_{j})$ for all $j$.
Hence $(\bfE_{n})$ is a Cauchy sequence in $\Ltom$. So the second embedding of the lemma is compact.
Let $(\bfM_{n})\subset\RcomS\cap\dDomS$ be a bounded sequence.
Then $\bfM_{n}=\sum_{i=1}^{I}\varphi_{i}\bfM_{n}$
and $(\varphi_{i}\bfM_{n})\subset\RcomiS\cap\dDzmoomiS$ is a bounded sequence
for all $i$ by Lemma \ref{comfctlem} as $\norm{\Div\bfM_{n}}_{\Hmoom}\leq\normLtom{\bfM_{n}}$. 
Using Lemma \ref{Helmholtz1} we decompose 
$$\varphi_{i}\bfM_{n}=u_{i,n}\,\bfI+\bfM_{0,i,n}
\in\hoc(\om_{i})\cdot\bfI\dotplus\big(\RcomiS\cap\dDzomiS\big).$$
Moreover, $(u_{i,n})$ is bounded in $\hoc(\om_{i})$ 
and $(\bfM_{0,i,n})$ is bounded in $\RcomiS\cap\dDzomiS$.
By Rellich's selection theorem and Lemma \ref{cptemblemtrivtop}
as well as picking successively subsequences we get that 
$(\varphi_{i}\bfM_{n})$ is a Cauchy sequence in $\lt(\om_{j})$ for all $j$.
Hence $(\bfM_{n})$ is a Cauchy sequence in $\Ltom$, showing that the first embedding of the lemma 
is also compact and finishing the proof.
\end{proof}

Utilizing the crucial compact embeddings of Lemma \ref{cptemblem},
we can apply the functional analysis toolbox Section \ref{fuanasec}
to the (linear, densely defined, and closed `complex') operators $\Az$, $\Ao$, $\At$, $\Azs$, $\Aos$, $\Ats$.
In this general case the reduced operators are
\begin{align*}
\cAz=\Gradgradc:\htcom\subset\ltom
&\To\ovl{\Gradgrad\htcom},\\
\cAo=\RotcS:\RcomS\cap\ovl{\symRot\symRomT}\subset\ovl{\symRot\symRomT}
&\To\ovl{\Rot\RcomS},\\
\cAt=\DivcT:\DcomT\cap\ovl{\devGrad\hoom}\subset\ovl{\devGrad\hoom}
&\To\RTz^{\bot_{\ltom}},\\
\cAzs=\divDivS:\dDomS\cap\ovl{\Gradgrad\htcom}\subset\ovl{\Gradgrad\htcom}
&\To\ltom,\\
\cAos=\symRotT:\symRomT\cap\ovl{\Rot\RcomS}\subset\ovl{\Rot\RcomS}
&\To\ovl{\symRot\symRomT},\\
\cAts=-\devGrad:\hoom\cap\RTz^{\bot_{\ltom}}\subset\RTz^{\bot_{\ltom}}
&\To\ovl{\devGrad\hoom}
\end{align*}
as 
\begin{align*}
\ovl{\divDiv\dDomS}
&=\ovl{R(\divDivS)}
=N(\Gradgradc)^{\bot_{\ltom}}
=\{0\}^{\bot_{\ltom}}
=\ltom,\\
\ovl{\Div\DcomT}
&=\ovl{R(\DivcT)}
=N(\devGrad)^{\bot_{\ltom}}
=\RTz^{\bot_{\ltom}}.
\end{align*}
Note that by the compact embeddings of Lemma \ref{cptemblem}
all ranges are actually closed and we can skip the closure bars.
We obtain the following theorem.

\begin{theo} 
\label{maintheogendom}
It holds:
\begin{itemize}
\item[\bf(i)]
The ranges
\begin{align*}
R(\Gradgradc)&=\Gradgrad\htcom,\\
\ltom
=R(\divDivS)
&=\divDiv\dDomS
=\divDiv\big(\dDomS\cap\Gradgrad\htcom\big),\\
R(\RotcS)
&=\Rot\RcomS
=\Rot\big(\RcomS\cap\symRot\symRomT\big),\\
R(\symRotT)
&=\symRot\symRomT
=\symRot\big(\symRomT\cap\Rot\RcomS\big),\\
\RTz^{\bot_{\ltom}}
=R(\DivcT)
&=\Div\DcomT
=\Div\big(\DcomT\cap\devGrad\hoom\big),\\
R(\devGrad)
&=\devGrad\hoom
=\devGrad\big(\hoom\cap\RTz^{\bot_{\ltom}}\big)
\end{align*}
are closed. The more regular potentials on the right hand sides 
are uniquely determined and depend linearly and continuously on the data, see (v).
\item[\bf(ii)]
The cohomology groups
\begin{align*}
\harmdSom:=\RczomS\cap\dDzomS,\qquad
\harmnTom:=\DczomT\cap\symRzomT
\end{align*}
are finite dimensional and may be called
symmetric Dirichlet resp. deviatoric Neumann tensor fields.
\item[\bf(iii)]
The Hilbert complexes from Remark \ref{complexGradgrad}, i.e.,
$$\begin{CD}
\{0\} @> 0 >>
\htcom @> \Gradgradc >>
\RcomS @> \RotcS >>
\DcomT @> \DivcT >>
\ltom @> \pi_{\RTz} >>
\RTz
\end{CD}$$
and its adjoint
$$\begin{CD}
\{0\} @< 0 <<
\ltom @< \divDivS <<
\dDomS @< \symRotT << 
\symRomT @< -\devGrad <<
\hoom @< \iota_{\RTz} <<
\RTz
\end{CD},$$
are closed. They are also exact, if and only if
$\harmdSom=\{0\}$, $\harmnTom=\{0\}$.
The latter holds, if $\om$ is topologically trivial.
\item[\bf(iv)]
The Helmholtz type decompositions
\begin{align*}
\ltomS
&=\Gradgrad\htcom
\oplus_{\ltomS}
\dDzomS\\
&=\RczomS
\oplus_{\ltomS}
\symRot\symRomT\\
&=\Gradgrad\htcom
\oplus_{\ltomS}
\harmdSom
\oplus_{\ltomS}
\symRot\symRomT,\\
\ltomT
&=\Rot\RcomS
\oplus_{\ltomT}
\symRzomT\\
&=\DczomT
\oplus_{\ltomT}
\devGrad\hoom\\
&=\Rot\RcomS
\oplus_{\ltomT}
\harmnTom
\oplus_{\ltomT}
\devGrad\hoom
\end{align*}
are valid. 
\item[\bf(v)]
There exist positive constants $c_{\mathsf{Gg}}$, $c_{\mathsf{D}}$, $c_{\mathsf{R}}$,
such that the Friedrichs/Poincar\'e type estimates
\begin{align*}
\forall\,u&\in\htcom
&
\normltom{u}&\leq c_{\mathsf{Gg}}\,\normLtom{\Gradgrad u},\\
\forall\,\bfM&\in\dDomS\cap\Gradgrad\htcom
&
\normLtom{\bfM}&\leq c_{\mathsf{Gg}}\,\normltom{\divDiv\bfM},\\
\forall\,\bfE&\in\DcomT\cap\devGrad\hoom
&
\normLtom{\bfE}&\leq c_{\mathsf{D}}\,\normltom{\Div\bfE},\\
\forall\,\vecv&\in\hoom\cap\RTz^{\bot_{\ltom}}
&
\normltom{\vecv}&\leq c_{\mathsf{D}}\,\normLtom{\devGrad\vecv},\\
\forall\,\bfM&\in\RcomS\cap\symRot\symRomT
&
\normLtom{\bfM}&\leq c_{\mathsf{R}}\,\normLtom{\Rot\bfM},\\
\forall\,\bfE&\in\symRomT\cap\Rot\RcomS
&
\normLtom{\bfE}&\leq c_{\mathsf{R}}\,\normLtom{\symRot\bfE}
\end{align*}
hold\footnote{Note $\Rot\bfM=\dev\Rot\bfM$ for $\bfM\in\RomS$ and thus
for all $\bfM\in\RcomS\cap\symRot\symRomT$
$$\normLtom{\bfM}\leq c_{\mathsf{R}}\,\normLtom{\Rot\bfM}
=c_{\mathsf{R}}\,\normLtom{\dev\Rot\bfM}.$$}.
\item[\bf(vi)]
The inverse operators 
\begin{align*}
(\Gradgradc)^{-1}:
\Gradgrad\htcom&\To\htcom,\\
(\divDivS)^{-1}:
\ltom&\To\dDomS\cap\Gradgrad\htcom,\\
(\DivcT)^{-1}:
\RTz^{\bot_{\ltom}}&\To\DcomT\cap\devGrad\hoom,\\
(\devGrad)^{-1}:
\devGrad\hoom&\To\hoom\cap\RTz^{\bot_{\ltom}},\\
(\RotcS)^{-1}:
\Rot\RcomS&\To\RcomS\cap\symRot\symRomT,\\
(\symRotT)^{-1}:
\symRot\symRomT&\To\symRomT\cap\Rot\RcomS
\end{align*}
are continuous with norms $(1+c_{\mathsf{Gg}}^2)^{\oh}$ resp. $(1+c_{\mathsf{D}}^2)^{\oh}$, resp.
$(1+c_{\mathsf{R}}^2)^{\oh}$, and their modifications
\begin{align*}
(\Gradgradc)^{-1}:
\Gradgrad\htcom&\To\hocom\subset\ltom,\\
(\divDivS)^{-1}:
\ltom&\To\Gradgrad\htcom\subset\ltomS,\\
(\DivcT)^{-1}:
\RTz^{\bot_{\ltom}}&\To\devGrad\hoom\subset\ltomT,\\
(\devGrad)^{-1}:
\devGrad\hoom&\To\RTz^{\bot_{\ltom}}\subset\ltom,\\
(\RotcS)^{-1}:
\Rot\RcomS&\To\symRot\symRomT\subset\ltomS,\\
(\symRotT)^{-1}:
\symRot\symRomT&\To\Rot\RcomS\subset\ltomT
\end{align*}
are compact with norms $c_{\mathsf{Gg}}$, $c_{\mathsf{D}}$, resp. $c_{\mathsf{R}}$.
\end{itemize}
\end{theo}

We note
\begin{align}
\mylabel{RDDDirNeu}
\begin{split}
\RczomS
&=\Gradgrad\htcom
\oplus_{\ltomS}
\harmdSom,\\
\dDzomS
&=\symRot\symRomT
\oplus_{\ltomS}
\harmdSom,\\
\DczomT
&=\Rot\RcomS
\oplus_{\ltomT}
\harmnTom,\\
\symRzomT
&=\devGrad\hoom
\oplus_{\ltomT}
\harmnTom.
\end{split}
\end{align}

Finally, even parts of Theorem \ref{maintheoregpot}, Theorem \ref{regdecotheo}, and Corollary \ref{regdecotheocor}, 
extend to the general case, i.e.,
we have regular potentials and regular decompositions 
for bounded strong Lipschitz domains as well.

\begin{theo} 
\label{regdecotheogen}
The regular decompositions
\begin{align*}
\text{\bf(i)}&
&
\RcomS
&=\HocomS+\Gradgrad\htcom,\\
\text{\bf(ii)}&
&
\DcomT
&=\HocomT+\Rot\HocomS,\\
\text{\bf(iii)}&
&
\symRomT
&=\HoomT+\devGrad\hoom,\\
\text{\bf(iv)}&
&
\dDomS
&=\HtomS+\dDzomS
\hspace*{50mm}
\end{align*}
hold with linear and continuous (regular) decomposition resp. potential operators
\begin{align*}
\Pott_{\RcomS,\HocomS}
&:\RcomS\To\HocomS,\\
\Pott_{\RcomS,\htcom}
&:\RcomS\To\htcom,\\
\Pott_{\DcomT,\HocomT}
&:\DcomT\To\HocomT,\\
\Pott_{\DcomT,\HocomS}
&:\DcomT\To\HocomS,\\
\Pott_{\symRomT,\HoomT}
&:\symRomT\To\HoomT,\\
\Pott_{\symRomT,\hoom}
&:\symRomT\To\hoom,\\
\Pott_{\dDomS,\HtomS}
&:\dDomS\To\HtomS,\\
\Pott_{\dDomS,\dDzomS}
&:\dDomS\To\dDzomS.
\end{align*}
\end{theo}

\begin{proof}
As in the proof of Lemma \ref{cptemblem},
let $(U_{i})$ be an open covering of $\ovl{\om}$, such that $\om_{i}:=\om\cap U_{i}$
is topologically trivial for all $i$. As $\ovl{\om}$ is compact, there is a finite subcovering
denoted by $(U_{i})_{i=1,\ldots,I}$ with $I\in\n$. Let $(\varphi_{i})$ 
with $\varphi_{i}\in\cic(U_{i})$ be a partition of unity subordinate to $(U_{i})$
and let additionally $\phi_{i}\in\cic(U_{i})$ with $\phi_{i}|_{\supp\varphi_{i}}=1$.
To prove (i), suppose $\bfM\in\RcomS$. By Lemma \ref{comfctlem}
and Theorem \ref{regdecotheo} we have
$$\varphi_{i}\bfM
\in\RcomiS
=\HocomiS+\RczomiS
=\HocomiS+\Gradgrad\htcomi.$$
Hence, $\varphi_{i}\bfM=\bfM_{i}+\Gradgrad u_{i}$ with 
$\bfM_{i}\in\HocomiS$ and $u_{i}\in\htcomi$.
Let $\hat\bfM_{i}$ and $\hat u_{i}$ denote the extensions by zero of $\bfM_{i}$ and $u_{i}$.
Then $\hat\bfM_{i}\in\HocomS$ and $\hat u_{i}\in\htcom$. Thus
$$\bfM
=\sum_{i}\varphi_{i}\bfM
=\sum_{i}\hat\bfM_{i}+\Gradgrad\sum_{i}\hat u_{i}
\in\HocomS+\Gradgrad\htcom,$$
and all applied operations are continuous.
Similarly we proof (ii). 
To show (iii), let $\bfE\in\symRomT$.
By Lemma \ref{comfctlem} and Theorem \ref{regdecotheo} we have
$$\varphi_{i}\bfE
\in\symRomiT
=\HoomiT+\symRzomiT
=\HoomiT+\devGrad\ho(\om_{i}).$$
Hence, $\varphi_{i}\bfE=\bfE_{i}+\devGrad\vecv_{i}$ with 
$\bfE_{i}\in\HoomiT$ and $\vecv_{i}\in\ho(\om_{i})$.
In $\om_{i}$ we observe
\begin{align*}
\varphi_{i}\bfE
=\phi_{i}\varphi_{i}\bfE
&=\phi_{i}\bfE_{i}+\phi_{i}\devGrad\vecv_{i}\\
&=\phi_{i}\bfE_{i}-\dev(\vecv_{i}\cdot\grad^{\top}\phi_{i})+\devGrad(\phi_{i}\vecv_{i})
\in\HoomiT+\devGrad\ho(\om_{i}).
\end{align*}
Let $\hat\bfE_{i}$ and $\hat\vecv_{i}$ denote the extensions by zero 
of $\phi_{i}\bfE_{i}-\dev(\vecv_{i}\cdot\grad^{\top}\phi_{i})$ and $\phi_{i}\vecv_{i}$.
Then $\hat\bfE_{i}\in\HoomT$ and $\hat\vecv_{i}\in\hoom$. Thus
$$\bfE
=\sum_{i}\varphi_{i}\bfE
=\sum_{i}\hat\bfE_{i}+\devGrad\sum_{i}\hat\vecv_{i}
\in\HoomT+\devGrad\hoom,$$
and all applied operations are continuous.
To show (iv), let $\bfM\in\dDomS$.
Then $\divDiv\bfM\in\ltom$ and by Theorem \ref{maintheoregpot}
and Remark \ref{maintheoregpotrem} (ii) there is some $\tilde\bfM\in\HtomS$,
together with a linear and continuous potential operator,
with $\divDiv\tilde\bfM=\divDiv\bfM$. Therefore, we have $\bfM-\tilde\bfM\in\dDzomS$,
completing the proof.
\end{proof}

Applying $\RotcS$, $\DivcT$, and $\symRotT$, $\divDivS$
to the regular decompositions in Theorem \ref{regdecotheogen} we get the following regular potentials.

\begin{theo} 
\label{maintheoregpotgen}
It holds
\begin{align*}
\text{\bf(i)}&
&
R(\RotcS)
=\Rot\RcomS
&=\Rot\HocomS,\\
\text{\bf(ii)}&
&
\RTz^{\bot_{\ltom}}
=R(\DivcT)
=\Div\DcomT
&=\Div\HocomT,\\
\text{\bf(iii)}&
&
R(\symRotT)
=\symRot\symRomT
&=\symRot\HoomT,\\
\text{\bf(iv)}&
&
\ltom 
=R(\divDivS)
=\divDiv\dDomS
&=\divDiv\HtomS
\end{align*}
with corresponding linear and continuous (regular) potential operators
(on the right hand sides).
\end{theo}

Using Theorem \ref{maintheogendom}, 
canonical linear and continuous regular potential operators in the latter theorem are given by
\begin{align}
\mylabel{tildePdef}
\begin{split}
\Pott_{\RotcS}:=\Pott_{\RcomS,\HocomS}(\RotcS)^{-1}
&:\Rot\RcomS\To\HocomS,\\
\Pott_{\DivcT}:=\Pott_{\DcomT,\HocomT}(\DivcT)^{-1}
&:\RTz^{\bot_{\ltom}}\To\HocomT,\\
\Pott_{\symRotT}:=\Pott_{\symRomT,\HoomT}(\symRotT)^{-1}
&:\symRot\symRomT\To\HoomT,\\
\Pott_{\divDivS}:=\Pott_{\dDomS,\HtomS}(\divDivS)^{-1}
&:\ltom\To\HtomS.
\end{split}
\end{align}
We get the following direct regular decompositions.

\begin{cor} 
\label{regdecotheogencor}
The direct regular decompositions
\begin{align*}
\RcomS
&=\Pott_{\RotcS}\Rot\RcomS\dotplus\RczomS,\\
\DcomT
&=\Pott_{\DivcT}\RTz^{\bot_{\ltom}}\dotplus\DczomT,\\
\symRomT
&=\Pott_{\symRotT}\symRot\symRomT\dotplus\symRzomT,\\
\dDomS
&=\Pott_{\divDivS}\ltom\dotplus\dDzomS
\end{align*}
hold. Moreover,
\begin{align*}
\Pott_{\RotcS}\Rot\RcomS
&\subset\HocomS,
&
\Pott_{\symRotT}\symRot\symRomT
&\subset\HoomT,\\
\Pott_{\DivcT}\RTz^{\bot_{\ltom}}
&\subset\HocomT,
&
\Pott_{\divDivS}\ltom
&\subset\HtomS.
\end{align*}
\end{cor}

Note that the second summands on the right hand sides
may be further decomposed by \eqref{RDDDirNeu}, Theorem \ref{maintheoregpotgen}, and \eqref{tildePdef}.

\begin{proof}
For $\bfM\in\RczomS\cap\Pott_{\RotcS}\Rot\RcomS$ we have
$\bfM=\Pott_{\RotcS}\bfN$ with some $\bfN\in\Rot\RcomS$.
Thus $0=\Rot\bfM=\bfN$ showing $\bfM=0$ and hence the directness of the first regular decomposition.
The other assertions follow similarly.
\end{proof}

\begin{rem} 
\label{regdecotheogenrem}
While the results about the regular potentials in Theorem \ref{maintheoregpotgen} hold 
in full generality for all operators, one may wonder 
that the regular decompositions from Theorem \ref{regdecotheogen}
hold in full generality only for (i)-(iii), but not for (iv), i.e., we just have in (iv) 
$$\dDomS
=\HtomS+\dDzomS
\supset\HtomS+\symRot\HoomT.$$
The reason for the failure of the partition of unity argument 
from the proof of Theorem \ref{regdecotheogen} is the following:
Let $\bfM\in\dDomS$. By Lemma \ref{comfctlem} (vi) 
we just get $\varphi_{i}\bfM\in\dDzmoomiS$,
see also the proof of Lemma \ref{cptemblem}.
Using Lemma \ref{Helmholtz1} and Theorem \ref{regdecotheo} we can decompose 
\begin{align*}
\varphi_{i}\bfM=u_{i}\,\bfI+\symRot\bfE_{i}
\in\hoc(\om_{i})\cdot\bfI\dotplus\symRot\HoomiT
\end{align*}
as $\dDzomiS=\symRot\HoomiT$.
In $\om_{i}$ we observe
\begin{align*}
\varphi_{i}\bfM
=\phi_{i}\varphi_{i}\bfM
&=\phi_{i}u_{i}\,\bfI+\phi_{i}\symRot\bfE_{i}\\
&=\phi_{i}u_{i}\,\bfI-\sym(\grad\phi_{i}\times\bfE_{i})+\symRot(\phi_{i}\bfE_{i})
\in\HoomiS+\symRot\HoomiT.
\end{align*}
Let $\hat\bfM_{i}$ and $\hat\bfE_{i}$ denote the extensions by zero 
of $\phi_{i}u_{i}\,\bfI-\sym(\grad\phi_{i}\times\bfE_{i})$ and $\phi_{i}\bfE_{i}$.
Then $\hat\bfM_{i}\in\HoomS$ and $\hat\bfE_{i}\in\HoomT$ and thus
$$\bfM
=\sum_{i}\varphi_{i}\bfM
=\sum_{i}\hat\bfM_{i}+\symRot\sum_{i}\hat\bfE_{i}
\in\HoomS+\symRot\HoomT,$$
and all applied operations are continuous. Therefore, we obtain
$$\HtomS+\symRot\HoomT
\subset\HtomS+\dDzomS
=\dDomS
\subset\HoomS+\symRot\HoomT.$$
So we have lost one Sobolev order in the summand $\HoomS$.
\end{rem}

\section{Application to Biharmonic Problems}
\mylabel{applsec}

By $\Delta^2=\divDiv\Gradgrad$,
a standard (primal) variational formulation of \eqref{primal} in $\rt$ reads as follows: 
For given $f\in\hmtom$, find $u\in\htcom$ such that
\begin{equation} 
\label{var_primal}
\scpLtom{\Gradgrad u}{\Gradgrad\phi}=\scp{f}{\phi}_{\hmtom}
\quad\text{for all}\quad
\phi\in\htcom.
\end{equation}
Existence, uniqueness, and continuous dependence on $f$ of a solution to \eqref{var_primal} is guaranteed 
by the theorem of Lax-Milgram, see, e.g., \cite{necas:12,lions:72}
or Lemma \ref{formulasone}. Note that then 
$$\bfM:=\Gradgrad u
\in\RczomS\ominus_{\ltomS}\harmdSom\subset\ltomS$$
with $\divDiv\bfM=f\in\hmtom$. In other words the operator
\begin{align}
\mylabel{surjdivDiv}
\divDiv:\ltomS\to\hmtom
\end{align}
is surjective and 
\begin{align}
\mylabel{bijdivDiv}
\divDiv:\RczomS\ominus_{\ltomS}\harmdSom\to\hmtom
\end{align}
is bijective and even a topological isomorphism by the bounded inverse theorem.
For our decomposition result we need the following variant 
of the Hilbert complex from Theorem \ref{maintheogendom}. 
$$\begin{CD}
\RTz @> \iota_{\RTz} >>
\hoom @> -\devGrad >>
\symRomT @> \symRotT >> 
\dDzmoomS @> \divDivS >>
\hmoom @> 0 >>
\{0\},
\end{CD}$$
where we recall $\dDzmoomS$ from Lemma \ref{Helmholtz1}.
This is obviously also a closed Hilbert complex
as $\divDiv:\dDzmoomS\to\hmoom$ is surjective as well by \eqref{surjdivDiv}. 
Observe that
$$\HoomS\subset\dDzmoomS\subset\ltomS.$$
For right-hand sides $f\in\hmoom$ we consider the following mixed variational problem 
for $u$ and the Hessian $\bfM$ of $u$:
Find $\bfM\in\dDzmoomS$ and $u\in\hocom$ such that
\begin{align}
\mylabel{var_mixed_final_one}
\scpLtom{\bfM}{\bfPsi}
+\scphmoom{u}{\divDiv\bfPsi}
&=0
& 
&\text{for all}
&
\bfPsi&\in\dDzmoomS,\\
\mylabel{var_mixed_final_two}
\scphmoom{\divDiv\bfM}{\psi}
&=-\scphmoom{f}{\psi}
&
&\text{for all} 
&
\psi&\in\hocom.
\end{align}
The first row and the second row of this mixed problem are variational formulations 
of \eqref{defM} and \eqref{divDivM}, respectively.
We recall the following two results related to these mixed problems from \cite{zulehner-2016-02}.

\begin{theo} 
\label{mainvartheo}
Let $f\in\hmoom$. Then:
\begin{itemize}
\item[\bf(i)]
Problem \eqref{var_mixed_final_one}-\eqref{var_mixed_final_two} 
is a well-posed saddle point problem.
\item[\bf(ii)]
The variational problems \eqref{var_primal} and 
\eqref{var_mixed_final_one}-\eqref{var_mixed_final_two} are equivalent, i.e., 
if $u\in\htcom$ solves \eqref{var_primal}, 
then $\bfM=-\Gradgrad u$ lies in $\dDzmoomS$
and $(\bfM,u)$ solves \eqref{var_mixed_final_one}-\eqref{var_mixed_final_two}.
And, vice versa, if $(\bfM,u)\in\dDzmoomS\times\hocom$ solves 
\eqref{var_mixed_final_one}-\eqref{var_mixed_final_two}, then $u\in\htcom$
with $\Gradgrad u=-\bfM$ and $u$ solves \eqref{var_primal}.
\end{itemize}
\end{theo}

Although only two-dimensional biharmonic problems were considered in \cite{zulehner-2016-02}, 
the proof of the latter theorem is completely identical 
for the three-dimensional case. The same holds for Lemma \ref{Helmholtz1}.

\begin{proof}
To show (i), we first note that $(\bfPhi,\bfPsi)\mapsto\scpLtom{\bfPhi}{\bfPsi}$
is coercive over the kernel of \eqref{var_mixed_final_two}, i.e.,
for $\bfPhi\in\dDzomS$ we have 
$\scpLtom{\bfPhi}{\bfPhi}=\normLtom{\bfPhi}^2=\norm{\bfPhi}_{\dDomS}^2=\norm{\bfPhi}_{\dDzmoomS}^2$.
Moreover, the inf-sup-condition holds, as
\begin{align*}
&\qquad\inf_{0\neq\varphi\in\hocom}
\sup_{0\neq\bfPhi\in\dDzmoomS}
\frac{\scphmoom{\varphi}{\divDiv\bfPhi}}{\normltom{\grad\varphi}\norm{\bfPhi}_{\dDzmoomS}}\\
&\geq\inf_{0\neq\varphi\in\hocom}
\frac{-\scphmoom{\varphi}{\divDiv(\varphi\,\bfI)}}{\normltom{\grad\varphi}\norm{\varphi\,\bfI}_{\dDzmoomS}}
=\inf_{0\neq\varphi\in\hocom}
\frac{\normltom{\grad\varphi}}{\big(\normLtom{\varphi\,\bfI}^2
+\normhmoom{\divDiv(\varphi\,\bfI)}^2\big)^{\oh}}\\
&=\inf_{0\neq\varphi\in\hocom}
\frac{\normltom{\grad\varphi}}{\big(3\normltom{\varphi}^2
+\normltom{\grad\varphi}^2\big)^{\oh}}
\geq(3\,c_{\mathsf{g}}^2+1)^{-\oh}
\end{align*}
by choosing $\bfPhi:=-\varphi\,\bfI\in\hocom\cdot\bfI\subset\dDzmoomS$ and observing 
\begin{align*}
-\scphmoom{\varphi}{\divDiv(\varphi\,\bfI)}
&=-\scphmoom{\varphi}{\div\grad\varphi}
=\normltom{\grad\varphi}^2,\\
\normhmoom{\divDiv(\varphi\,\bfI)}
&=\sup_{0\neq\phi\in\hocom}
\frac{\scphmoom{\phi}{\div\grad\varphi}}{\normltom{\grad\phi}}\\
&=\sup_{0\neq\phi\in\hocom}
\frac{\scpltom{\grad\phi}{\grad\varphi}}{\normltom{\grad\phi}}
=\normltom{\grad\varphi}.
\end{align*}
Note that both the primal problem \eqref{var_primal} and the mixed problem 
\eqref{var_mixed_final_one}-\eqref{var_mixed_final_two} are well-posed. 
So, it suffices to show the first part of (ii) only. The reverse direction follows then automatically.
Let $u\in\htcom$ solve \eqref{var_primal}.
Then $\bfM:=-\Gradgrad u\in\dDzmoomS$ with $\divDiv\bfM=-f$ 
in $\hmtom$ and hence in $\hmoom$. Thus \eqref{var_mixed_final_two} holds. 
Moreover, for $\bfPsi\in\dDzmoomS$ we see
$$\scpLtom{\bfM}{\bfPsi}
=-\scpLtom{\Gradgrad u}{\bfPsi}
=-\scp{u}{\divDiv\bfPsi}_{\hmtom}
=-\scp{u}{\divDiv\bfPsi}_{\hmoom}$$
and hence \eqref{var_mixed_final_one} is true.
Therefore, $(\bfM,u)$ solves \eqref{var_mixed_final_one}-\eqref{var_mixed_final_two}.
\end{proof}

\begin{rem}
\mylabel{{mainvartheo}}
For convenience of the reader, we give additionally a proof of the other direction as well:
If $(\bfM,u)$ in $\dDzmoomS\times\hocom$ solves 
\eqref{var_mixed_final_one}-\eqref{var_mixed_final_two}, then 
$\divDiv\bfM=-f$ in $\hmoom$ and \eqref{var_mixed_final_one} holds.
Especially, \eqref{var_mixed_final_one} holds for 
$\bfPsi\in\HtomS\subset\HoomS\subset\dDzmoomS$, i.e.,
\begin{align}
\mylabel{uhtcom}
-\scpLtom{\bfM}{\bfPsi}
=\scphmoom{u}{\divDiv\bfPsi}
=\scpltom{u}{\divDiv\bfPsi}.
\end{align}
But then \eqref{uhtcom} holds for all $\bfPsi\in\Htom$ as $\sym\bfPsi\in\HtomS$ and 
\begin{align}
\mylabel{uhtcomfull}
-\scpLtom{\bfM}{\bfPsi}
=-\scpLtom{\bfM}{\sym\bfPsi}
=\scpltom{u}{\divDiv\sym\bfPsi}
=\scpltom{u}{\divDiv\bfPsi},
\end{align}
since $\divDiv\skw\bfPsi=0$ by
$$\scpltom{\divDiv\skw\bfPsi}{\phi}
=\scpltom{\skw\bfPsi}{\Gradgrad\phi}
=0$$
for all $\phi\in\cicom$. \eqref{uhtcomfull} yields that $u\in\htcom$
with $\Gradgrad u=-\bfM$. Finally, for all $\phi\in\htcom$
$$\scpLtom{\Gradgrad u}{\Gradgrad\phi}
=-\scpLtom{\bfM}{\Gradgrad\phi}
=-\scp{\divDiv\bfM}{\phi}_{\hmtom}
=\scp{f}{\phi}_{\hmtom},$$
showing that $u\in\htcom$ solves \eqref{var_primal}.
\end{rem}

We note that the decomposition of $\dDzmoomS$ in Lemma \ref{Helmholtz1} 
is different to the Helmholtz type decomposition of the larger space $\ltomS$ 
in Theorem \ref{maintheo} and Theorem \ref{maintheogendom} 
and does not involve the Hessian of scalar functions in $\htcom$.
Using the decomposition of $\dDzmoomS$ in Lemma \ref{Helmholtz1},
we have the following decomposition result for the biharmonic problem.
Let $(\bfM,u)\in\dDzmoomS\times\hocom$ be the unique solution
of \eqref{var_mixed_final_one}-\eqref{var_mixed_final_two}.
Using Lemma \ref{Helmholtz1} we have the following direct decompositions
for $\bfM,\bfPsi\in\dDzmoomS$
\begin{align*}
\bfM=p\,\bfI+\bfM_{0},\quad
\bfPsi=\varphi\,\bfI+\bfPsi_{0},\qquad
p,\varphi\in\hocom,\quad
\bfM_{0},\bfPsi_{0}\in\dDzomS.
\end{align*}
This allows to rewrite \eqref{var_mixed_final_one}-\eqref{var_mixed_final_two}
equivalently in terms of $(p,\bfM_{0},u)$ and for all $(\varphi,\bfPsi_{0},\psi)$, i.e.,
\begin{align*}
\scpLtom{p\,\bfI}{\varphi\,\bfI}
+\scpLtom{\bfM_{0}}{\bfPsi_{0}}
+\scpLtom{p\,\bfI}{\bfPsi_{0}}
+\scpLtom{\bfM_{0}}{\varphi\,\bfI}
+\scphmoom{u}{\divDiv(\varphi\,\bfI)}
&=0,\\
\scphmoom{\divDiv(p\,\bfI)}{\psi}
&=-\scphmoom{f}{\psi}
\intertext{or equivalently}
\scpltom{\grad u}{\grad\varphi}
+3\scpltom{p}{\varphi}
+\scpLtom{\bfM_{0}}{\bfPsi_{0}}
+\scpltom{p}{\tr\bfPsi_{0}}
+\scpltom{\tr\bfM_{0}}{\varphi}
&=0,\\
\scpltom{\grad p}{\grad\psi}
&=-\scphmoom{f}{\psi},
\intertext{which leads to the equivalent system}
\scpltom{\grad u}{\grad\varphi}
+3\scpltom{p}{\varphi}
+\scpltom{\tr\bfM_{0}}{\varphi}
&=0,\\
\scpLtom{\bfM_{0}}{\bfPsi_{0}}
+\scpltom{p}{\tr\bfPsi_{0}}
&=0,\\
\scpltom{\grad p}{\grad\psi}
&=-\scphmoom{f}{\psi}.
\end{align*}

\begin{theo}
\label{almostfinalddzvarformapp}
The variational problem \eqref{var_mixed_final_one}-\eqref{var_mixed_final_two} 
is equivalent to the following well-posed and 
uniquely solvable variational problem.
For $f\in\hmoom$ find $p\in\hocom$, $\bfM_{0}\in\dDzomS$, and $u\in\hocom$ such that
\begin{align}
\mylabel{finalddzvareqone}
\scpltom{\grad u}{\grad\varphi}
+\scpltom{\tr\bfM_{0}}{\varphi}
+3\scpltom{p}{\varphi}
&=0,\\
\mylabel{finalddzvareqtwo}
\scpLtom{\bfM_{0}}{\bfPsi_{0}}
+\scpltom{p}{\tr\bfPsi_{0}}
&=0,\\
\mylabel{finalddzvareqthree}
\scpltom{\grad p}{\grad\psi}
&=-\scphmoom{f}{\psi}
\end{align}
for all $\psi\in\hocom$, $\bfPsi_{0}\in\dDzomS$, and $\varphi\in\hocom$. 
Moreover, the unique solution $(\bfM,u)$
of \eqref{var_mixed_final_one}-\eqref{var_mixed_final_two}
is given by $\bfM:=p\,\bfI+\bfM_{0}$ and $u$
for the unique solution $(p,\bfM_{0},u)$ of \eqref{finalddzvareqone}-\eqref{finalddzvareqthree}.
\end{theo}

If $\om$ is additionally topologically trivial, 
then by Theorem \ref{maintheo} or Theorem \ref{maintheogendom} 
$$\dDzomS
=\symRot\symRomT
=\symRot\big(\symRomT\cap\DczomT\big)$$
and we obtain the following result.

\begin{theo}
\label{almostfinalvarformapp}
Let $\om$ be additionally topologically trivial.
The variational problem \eqref{var_mixed_final_one}-\eqref{var_mixed_final_two} 
is equivalent to the following well-posed and 
uniquely solvable variational problem.
For $f\in\hmoom$ find $p\in\hocom$, $\bfE\in\symRomT\cap\DczomT$, and $u\in\hocom$ such that
\begin{align}
\mylabel{finalvareqone}
\scpltom{\grad u}{\grad\varphi}
+\scpltom{\tr\symRot\bfE}{\varphi}
+3\scpltom{p}{\varphi}
&=0,\\
\mylabel{finalvareqtwo}
\scpLtom{\symRot\bfE}{\symRot\bfPhi}
+\scpltom{p}{\tr\symRot\bfPhi}
&=0,\\
\mylabel{finalvareqthree}
\scpltom{\grad p}{\grad\psi}
&=-\scphmoom{f}{\psi}
\end{align}
for all $\psi\in\hocom$, $\bfPhi\in\symRomT\cap\DczomT$, 
and $\varphi\in\hocom$. 
Moreover, the unique solution $(\bfM,u)$
of \eqref{var_mixed_final_one}-\eqref{var_mixed_final_two}
is given by $\bfM:=p\,\bfI+\symRot\bfE$ and $u$
for the unique solution $(p,\bfE,u)$ of \eqref{finalvareqone}-\eqref{finalvareqthree}.
\end{theo}

Note that, e.g., $\scpltom{\tr\symRot\bfE}{\varphi}=\scpLtom{\symRot\bfE}{\varphi\,\bfI}$
and $3\scpltom{p}{\varphi}=\scpLtom{p\,\bfI}{\varphi\,\bfI}$.

\begin{proof}
\eqref{var_mixed_final_one}-\eqref{var_mixed_final_two}
is equivalent to \eqref{finalddzvareqone}-\eqref{finalddzvareqthree}
and hence also to \eqref{finalvareqone}-\eqref{finalvareqthree},
if the latter system is well-posed.
By Theorem \ref{maintheo} or Theorem \ref{maintheogendom} the bilinear form  
$\scpLtom{\symRot\,\cdot\,}{\symRot\,\cdot\,}$ is coercive over $\symRomT\cap\DczomT$,
which shows the consecutive unique solvability of \eqref{finalvareqone}-\eqref{finalvareqthree}.
\end{proof}

The three problems in the previous theorem are weak formulations 
of the following three second-order problems in strong form. 
A Dirichlet-Poisson problem for the auxiliary scalar function $p$
$$\Delta p=f\quad\text{in }\om,\qquad 
p=0\quad\text{on }\ga,$$
a second-order Neumann type $\Rot\symRot$-$\Div$-system for the auxiliary tensor field $\bfE$
\begin{align*}
\tr\bfE
&=0,
&
\Rot\symRot\bfE
&=-\Rot(p\,\bfI)=\spn\grad p,
&
\Div\bfE
&=0
&
&\text{in }\om,\\
&&
n\times\symRot\bfE
&=-n\times p\,\bfI=p\spn n=0,
&
\bfE\,n
&=0
&
&\text{on }\ga,
\end{align*}
and, finally, a Dirichlet-Poisson problem for the original scalar function $u$
$$\Delta u=3p+\tr\symRot\bfE=\tr(p\,\bfI+\symRot\bfE)\quad\text{in }\om,\qquad 
u=0\quad\text{on }\ga.$$
In other words,
the system \eqref{finalvareqone}-\eqref{finalvareqthree} has triangular structure
$$\begin{bmatrix}
3 & \tr\symRotT & -\mathring{\Delta} \\
\RotcS(\,\cdot\,\bfI) & \RotcS\symRotT &  0 \\
-\mathring{\Delta} & 0 & 0 
\end{bmatrix}
\begin{bmatrix}
p \\[1.3ex]
\bfE \\[1.3ex]
u
\end{bmatrix}
=
\begin{bmatrix}
0 \\[1.3ex]
0 \\[1.3ex]
-f
\end{bmatrix}$$
with $(\tr\symRotT)^{*}=\RotcS(\,\cdot\,\bfI)$
and $\mathring{\Delta}=\div\gradc$.
Indeed, $\bfE\in\symRomT\cap\DczomT$ with
$$\scpLtom{\symRot\bfE}{\symRot\bfPhi}
+\scpltom{p}{\tr\symRot\bfPhi}=0$$
for all $\bfPhi\in\symRomT\cap\DczomT$
is equivalent to $\bfE\in\symRomT\cap\DczomT$ and
\begin{align}
\mylabel{dualsymRotapp}
\scpLtom{\symRot\bfE+p\,\bfI}{\symRot\bfPhi}=0
\end{align}
for all $\bfPhi\in\symRomT$
as by Theorem \ref{maintheo} 
\begin{align}
\mylabel{rangesymRotapp}
\symRot\big(\symRomT\cap\DczomT\big)
=\symRot\symRomT.
\end{align}
Now \eqref{dualsymRotapp} shows that 
$$\symRot\bfE+p\,\bfI\in R(\symRotT)^{\bot_{\Ltom}}=N(\symRotT^{*})=N(\RotcS)=\RczomS,$$
especially $\Rot(\symRot\bfE+p\,\bfI)=0$ in $\om$ and $n\times(\symRot\bfE+p\,\bfI)=0$ on $\ga$. 

Finally, we want to get rid of the complicated space
$\symRomT\cap\DczomT$ in the variational formulation in Theorem \ref{almostfinalvarformapp}.
For a given $p\in\hocom$ the part \eqref{finalvareqtwo} of \eqref{finalvareqone}-\eqref{finalvareqthree}, i.e.,
find a tensor field $\bfE\in\symRomT\cap\DczomT$ such that
\begin{align}
\mylabel{varsymRotpartone}
\scpLtom{\symRot\bfE}{\symRot\bfPhi}
+\scpltom{p}{\tr\symRot\bfPhi}=0
\end{align}
for all $\bfPhi\in\symRomT\cap\DczomT$,
has also a saddle point structure.
By Theorem \ref{maintheo} we have \eqref{rangesymRotapp} as well as
$$\DczomT
=N(\DivcT)
=R(\DivcT^{*})^{\bot_{\ltomT}}
=R(\devGrad)^{\bot_{\ltomT}}
=\big(\devGrad\big(\hoom\cap\RTz^{\bot_{\ltom}}\big)\big)^{\bot_{\ltomT}}.$$
Hence \eqref{varsymRotpartone} is equivalent to 
find $\bfE\in\symRomT$ such that
\begin{align}
\mylabel{varsymRotparttwoone}
\scpLtom{\symRot\bfE}{\symRot\bfPhi}
+\scpltom{p}{\tr\symRot\bfPhi}
&=0,\\
\mylabel{varsymRotparttwotwo}
\scpLtom{\bfE}{\devGrad\theta}
&=0
\end{align}
for all $\bfPhi\in\symRomT$ and $\theta\in\hoom\cap\RTz^{\bot_{\ltom}}$.
Observe that 
$$(\bfE,\vecv):=(\bfE,0)\in\symRomT\times\big(\hoom\cap\RTz^{\bot_{\ltom}}\big)$$
solves the modified variational system
\begin{align}
\mylabel{varsymRotpartthreeone}
\scpLtom{\symRot\bfE}{\symRot\bfPhi}
+\scpLtom{\bfPhi}{\devGrad\vecv}
&=-\scpltom{p}{\tr\symRot\bfPhi},\\
\mylabel{varsymRotpartthreetwo}
\scpLtom{\bfE}{\devGrad\theta}
&=0
\end{align}
for all $\bfPhi\in\symRomT$ and $\theta\in\hoom\cap\RTz^{\bot_{\ltom}}$.
On the other hand, any solution
$$(\bfE,\vecv)\in\symRomT\times\big(\hoom\cap\RTz^{\bot_{\ltom}}\big)$$
of \eqref{varsymRotpartthreeone}-\eqref{varsymRotpartthreetwo}
satisfies $\vecv=0$, as \eqref{varsymRotpartthreeone} tested with 
$$\bfPhi:=\devGrad\vecv\in\devGrad\hoom=\symRzomT$$
shows $\devGrad\vecv=0$ and thus $\vecv\in\RTz$ by Lemma \ref{kerneldevGrad},
yielding $\vecv=0$. 
Note that \eqref{varsymRotpartthreeone}-\eqref{varsymRotpartthreetwo} 
has the saddle point structure
$$\begin{bmatrix}
\mathring{\Rot}_{\bbS}\symRotT & \devGrad \\
-\DivcT & 0
\end{bmatrix}
\begin{bmatrix}
\bfE \\[1.3ex]
\vecv
\end{bmatrix}
=
\begin{bmatrix}
-\mathring{\Rot}_{\bbS}(v\,\bfI) \\[1.3ex]
0
\end{bmatrix},\qquad
(\devGrad)^{*}=-\DivcT.$$

We obtain the following final result.

\begin{theo}
\label{finalvarformapp}
Let $\om$ be additionally topologically trivial.
The variational problem \eqref{finalvareqone}-\eqref{finalvareqthree}
is equivalent to the following well-posed and 
uniquely solvable variational system.
For $f\in\hmoom$ find  $p\in\hocom$, $\bfE\in\symRomT$, 
$\vecv\in\hoom\cap\RTz^{\bot_{\ltom}}$, and $u\in\hocom$ such that
\begin{align}
\mylabel{relfinalvareqone}
\scpltom{\grad u}{\grad\varphi}
+\scpltom{\tr\symRot\bfE}{\varphi}
+3\scpltom{p}{\varphi}
&=0,\\
\mylabel{relfinalvareqtwo}
\scpLtom{\symRot\bfE}{\symRot\bfPhi}
+\scpLtom{\bfPhi}{\devGrad\vecv}
+\scpltom{p}{\tr\symRot\bfPhi}
&=0,\\
\mylabel{relfinalvareqthree}
\scpLtom{\bfE}{\devGrad\theta}
&=0,\\
\mylabel{relfinalvareqfour}
\scpltom{\grad p}{\grad\psi}
&=-\scphmoom{f}{\psi}
\end{align}
for all $\psi\in\hocom$, $\bfPhi\in\symRomT$, 
$\theta\in\hoom\cap\RTz^{\bot_{\ltom}}$, and $\varphi\in\hocom$. 
Moreover, the unique solution $(p,\bfE,\vecv,u)$ of \eqref{relfinalvareqone}-\eqref{relfinalvareqfour}
satisfies $\vecv=0$ and $(p,\bfE,u)$ is the unique solution of \eqref{finalvareqone}-\eqref{finalvareqthree}.
\end{theo}

Note that the system \eqref{relfinalvareqone}-\eqref{relfinalvareqfour} 
has the block triangular saddle point structure
\begin{align}
\mylabel{dsadpointmatrix}
\begin{bmatrix}
3 & \tr\symRotT & 0 & -\mathring{\Delta} \\
\RotcS(\,\cdot\,\bfI) & \RotcS\symRotT & \devGrad & 0 \\
0 & -\DivcT & 0  & 0 \\
-\mathring{\Delta} & 0 & 0 & 0
\end{bmatrix}
\begin{bmatrix}
p \\[1.3ex]
\bfE \\[1.3ex]
\vecv \\[1.3ex]
u
\end{bmatrix}
=
\begin{bmatrix}
0 \\[1.3ex]
0 \\[1.3ex]
0 \\[1.3ex]
-f
\end{bmatrix}.
\end{align}
with $(\tr\symRotT)^{*}=\RotcS(\,\cdot\,\bfI)$
and $(\devGrad)^{*}=-\DivcT$.

\begin{proof}
We only have to show well-posedness of the partial system \eqref{relfinalvareqtwo}-\eqref{relfinalvareqthree}.
First note that by Theorem \ref{maintheo} the bilinear form  
$\scpLtom{\symRot\,\cdot\,}{\symRot\,\cdot\,}$ is coercive over $\symRomT\cap\DczomT$,
which equals the kernel of \eqref{relfinalvareqthree}.
Indeed it follows from \eqref{relfinalvareqthree} that
$$\bfE\in\big(\devGrad\big(\hoom\cap\RTz^{\bot_{\ltom}}\big)\big)^{\bot_{\ltomT}}
=\DczomT.$$
Moreover, the inf-sup-condition is satisfied as by picking
for fixed $0\neq\theta\in\hoom\cap\RTz^{\bot_{\ltom}}$
the tensor $\bfPhi:=\devGrad\theta\in\devGrad\hoom=\symRzomT$ we have
\begin{align*}
\inf_{\substack{0\neq\theta\in\hoom,\\\theta\bot_{\ltom}\RTz}}
\;\;
\sup_{\bfPhi\in\symRomT}
\;
\frac{\scpLtom{\bfPhi}{\devGrad\theta}}{\norm{\bfPhi}_{\symRomT}\normhoom{\theta}}
\geq\inf_{\substack{0\neq\theta\in\hoom,\\\theta\bot_{\ltom}\RTz}}
\;\;
\frac{\normLtom{\devGrad\theta}}{\normhoom{\theta}}
\geq\frac{1}{c}
\end{align*}
by Lemma \ref{kerneldevGrad} (iv).
\end{proof}

\begin{rem}
The corresponding result for the two-dimensional case 
is completely analogous with the exception that the tensor potential 
$\bfE\in\symRomT\cap\DczomT$
is to be replaced by a much simpler vector potential $\vecw\in\hoom$. 
Furthermore, observe that
$$\scpLtom{\symRot\vecw}{\symRot\theta}
=\scpLtom{\symGrad^{\bot}\vecw}{\symGrad^{\bot}\theta}$$
holds for vector fields $\vecw,\theta\in\hoom$. 
Here the superscript $\bot$ denotes the rotation of a vector field by $90^\circ$. 
Note that the complicated second-order Neumann type 
$\Rot\symRot$-$\Div$-system for the auxiliary tensor field $\bfE$
is replaced in 2D by a much simpler Neumann linear elasticity problem,
where the standard Sobolev space $\hoom$ 
resp. $\hoom\cap\RM^{\bot_{\ltom}}$ can be used.
Here $\RM$ denotes the space of rigid motions.
This yields the decomposition result in \cite{zulehner-2016-02} for the two-dimensional case, 
which was shortly mentioned in the introduction.
\end{rem}

%%%%%%%%%%%%%%%% 
% bibliography %
%%%%%%%%%%%%%%%%

\bibliographystyle{plain} 
\bibliography{biharmonic-biblio}

\begin{thebibliography}{10}

\bibitem{Arnold:2006}
Douglas~N. {Arnold}, Richard~S. {Falk}, and Ragnar {Winther}.
\newblock {Finite element exterior calculus, homological techniques, and
  applications.}
\newblock {\em {Acta Numerica}}, 15:1--155, 2006.

\bibitem{bauerpaulyschomburgmcpweaklip}
S.~Bauer, D.~Pauly, and M.~Schomburg.
\newblock The {M}axwell compactness property in bounded weak {L}ipschitz
  domains with mixed boundary conditions.
\newblock {\em SIAM J. Math. Anal.}, 48(4):2912--2943, 2016.

\bibitem{Beirao:2007a}
L.~{Beir\~ao da Veiga}, J.~{Niiranen}, and R.~{Stenberg}.
\newblock {A posteriori error estimates for the Morley plate bending element}.
\newblock {\em {Numer. Math.}}, 106(2):165--179, 2007.

\bibitem{bogovskii1979}
M.E. Bogovskii.
\newblock Solution of the first boundary value problem for an equation of
  continuity of an incompressible medium (russian).
\newblock {\em Dokl. Akad. Nauk SSSR}, 248(5):1037--1040, 1979.

\bibitem{bogovskii1980}
M.E. Bogovskii.
\newblock Solutions of some problems of vector analysis, associated with the
  operators div and grad. (russian) theory of cubature formulas and the
  application of functional analysis to problems of mathematical physics.
\newblock {\em Akad. Nauk SSSR Sibirsk. Otdel., Inst. Mat., Novosibirsk},
  1:5--40, 1980.

\bibitem{costabelmcintoshgenbogovskii}
M.~Costabel and A.~McIntosh.
\newblock On {B}ogovski\u{i} and regularized {P}oincar\'e integral operators
  for de {R}ham complexes on {L}ipschitz domains.
\newblock {\em Math. Z.}, 265(2):297--320, 2010.

\bibitem{galdibook1}
G.P. Galdi.
\newblock {\em An introduction to the mathematical theory of the Navier-Stokes
  equations. Steady-state problems}.
\newblock Springer Monographs in Mathematics, New York, second edition, 1986.

\bibitem{hellan:67}
K.~Hellan.
\newblock {Analysis of elastic plates in flexure by a simplified finite element
  method}.
\newblock {Acta Polytech. Scand. CI 46}, 1967.

\bibitem{herrmann:67}
L.~Herrmann.
\newblock {Finite element bending analysis for plates}.
\newblock {\em {J. Eng. Mech., Div. ASCE EM5}}, 93:49 -- 83, 1967.

\bibitem{hiptmairlizouuniext}
R.~Hiptmair, J.~Li, and J.~Zou.
\newblock Universal extension for {S}obolev spaces of differential forms and
  applications.
\newblock {\em J. Funct. Anal.}, 263(2):364--382, 2012.

\bibitem{huang:11}
J.~{Huang}, X.~{Huang}, and Y.~{Xu}.
\newblock {Convergence of an adaptive mixed finite element method for Kirchhoff
  plate bending problems}.
\newblock {\em {SIAM J. Numer. Anal.}}, 49(2):574--607, 2011.

\bibitem{jochmanncompembmaxmixbc}
F.~Jochmann.
\newblock A compactness result for vector fields with divergence and curl in
  ${L}^q({\Omega})$ involving mixed boundary conditions.
\newblock {\em Appl. Anal.}, 66:189--203, 1997.

\bibitem{johnson:73}
C.~Johnson.
\newblock {On the convergence of a mixed finite-element method for plate
  bending problems}.
\newblock {\em {Numer. Math.}}, 21:43--62, 1973.

\bibitem{zulehner-2016-02}
Wolfgang Krendl, Katharina Rafetseder, and Walter Zulehner.
\newblock A decomposition result for biharmonic problems and the
  {H}ellan-{H}errmann-{J}ohnson method.
\newblock {\em {ETNA, Electron. Trans. Numer. Anal.}}, 45:257--282, 2016.

\bibitem{lions:72}
J.-L. Lions and E.~Magenes.
\newblock {\em Non-Homogeneous Boundary Value Problems and Applications. {V}ol.
  {I}}.
\newblock Springer-Verlag, New York-Heidelberg, 1972.
\newblock Translated from the French by P. Kenneth, Die Grundlehren der
  mathematischen Wissenschaften, Band 181.

\bibitem{necas:12}
J.~Ne{\v{c}}as.
\newblock {\em Direct Methods in the Theory of Elliptic Equations}.
\newblock Springer Monographs in Mathematics. Springer, Heidelberg, 2012.
\newblock Translated from the 1967 French original by Gerard Tronel and Alois
  Kufner.

\bibitem{paulymaxconst0}
D.~Pauly.
\newblock On constants in {M}axwell inequalities for bounded and convex
  domains.
\newblock {\em Zapiski POMI{\rm, 435:46-54, 2014}, \& J. Math. Sci. (N.Y.){\rm,
  210(6):787-792}}, 2015.

\bibitem{paulymaxconst1}
D.~Pauly.
\newblock On {M}axwell's and {P}oincar\'e's constants.
\newblock {\em Discrete Contin. Dyn. Syst. Ser. S}, 8(3):607--618, 2015.

\bibitem{paulymaxconst2}
D.~Pauly.
\newblock On the {M}axwell constants in 3{D}.
\newblock {\em Math. Methods Appl. Sci.}, 40(2):435--447, 2017.

\bibitem{picardharmdiff}
R.~Picard.
\newblock Zur {T}heorie der harmonischen {D}ifferentialformen.
\newblock {\em Manuscripta Math.}, 1:31--45, 1979.

\bibitem{picardboundaryelectro}
R.~Picard.
\newblock On the boundary value problems of electro- and magnetostatics.
\newblock {\em Proc. Roy. Soc. Edinburgh Sect. A}, 92:165--174, 1982.

\bibitem{picardcomimb}
R.~Picard.
\newblock An elementary proof for a compact imbedding result in generalized
  electromagnetic theory.
\newblock {\em Math. Z.}, 187:151--164, 1984.

\bibitem{picardweckwitschxmas}
R.~Picard, N.~Weck, and K.-J. Witsch.
\newblock Time-harmonic {M}axwell equations in the exterior of perfectly
  conducting, irregular obstacles.
\newblock {\em Analysis (Munich)}, 21:231--263, 2001.

\bibitem{Quenneville:2015}
Vincent Quenneville-B\'{l}air.
\newblock {\em A New Approach to Finite Element Simulation of General
  Relativity. APAM, Columbia University}.
\newblock PhD thesis, University of Minnesota, Minneapolis, USA, 2015.

\bibitem{sohrbook}
H.~Sohr.
\newblock {\em The {N}avier-{S}tokes Equations}.
\newblock Birkh\"auser, Basel, 2001.

\bibitem{webercompmax}
C.~Weber.
\newblock A local compactness theorem for {M}axwell's equations.
\newblock {\em Math. Methods Appl. Sci.}, 2:12--25, 1980.

\bibitem{weckmax}
N.~Weck.
\newblock {M}axwell's boundary value problems on {R}iemannian manifolds with
  nonsmooth boundaries.
\newblock {\em J. Math. Anal. Appl.}, 46:410--437, 1974.

\bibitem{witschremmax}
K.-J. Witsch.
\newblock A remark on a compactness result in electromagnetic theory.
\newblock {\em Math. Methods Appl. Sci.}, 16:123--129, 1993.

\end{thebibliography}

\appendix 
\section{Proofs of Some Useful Identities}

Note that for $a,b\in\rt$ and $A\in\rttt$
\begin{align}
\mylabel{skwspn}
\spn a:\spn b=2\,a\cdot b,\quad
\skw A=\foh\spn\begin{bmatrix}A_{32}-A_{23}\\A_{13}-A_{31}\\A_{21}-A_{12}\end{bmatrix}
\end{align}
hold and hence for skew-symmetric $A$
\begin{align}
\mylabel{spnadj}
\spn a:A=\spn a:\spn\spn^{-1}A=2\,a\cdot\spn^{-1}A,
\end{align}
i.e., $\spn^{*}=2\spn^{-1}$. Moreover, we have for two matrices $A,B$
$$A^{\top}:B=\tr(AB)=\tr(BA)=B^{\top}:A=A:B^{\top}.$$
The assertions of Lemma \ref{formulastwo} and Lemma \ref{formulas}
are contained in the assertions of the following lemma.

\begin{lem}
\label{appformulasproof}
For smooth functions, vector fields and tensor fields we have
\begin{itemize}
\item[\bf(i)] 
$\skw\Gradgrad u =0$,
\item[\bf(ii)]
$\divDiv\bfM=0$, if $\bfM$ is skew-symmetric,
\item[\bf(iii)]
$\Rot(u\,\bfI)=-\spn\grad u$,
\item[\bf(iv)]
$\tr\Rot\bfM=2\div(\spn^{-1}\skw\bfM)$,\\
especially $\tr\Rot\bfM=0$, if $\bfM$ is symmetric,
\item[\bf(v)]
$\Div(u\,\bfI)=\grad u$,
\item[\bf(vi)]
$\tr\Grad\vecv=\div\vecv$,
\item[\bf(vii)]
$\Div(\spn\vecv)=-\rot \vecv$,\\ 
especially $\Div(\skw\bfM)=-\rot \vecv$ for $\vecv=\spn^{-1}\skw\bfM$,
\item[\bf(viii)]
$\Rot(\spn\vecv)=(\div\vecv)\,\bfI-(\Grad\vecv)^{\top}$,\\ 
especially $\Rot\skw\bfM=(\div\vecv)\,\bfI-(\Grad\vecv)^{\top}$ for $\vecv=\spn^{-1}\skw\bfM$,
\item[\bf(ix)]
$\skw\Grad\vecv=\foh\spn\rot \vecv$ and 
$\Rot(\sym\Grad\vecv)=-\Rot(\skw\Grad\vecv)=-\foh\Rot(\spn\rot \vecv)$,
\item[\bf(x)]
$\skw\Rot\bfM=\spn\vecv$
and
$\Div(\symRot\bfM)=-\Div(\skw\Rot\bfM)=\rot \vecv$\\
with
$\vecv=\foh\big(\Div\bfM^{\top}-\grad(\tr\bfM)\big)$,\\ 
especially $\Div(\symRot\bfM)=-\Div(\skw\Rot\bfM)=\foh\rot\Div\bfM^{\top}$, if $\tr\bfM=0$,
\item[\bf(xi)]
$\grad\div\vecv=\frac{3}{2}\Div\dev\,(\Grad\vecv)^{\top}$.
\end{itemize}
These formulas hold for distributions as well.
\end{lem}

\begin{proof}
(i)-(ix) and the first identity in (x) follow by elementary calculations. 
For the second identity in (x) observe that 
$0=\Div\Rot\bfM=\Div(\symRot\bfM)+\Div(\skw\Rot\bfM)$
for $\bfM\in\Cic(\rt)$ and hence,
using the first identity in (x) and (vii), we obtain
$$\Div(\symRot\bfM)=-\Div(\skw\Rot\bfM)=-\Div(\spn\vecv)=\rot \vecv.$$
To see (xi) we compute
\begin{align*}
0
&=\Div\Rot\spn\vecv
=\Div\big((\div\vecv)\,\bfI\big)
-\Div(\Grad\vecv)^{\top}\\
&=\Div\big((\div\vecv)\,\bfI\big)
-\Div\dev\,(\Grad\vecv)^{\top}
-\frac{1}{3}\Div\big((\tr\,(\Grad\vecv)^{\top})\,\bfI\big)\\
&=\frac{2}{3}\Div\big((\div\vecv)\,\bfI\big)
-\Div\dev\,(\Grad\vecv)^{\top}
=\frac{2}{3}\grad\div\vecv
-\Div\dev\,(\Grad\vecv)^{\top}.
\end{align*}
Therefore, the stated formulas hold in the smooth case.
By density these formulas extend to $u$, $\vecv$, and $\bfM$ in respective Sobolev spaces.
Let us give proofs for distributions as well.
For this, let $m\in\nz$ and $u\in\hmmom$, $\vecv\in\hmmom$, $\bfM\in\Hmmom$
and $\varphi\in\cicom$, $\theta\in\cicom$, and $\bfPhi\in\Cicom$. By
$$\scphmmom{u}{\p_{i}\p_{j}\varphi}
=\scphmmom{u}{\p_{j}\p_{i}\varphi},
\quad\text{or (with (ii))}\quad
\scphmmom{u}{\divDiv\skw\bfPhi}=0,$$
we see that $\Gradgrad u\in\hmmmtom$ is symmetric and hence (i). 
Note that we observe formally $(\skw\Gradgrad)^{*}=\divDiv\skw$.
If $\bfM$ is skew-symmetric we have $\scpHmmom{\bfM}{\Gradgrad\varphi}=0$, i.e., (ii). 
We compute with (iv)
\begin{align*}
\scpHmmom{u\,\bfI}{\Rot\bfPhi}
&=\scphmmom{u}{\tr(\Rot\bfPhi)}
=2\scphmmom{u}{\div(\spn^{-1}\skw\bfPhi)}\\
&=-\scp{\spn\grad u}{\skw\bfPhi}_{\Hmmmoom}
=-\scp{\spn\grad u}{\bfPhi}_{\Hmmmoom},
\end{align*}
showing (iii). Formally, $(\tr\Rot)^{*}=\Rot(\,\cdot\,\bfI)$. Hence by (iii)
\begin{align*}
\scpHmmom{\bfM}{\Rot(\varphi\,\bfI)}
&=-\scpHmmom{\bfM}{\spn\grad\varphi}
=-\scpHmmom{\skw\bfM}{\spn\grad\varphi}\\
&=-2\scphmmom{\spn^{-1}\skw\bfM}{\grad\varphi}
=2\scp{\div\spn^{-1}\skw\bfM}{\varphi}_{\hmmmoom},
\end{align*}
yielding (iv). (v) follows by
\begin{align*}
-\scpHmmom{u\,\bfI}{\Grad\theta}
=-\scphmmom{u}{\tr(\Grad\theta)}
=-\scphmmom{u}{\div\theta}.
\end{align*}
Formally, $(\tr\Grad)^{*}=-\Div(\,\cdot\,\bfI)$. Thus by (v)
\begin{align*}
-\scphmmom{\vecv}{\Div(\varphi\,\bfI)}
=-\scphmmom{\vecv}{\grad\varphi}
=\scp{\div\vecv}{\varphi}_{\hmmmoom},
\end{align*}
yielding (vi). 
We have the formal adjoint $(\Div\spn)^{*}=(\Div\skw\spn)^{*}=-2\spn^{-1}\skw\Grad$, 
and by the formula $2\skw\Grad\theta=\spn\rot\theta$ from (ix), we obtain (vii), i.e.,
\begin{align*}
-2\scphmmom{\vecv}{\spn^{-1}\skw\Grad\theta}
=-\scphmmom{\vecv}{\rot\theta}.
\end{align*}
Using the formal adjoint $(\Rot\spn)^{*}=2\spn^{-1}\skw\Rot$ we calculate with (x)
\begin{align*}
2\scphmmom{\vecv}{\spn^{-1}\skw\Rot\bfPhi}
&=\scphmmom{\vecv}{\Div\bfPhi^{\top}-\grad(\tr\bfPhi)}\\
&=-\scp{\Grad\vecv}{\bfPhi^{\top}}_{\Hmmmoom}
+\scp{\div\vecv}{\tr\bfPhi}_{\hmmmoom},
\end{align*}
i.e., (viii) holds. Formally $(\skw\Grad)^{*}=-\Div\skw$. Using (vii) we see
\begin{align*}
-\scphmmom{\vecv}{\Div\skw\bfPhi}
=\scphmmom{\vecv}{\rot\spn^{-1}\skw\bfPhi}
=\foh\scp{\spn\rot \vecv}{\skw\bfPhi}_{\Hmmmoom},
\end{align*}
which proves (ix). We compute by (viii)
\begin{align*}
\scpHmmom{\bfM}{\Rot\skw\bfPhi}
&=\scphmmom{\tr\bfM}{\div(\spn^{-1}\skw\bfPhi)}
-\scpHmmom{\bfM^{\top}}{\Grad(\spn^{-1}\skw\bfPhi)}\\
&=-\scp{\grad(\tr\bfM)}{\spn^{-1}\skw\bfPhi}_{\hmmmoom}
+\scp{\Div\bfM^{\top}}{\spn^{-1}\skw\bfPhi}_{\hmmmoom}\\
&=-\foh\scp{\spn(\grad\tr\bfM)}{\skw\bfPhi}_{\Hmmmoom}
+\foh\scp{\spn\Div\bfM^{\top}}{\skw\bfPhi}_{\Hmmmoom},
\end{align*}
showing the first formula in (x) and the second one follows by $\Div\Rot=0$ and (vii).
To prove (xi) we observe
$$\scphmmom{\vecv}{\Div(\devGrad\theta)^{\top}}
=\scphmmom{\vecv}{\Div\dev(\Grad\theta)^{\top}}
=\frac{2}{3}\scphmmom{\vecv}{\grad\div\theta},$$
completing the proof.
\end{proof}

\begin{proof}[Proof of Lemma \ref{comfctlem}]
For $\bfM\in\RcomS$ there exists a sequence $(\bfPhi_{n})\subset\Cicom\cap\ltomS$
with $\bfPhi_{n}\to\bfM$ in $\Rom$. But then $(\varphi\bfPhi_{n})\subset\Cicom\cap\ltomS$
with $\varphi\bfPhi_{n}\to\varphi\bfM$ in $\Rom$,
proving $\varphi\bfM\in\RcomS$, as we have
$\Rot(\varphi\bfPhi_{n})=\varphi\Rot\bfPhi_{n}+\grad\varphi\times\bfPhi_{n}$.
This formula also shows for $\bfPsi\in\Cicom$ \big(note that $\varphi\bfPsi\in\Cicom$\big)
\begin{align}
\label{partintRot}
\begin{split}
\scpLtom{\varphi\bfM}{\Rot\bfPsi}
&=\scpLtom{\bfM}{\varphi\Rot\bfPsi}
=\scpLtom{\bfM}{\Rot(\varphi\,\bfPsi)}
-\scpLtom{\bfM}{\grad\varphi\times\bfPsi}\\
&=\scpLtom{\Rot\bfM}{\varphi\bfPsi}
+\scpLtom{\grad\varphi\times\bfM}{\bfPsi},
\end{split}
\end{align}
and thus $\Rot(\varphi\bfM)=\varphi\Rot\bfM+\grad\varphi\times\bfM$.
Analogously we prove the other cases of (i).
Similarly we show (iii) using the formula
$\Div(\varphi\bfPhi_{n})=\varphi\Div\bfPhi_{n}+\grad\varphi\cdot\bfPhi_{n}$.
To show (ii), let $\bfM\in\RomS$. Then $\varphi\bfM\in\ltomS$ and
\eqref{partintRot} shows $\varphi\bfM\in\RomS$ with the desired formula.
Analogously the other cases of (ii) follow.
Similarly we prove (iv).
Let $\bfE\in\symRomT$ and $\bfPhi\in\Cicom$. Then $\varphi\bfE\in\ltomT$ and
with $\varphi\,\bfPhi\in\Cicom$ we get
\begin{align*}
\scpLtom{\varphi\bfE}{\Rot\sym\bfPhi}
&=\scpLtom{\bfE}{\varphi\Rot\sym\bfPhi}
=\scpLtom{\bfE}{\Rot\sym(\varphi\,\bfPhi)}
-\scpLtom{\bfE}{\grad\varphi\times\sym\bfPhi}\\
&=\scpLtom{\symRot\bfE}{\varphi\bfPhi}
+\scpLtom{\grad\varphi\times\bfE}{\sym\bfPhi},
\end{align*}
which shows $\varphi\bfE\in\symRomT$ and 
$\symRot(\varphi\bfE)=\varphi\symRot\bfE+\sym(\grad\varphi\times\bfE)$
and hence (v). To prove (vi), let $\bfM\in\dDomS$ and $\phi\in\cicom$. 
Then $\varphi\bfM\in\ltomS$ and we compute by
\begin{align*}
\Gradgrad(\varphi\,\phi)
&=\varphi\Gradgrad\phi
+\phi\Gradgrad\varphi
+2\sym\big((\grad\varphi)(\grad\phi)^{\top}\big),\\
(\grad\varphi)(\grad\phi)^{\top}
&=\Grad(\phi\grad\varphi)
-\phi\Gradgrad\varphi
\intertext{the identity}
\Gradgrad(\varphi\,\phi)
&=\varphi\Gradgrad\phi
-\phi\Gradgrad\varphi
+2\sym\Grad(\phi\grad\varphi).
\end{align*}
Finally with $\varphi\phi\in\cicom$ we get
\begin{align*}
&\qquad\scpLtom{\varphi\bfM}{\Gradgrad\phi}
=\scpLtom{\bfM}{\varphi\Gradgrad\phi}\\
&=\bscp{\bfM}{\Gradgrad(\varphi\,\phi)}_{\Ltom}
+\scpLtom{\bfM}{\phi\Gradgrad\varphi}
-2\bscp{\bfM}{\sym\Grad(\phi\grad\varphi)}_{\Ltom}\\
&=\scpltom{\divDiv\bfM}{\varphi\,\phi}
+\scpLtom{\bfM\!:\!\Gradgrad\varphi}{\phi}
-2\bscp{\bfM}{\Grad(\phi\grad\varphi)}_{\Ltom}\\
&=\scpltom{\varphi\divDiv\bfM}{\phi}
+\bscp{\tr\,(\bfM\Gradgrad\varphi)}{\phi}_{\ltom}
+2\underbrace{\scphmoom{\Div\bfM}{\phi\grad\varphi}}_{\displaystyle=\scphmoom{\Div\bfM\cdot\grad\varphi}{\phi}},
\end{align*}
which shows (vi), i.e., $\varphi\bfM\in\dDzmoomS$ and 
$$\divDiv(\varphi\bfM)
=\varphi\divDiv\bfM
+2\grad\varphi\cdot\Div\bfM
+\tr\,(\bfM\Gradgrad\varphi)\in\hmoom.$$
The proof is finished.
\end{proof}

%%%%%%%%%%%%%%%%%%%%%%%%%%%%%%%%%%%%%%%%%%%%%%%%%%%%%%%%%%%%%%%%%%%%%%%%%%%%%%%%
\end{document}